\documentclass[10pt, reqno]{amsart}

\usepackage{silence}
\usepackage{fix-cm}
\usepackage{amsmath}
\usepackage{amssymb}
\usepackage{xcolor}
\usepackage{comment}
\usepackage{graphicx}
\usepackage{amsthm}
\usepackage{hyperref}
\usepackage{mathptmx}
\usepackage{wasysym}
\usepackage{tikz}
\usepackage{yfonts}
\usepackage{mathrsfs}
\usepackage {latexsym}
\usepackage{txfonts}
\usepackage{xfrac}
\usepackage[shortlabels]{enumitem}

\usetikzlibrary{decorations}
\usetikzlibrary{arrows.meta}
\usetikzlibrary{decorations.pathmorphing}
\usetikzlibrary{er, positioning}

\makeatletter
\newcommand{\circlesign}[1]{%
  \mathbin{%
    \mathchoice
      {\buildcirclesign{\displaystyle}{#1}}
      {\buildcirclesign{\textstyle}{#1}}
      {\buildcirclesign{\scriptstyle}{#1}}
      {\buildcirclesign{\scriptscriptstyle}{#1}}%
  }%
}

\newcommand{\buildcirclesign}[2]{%
  \ooalign{%
    $\m@th#1\bigcirc$\cr
    \hidewidth$\m@th#1#2$\hidewidth\cr
  }%
}
\makeatother

\newtheorem{theorem}{Theorem}[section]
\newtheorem{lemma}[theorem]{Lemma}
\newtheorem{proposition}[theorem]{Proposition}
\newtheorem{corollary}[theorem]{Corollary}     
\newtheorem{define}{Definition}[section]

\newtheorem{remark}{Remark}[section]

\newcommand{\joinR}{\hspace{-.1em}}
\newcommand{\RomanI}{I}
\newcommand{\RomanII}{\mbox{\RomanI\joinR\RomanI}}
\newcommand{\RomanIII}{\mbox{\RomanI\joinR\RomanII}}
\newcommand{\RomanIV}{\mbox{\RomanI\joinR\RomanV}}
\newcommand{\RomanV}{V}
\newcommand{\RomanVI}{\mbox{\RomanV\joinR\RomanI}}

\DeclareMathOperator*{\Id}{Id}

\DeclareMathOperator*{\supp}{supp}
\DeclareMathOperator*{\divergence}{div}

\newif\ifshow
\showtrue
\ifshow

\else
\excludecomment{details}
\fi

\allowdisplaybreaks 
\begin{document}

\title[2D singular Cahn-Hilliard equation]{Remarks on the two-dimensional Cahn-Hilliard equation forced by divergence of space-time white noise}

 
\author[Kazuo Yamazaki]{Kazuo Yamazaki}  
\address{Department of Mathematics, University of Nebraska, Lincoln, 243 Avery Hall, PO Box 880130, Lincoln, NE 68588-0130, U.S.A.; Phone: 402-473-3731; Fax: 402-472-8466}
\email{kyamazaki2@nebraska.edu}
\date{}
\thanks{This work was supported by NSF Award No. 2531744 and Simons Foundation MPS-TSM-00962572.}

\begin{abstract}
We consider the two-dimensional Cahn-Hilliard equation forced by divergence of space-time white noise that represents Kawasaki dynamics in conservative form. The standard heuristic argument shows that the solution is a distribution and thus the product within the nonlinear term is ill-defined, making this stochastic partial differential equation challenging. We prove the global-in-time unique solution theory after multiple Da Prato-Debussche tricks where the last equation is solved in a weak formulation. Besides a minimum amount of the standard renormalization procedure to deal with ill-defined products, our proof consists of applications of deterministic analysis tools and taking advantage of the unique structure of the equation, which is crucial. We also prove the local-in-time unique solution theory of a mild solution, demonstrating that our improved estimate gets us arbitrarily close, although not quite, to being able to prove the global-in-time solution theory with the last equation in not only a weak formulation but also a mild formulation. Additionally, our approach can be readily extended to prove the global-in-time existence, although not path-wise uniqueness, of solution to the singular surface quasi-geostrophic equations with the resulting equation after multiple Da Prato-Debussche tricks in weak formulation. 
\end{abstract}

\maketitle

\section{Introduction}
\subsection{Motivation from physics and real-world applications}
Stochastic partial differential equations (SPDEs) are partial differential equations (PDEs) forced by random noise and there are various contexts when their advantages over their deterministic counterparts can be justified: the study of turbulence, unpredictable external force, and fluid in microscopic scales displaying chaotic motions. Among various types of random noise, space-time white noise (STWN) (see Definition~\ref{Def:STWN}) has received an exceptional amount of attention from physicists since the pioneering work of Landau and Lifshitz \cite{LL57}. Examples of SPDEs forced by STWN include, but are not limited to, the ferromagnetic model \cite{MM75}; Kardar-Parisi-Zhang (KPZ) equation \eqref{KPZ} \cite{KPZ86}; magnetohydrodynamics (MHD) system \cite{CT92}; Navier-Stokes equations \eqref{NS} \cite{FNS77, YO86}; $\Phi^{4}$ model \eqref{Phi4} \cite{GJ87, PW81}; and the Rayleigh-B\'{e}nard equation \cite{ACHS81, GP75, HS92, SH77, ZS71}. The fundamental difficulty in the case of forcing by STWN is its roughness (see \eqref{Reg:xi}) which leads to the solution becoming distribution-valued and making the products within the nonlinear term ill-defined according to the classical Bony's estimates (see Corollary~\ref{Bony's threshold}); we refer to such SPDEs as singular SPDEs. 

The global-in-time solution theory for singular SPDEs, to be discussed in Section~\ref{Singular:global}, remains mathematically challenging and limited in the literature mostly in the following cases.  
\begin{enumerate}
\item The SPDE admits favorable transformation such as the Cole-Hopf transform of the KPZ equation (e.g. \cite{GP17, H13}).  
\item The nonlinearity of the SPDE such as the $\Phi_{d}^{4}$ model \eqref{Phi4} for $d \in \{2,3\}$ has damping effect and diffusion consists of a Laplacian, and consequently the SPDE allows $L_{x}^{p}$-estimates (e.g. \cite{MW17a, MW17b}) and maximum principle  (e.g. \cite{GH19}). 
\item The SPDE such as the 2D Navier-Stokes equations forced by STWN \eqref{Phi4} grants explicit knowledge of invariant measure (e.g. \cite{DD02}, also \cite{DD03}). The global-in-time solution theory that was derived depending on such an advantage is also restricted to initial data with respect to (w.r.t.) the invariant measure.  
\end{enumerate}
Very recently, Hairer and Rosati \cite{HR24} provided a new proof of the global-in-time unique solution theory for the 2D Navier-Stokes equations forced by STWN without relying on the explicit knowledge of its invariant measure, but instead on results from Anderson Hamiltonian (see Section~\ref{Subsec:HR24}). Another new development due to the convex integration technique consists of the global-in-time non-unique solution theory of the 2D singular surface quasi-geostrophic (SQG) equations \eqref{SQG} due to Hofmanov\'{a}, Luo, Zhu, and Zhu \cite{HLZZ23} (also \cite{HZZ22a} in case of the noise white only in space) and the 2D Navier-Stokes equations forced by STWN by L$\ddot{\mathrm{u}}$ and Zhu \cite{LZ23, LZ24}.   

Cahn-Hilliard equation \eqref{Cahn-Hilliard} describes phase separation in materials and simulates multiphase fluid flows (e.g. \cite{BDMST23}). The purpose of this manuscript is to prove a global-in-time unique solution theory for the 2D Cahn-Hilliard equation forced by divergence of STWN, specifically \eqref{Singular:CH} that represents Kawasaki dynamics and appeared in \cite[Equation (4b)]{H14b} by Hairer. 

Curiously, \eqref{Singular:CH} has damping type nonlinearity but its double Laplacian $\Delta^{2}$ as its diffusion disallows $L^{p}(\mathbb{T}^{d})$-estimate or maximum principle in sharp contrast to the previous success on the $\Phi_{d}^{4}$ model, $d \in \{2,3\}$ (see Section~\ref{Subsec:MW17b}). At the time of writing this manuscript, the new approaches of \cite{HR24} by Hairer and Rosati seems to face difficulty as well (see Section~\ref{Subsec:HR24}). 

It turns out that the Cahn-Hilliard equation has a favorable structure (see \eqref{CH:Laplace:inverse}) that can admit $\dot{H}^{-1}(\mathbb{T}^{2})$-estimate (see Proposition~\ref{Prop:3.2}). We bootstrap this bound to $\dot{H}^{-1+\alpha}(\mathbb{T}^{2})$-estimate for any $\alpha \in (0,1)$, which seems to be new even in the deterministic setup (see Proposition~\ref{Prop:3.3}). These two estimates allow us to pass the limit in the weak formulation and verify the solution's path-wise uniqueness, implying that the solution is probabilistically strong by the classical Yamada-Watanabe theorem. 

Our proof is very simple but interesting; besides a minimum amount of standard renormalization procedure (see \eqref{Renorm} and Section~\ref{Sec:5}), we simply apply deterministic analysis tools from harmonic analysis to achieve our goal, particularly the positivity lemma of fractional Laplacian (see Lemma~\ref{Lem:Max:Prin}). In the next Section~\ref{Intro:main:eq} we will introduce the Cahn-Hilliard equation (see \eqref{Cahn-Hilliard}) and clarify the meaning of STWN (see Definition~\ref{Def:STWN}) and singular SPDEs. 

\subsection{Introduction of the main equation}\label{Intro:main:eq}
We set a minimum amount of notations and preliminaries for completeness. We define $\mathbb{N} \triangleq \{ 1, 2, \hdots \}$ and $\mathbb{N}_{0} \triangleq \mathbb{N} \cup \{0\}$, and mostly work with a spatial variable $x \in \mathbb{T}^{d} = (\mathbb{R} \setminus \mathbb{Z})^{d}$ for $d \in \mathbb{N}$. We abbreviate by denoting $\partial_{t} \triangleq  \frac{\partial}{\partial t}$ and ``dD'' = ``$d$-dimensional.'' We write $A \lesssim_{\alpha, \beta} B$ whenever there exists a constant $C = C(\alpha,\beta) \geq 0$ such that $A \leq CB$ and $A \approx_{\alpha,\beta} B$ in case $A \lesssim_{\alpha,\beta}B$ and $A \gtrsim_{\alpha,\beta}B$. We write $A \overset{( \cdot)}{\lesssim} B$ whenever $A\lesssim B$ due to $(\cdot)$.  We follow the typical convention that a universal constants in a series of inequalities continue to be denoted by ``$C$.'' We denote a Fourier transform of $f$ by $\hat{f} = \mathcal{F}(f)$ and define $\Lambda^{\gamma} \triangleq (-\Delta)^{\frac{\gamma}{2}}$ as a fractional Laplacian of order $\frac{\gamma}{2} \in \mathbb{R}$, specifically a Fourier operator with a Fourier symbol $\lvert k \rvert^{\gamma}$ so that $\widehat{ \Lambda^{\gamma} f} (k) = \lvert k \rvert^{\gamma} \hat{f} (k)$. We denote the Lebesgue, homogeneous and inhomogeneous Sobolev spaces by $L^{p}, \dot{H}^{s}$, and $H^{s}$ for $p\in [1,\infty], s\in \mathbb{R}$ with corresponding norms of $\lVert \cdot\rVert_{L^{p}}, \lVert \cdot \rVert_{\dot{H}^{s}}$, and $\lVert \cdot \rVert_{H^{s}}$, respectively. We also write $L_{x}^{p}$ to specify that it is $L^{p}$-space w.r.t. spatial variable $x$ without specifying the spatial domain. At last, $\mathcal{S}(\mathbb{R}^{d})$ denotes the Schwartz space on $\mathbb{R}^{d}$. 

The equation of our main interest is the Cahn-Hilliard equation that was originally proposed by Cahn and Hilliard \cite{CH81}. To introduce the model, let us denote the order parameter that continuously varies through the diffuse interface separating the pure states from -1 to 1 by $\Psi$, the mobility constant by $\kappa > 0$, the chemical potential by $\mu \triangleq - \alpha \Delta \Psi + f(\Psi)$ where $\alpha > 0$ is the surface tension coefficient and $f = F'$, where typically $F(s) = \frac{1}{4} (s^{2} -1)^{2}$ so that 
\begin{equation*}
f(s) = s^{3} - s.
\end{equation*} 
Under such notations, the Cahn-Hilliard equation reads 
\begin{equation}\label{Cahn-Hilliard}
\partial_{t} \Psi = \kappa \Delta \mu = \kappa \Delta \left(- \alpha \Delta \Psi + \Psi^{3} - \Psi  \right). 
\end{equation} 
We refer to \cite{ES86, M19} and references therein for further details. 

To write down the singular Cahn-Hilliard equation, we first define STWN formally. 
\begin{define}\label{Def:STWN}
We fix a probability space $(\Omega, \mathcal{F}, \mathbb{P})$ so that the STWN $\xi = (\xi_{1} \, \xi_{2} \, ..., \xi_{d})$ in a vector form can be introduced as a distribution-valued Gaussian field with a  covariance of 
\begin{equation*}
\mathbb{E} [ \xi_{i}(t,x) \xi_{j}(s,y) ] = 1_{\{ i=j\}} \delta(t-s) \prod_{l=1}^{d} \delta(x_{l} - y_{l}); 
\end{equation*} 
i.e.,
\begin{align*}
\mathbb{E} [ \xi_{i} (\phi) \xi_{j} (\psi) ] = 1_{\{ i = j \}} \int_{\mathbb{R} \times \mathbb{T}^{d}} \phi(t,x) \psi(t,x) dx dt \hspace{3mm} \forall \, \phi, \psi \in \mathcal{S} ( \mathbb{R} \times \mathbb{T}^{d}). 
\end{align*}    
\end{define}

Before introducing the singular Cahn-Hilliard equation, we first introduce the $\Phi_{d}^{4}$ model. It is a special case of Glauber dynamics \cite{G63} associated to Euclidean $\Phi^{4}$ field theory in the context of stochastic quantization \cite{PW81} that are non-conservative and takes the form $\partial_{t} \phi  = - \frac{\delta V}{\delta \phi} + \xi$. Specifically, the solution $\Phi: \mathbb{R}_{+} \times \mathbb{T}^{d} \to \mathbb{R}$ satisfies 
\begin{equation}\label{Phi4}
\partial_{t} \Phi  = \Delta \Phi - \lambda \Phi^{3} + \xi 
\end{equation} 
where $\lambda > 0$ is a coupling constant (e.g. \cite[Equation (4.1)]{PW81} and \cite[Section 23.1]{GJ87}). 

In contrast to Glauber dynamics, Kawasaki dynamics $\partial_{t} A = \nabla \cdot ( \nabla \frac{\delta S}{\delta A}) + \nabla \cdot \xi$ from \cite{K66} is conservative and the following Cahn-Hilliard equation forced by divergence of STWN  appeared in \cite[Equation (4b)]{H14b} by Hairer: 
\begin{equation}\label{CH:Hairer}
\partial_{t} \Psi = -\Delta (\Delta \Psi + C \Psi - \Psi^{3}) + \divergence \xi.  
\end{equation} 

Let us now specify the meaning of singular SPDEs, the category of SPDEs into which both \eqref{Phi4} and \eqref{CH:Hairer} fall in case $d \geq 2$. We denote the inhomogeneous Besov spaces by $B_{p,q}^{s}$ where $p, q \in [1,\infty]$ and $s \in \mathbb{R}$ with norms defined by $\lVert f \rVert_{B_{p,q}^{s}} \triangleq  \lVert 2^{sm} \lVert \Delta_{m} f \rVert_{L_{x}^{p}} \rVert_{l^{q}(m\geq -1)}$ (see  Section~\ref{Prelim:Besov} for details). Importantly, we recall that if $\alpha \in \mathbb{R}_{+} \setminus \mathbb{N}$, then the Besov-H$\ddot{\mathrm{o}}$lder spaces $B_{\infty,\infty}^{\alpha}$ coincide with the classical H$\ddot{\mathrm{o}}$lder space; i.e., $B_{\infty,\infty}^{\alpha} = C^{\lfloor \alpha \rfloor, \alpha - \lfloor \alpha \rfloor}$ (see \cite[Sections 2.3 and 2.7]{BCD11} and \cite[Proposition 3.2.8]{Y26}). Let us denote $\mathscr{C}^{\alpha} \triangleq B_{\infty,\infty}^{\alpha}$ to distinguish from the classical H$\ddot{\mathrm{o}}$lder space $C^{\alpha}$. 

To explain the notion of local subcriticality \cite[Assumption 8.3]{H14a}, we follow the setup of \cite{H14a} and fix the scaling and scaling dimension  corresponding to \eqref{CH:Hairer} as 
\begin{equation*}
\mathfrak{s} \triangleq (4, \underbrace{1, \hdots, 1}_{\text{$d$-many times}}), \qquad \lvert \mathfrak{s} \rvert = 4 + d. 
\end{equation*}
For the precise definition of $\mathscr{C}_{\mathfrak{s}}^{\alpha}(\mathbb{R} \times \mathbb{T}^{d})$ for $\alpha < 0$, we refer to \cite[Definition 3.7]{H14a} and only mention that $\mathscr{C}_{\mathfrak{s}}^{\alpha}(\mathbb{R} \times \mathbb{T}^{d})$ is essentially $B_{\infty,\infty}^{\alpha}(\mathbb{R} \times \mathbb{T}^{d})$ except that $f \in \mathscr{C}_{\mathfrak{s}}^{\alpha}(\mathbb{R} \times \mathbb{T}^{d})$ if and only if $\lVert \Delta_{j}^{\mathfrak{s}} f \rVert_{L_{t,x}^{\infty}} \lesssim 2^{-j\alpha}$ for all $j \geq -1$ where $\Delta_{j}^{\mathfrak{s}}$ is a parabolic dyadic projection onto $\{ (\tau, \xi): \lvert \tau \rvert^{\frac{1}{4}} + \lvert \xi \rvert \approx 2^{j} \}$ (see \cite[Remark 3.8 and Definition 2.14]{H14a}). Under such a setting, it is known (e.g. \cite[Lemma 10.2 and Proposition 3.20]{H14a} and \cite[Lemma 4.1]{BK17}) that $\mathbb{P}$-almost surely ($\mathbb{P}$-a.s.)
\begin{equation}\label{Reg:xi}
\xi \in \mathscr{C}_{\mathfrak{s}}^{\alpha} (\mathbb{R} \times \mathbb{T}^{d}) \qquad \text{ for } \qquad \alpha < - \frac{ \lvert \mathfrak{s} \rvert}{2}. 
\end{equation}
Considering that $\lvert \mathfrak{s} \rvert = 4+d$ and the divergence operator applied on $\xi$ within \eqref{CH:Hairer}, we see that $\Psi$ of \eqref{CH:Hairer} satisfies $\mathbb{P}$-a.s.  
\begin{equation}\label{Reg:Psi}
\Psi \in C_{t} \mathscr{C}^{\alpha}(\mathbb{T}^{d}) \qquad \text{for}  \qquad \alpha < 1 - \frac{d}{2}.     
\end{equation}
According to Corollary~\ref{Bony's threshold}, this implies that $\Psi^{3}$ within \eqref{CH:Hairer} is ill-defined making this SPDE singular already in case $d = 2$. We note that the argument of defining $\mathfrak{s}$, determining the regularity of the corresponding STWN, and then that of the solution is general and can be adapted to any other singular SPDEs to be discussed next. For example, the solution $\Phi$ to \eqref{Phi4} satisfies $\mathbb{P}$-a.s.
\begin{equation}\label{Reg:Phi}
\Phi \in C_{t}\mathscr{C}^{\alpha}(\mathbb{T}^{d}) \qquad \text{for} \qquad \alpha < 1 - \frac{d}{2}.       
\end{equation}
In the next Section~\ref{Sec:Review} we will briefly elaborate on the history and recent developments of analysis on such singular SPDEs. 

\subsection{Review of previous relevant works}\label{Sec:Review}
\subsubsection{Non-singular}
First, the 1D Burgers' equation forced by STWN, 
\begin{equation}\label{Burgers'}
\partial_{t} u + \frac{1}{2} \partial_{x} u^{2} = \partial_{x}^{2} u + \xi,     
\end{equation}
admits a solution with regularity $C_{t}\mathscr{C}^{\alpha}(\mathbb{T})$ for $\alpha < \frac{1}{2}$ $\mathbb{P}$-a.s. and thus the product $u^{2}$ within its nonlinear term is well-defined. For its global-in-time unique solution theory, we refer to \cite[Theorem 2.2]{BCJ94} by Bertini, Cancrini, and Jona-Lasinio who used stochastic Cole-Hopf transform and \cite[Theorem 3.1]{DDT94} by Da Prato, Debussche, and Temam who used, what is now known as, the Da Prato-Debussche trick. This trick boils down to defining a stochastic linear equation, say solved by $X$, such that subtracting it from the original singular SPDE solved by $u$ produces a nonlinear random PDE, say solved by $v$, so that the solution theory of $u = v + X$ reduces to the existence and uniqueness of $v$; we will elaborate on this technique in Section~\ref{Subsec:HR24}. 

Second, the stochastic Cahn-Hilliard equation was studied by Da Prato and Debussche in \cite{DD96} too, but we emphasize that they did not cover the singular case. Specifically, they forced the Cahn-Hilliard equation \eqref{Cahn-Hilliard} by $\xi$ instead of $\divergence \xi$ and covered the case $d \in \{1,2,3\}$. An  argument analogous to Section~\ref{Intro:main:eq} shows that its corresponding solution would have a regularity of $2- \frac{d}{2}$ and thus the singular case appears only for $d \geq 4$. Because the solution is a function instead of distribution in their case of $d \in \{1,2,3\}$, via standard approach through Galerkin approximation, the authors obtained the global existence and uniqueness of solutions; we refer to \cite{DD96} for further details concerning ergodicity results. 

\subsubsection{Singular but local-in-time}\label{Sec:Sing:local}
The Navier-Stokes equations forced by STWN in $\mathbb{T}^{d}, d \geq 2$, has been studied intensively many:
\begin{equation}\label{NS}
\partial_{t} u + \mathbb{P}_{L} \divergence (u\otimes u) = \Delta u + \mathbb{P}_{L} \mathbb{P}_{\neq 0} \xi, 
\end{equation}
where $u: \mathbb{R}_{+} \times \mathbb{T}^{d} \to \mathbb{R}^{d}$ represents velocity vector field, $\mathbb{P}_{L}$ is a Leray projection onto the space of divergence-free vector fields and we defined $\mathbb{P}_{\neq 0} f \triangleq f - \fint_{\mathbb{T}^{d}} f(x) dx$. Analogous argument in Section~\ref{Intro:main:eq} shows that $\mathbb{P}$-a.s. 
\begin{equation}\label{Reg:u}
u \in C_{t} \mathscr{C}^{\alpha}(\mathbb{T}^{d}) \qquad \text{ for } \qquad \alpha < 1 - \frac{d}{2}.  
\end{equation}
Using Da Prato-Debussche trick and Wick products (e.g. \cite{J97}), Da Prato and Debussche \cite{DD02} were able to prove the local-in-time unique solution theory for \eqref{NS} in the 2D case; in fact, they were able to repeat analogous result for the $\Phi_{2}^{4}$ model in \cite{DD03} (cf. \eqref{Reg:Phi} and \eqref{Reg:u}). We will shortly elaborate on their global-in-time extensions.  

The products within the nonlinear terms of \cite{DD02, DD03} are respectively $u\otimes u$ and $\Phi^{3}$ with solutions' regularity in $C_{t}\mathscr{C}^{-\kappa}(\mathbb{T}^{2})$ for an arbitrary $\kappa > 0$ so that the hypothesis of Corollary~\ref{Bony's threshold} is only barely violated. The breakthrough in case the product in the nonlinear term is ill-defined in a more severe manner came from the study of 1D KPZ equation
\begin{equation}\label{KPZ}
\partial_{t} h = \partial_{x}^{2} h + \lambda \lvert \partial_{x} h \rvert^{2} + \xi    
\end{equation}
solved by $h: \mathbb{R}_{+} \times \mathbb{T} \to \mathbb{R}$ that represents interface height and $\lambda > 0$ is the coupling strength. Analogous argument in Section~\ref{Intro:main:eq} shows that $\nabla h \in C_{t} \mathscr{C}^{\alpha}(\mathbb{T})$ for $\alpha < - \frac{1}{2}$ $\mathbb{P}$-a.s. making the issue of $\lvert \nabla h \rvert^{2}$ more serious. The KPZ equation was studied by many, e.g. Bertini and Giacomin \cite{BG97} using Cole-Hopf transform, and it was Hairer \cite{H13} who first proved the existence and uniqueness of its local-in-time solution using rough path theory due to Lyons \cite{L94}. 

Then the groundbreaking theories of regularity structures (RS) by Hairer \cite{H14a} and paracontrolled distributions by Gubinelli, Imkeller, and Perkowski \cite{GIP15} (see also \cite{GP17}) appeared and they now allow us to systematically prove the local-in-time existence of the unique solution to singular SPDEs as long as they are locally subcritical. For the precise definition of local subcriticality, we refer readers to \cite[Assumption 8.3]{H14a} and only mention that loosely speaking, local subcriticality requires  the homogeneity of the nonlinear term to be strictly larger than that of the force, where the homogeneity of a product is considered as a sum of homogeneities. 

Examples of locally subcritical singular SPDEs include the Cahn-Hilliard equation \eqref{CH:Hairer},  $\Phi_{d}^{4}$ model \eqref{Phi4}, and the Navier-Stokes equations \eqref{NS}, all in case  $d < 4$. First, the discussion of the local-in-time unique solution theory of \eqref{CH:Hairer} in the 3D and hence fully locally subcritical case via the theory of RS can be found in \cite{H14b} by Hairer (see also \cite{G24} in case $d \in \{1,2,3,4\}$ but only for the Cahn-Hilliard equation forced by STWN, not the divergence of STWN). Second, the local-in-time unique solution theory of the $\Phi_{3}^{4}$ model can be found in \cite{H14a} by Hairer via the theory of RS and \cite{CC18} by Catellier and Chouk using the theory of paracontrolled distributions; we also refer to \cite{HM18a} on its strong Feller property. Finally, the local-in-time unique solution theory of the 3D Navier-Stokes equations \eqref{NS} can be found in \cite{ZZ15} by Zhu and Zhu. 

\subsubsection{Singular and global-in-time}\label{Singular:global}
Despite the universality of local-in-time solution theory for locally subcritical SPDEs, there is only a handful of singular SPDEs for which the theory has been proven to be extendable globally in time and many of such proofs seem to rely on special properties of the singular SPDEs. 

The first example is the Cole-Hopf transform of the KPZ equation \eqref{KPZ}. Gubinelli and Perkowski \cite{GP17} observed that the local-in-time solution to the 1D KPZ equation constructed by Hairer \cite{H13} (recall from  Section~\ref{Sec:Sing:local}) is global-in-time (see \cite[p. 170 and Corollary 7.5]{GP17}). 

The second example is the explicit knowledge of the invariant measure. In case $x \in \mathbb{T}^{2}$, Da Prato and Debussche \cite{DD03} extended the local-in-time unique solution to \eqref{Phi4} that was discussed in Section~\ref{Sec:Sing:local} to globally-in-time but starting from initial data w.r.t. invariant measure. In fact, the authors were able to do so even with $\Phi^{n}$ for any $n \in \mathbb{N}$ odd instead of $\Phi^{3}$; informally, this is possible because the regularity of $\Phi^{n}$ and $\Phi^{3}$ are both $C_{t}\mathscr{C}^{-\kappa}(\mathbb{T}^{2})$ for any $\kappa > 0$ after appropriate renormalization, considering \eqref{Reg:Phi}. Because the solution to the 2D Navier-Stokes equations forced by STWN \eqref{NS} satisfies $\int_{\mathbb{T}^{2}} (u\cdot\nabla) u \cdot \Delta u dx = 0$ and thus the explicit knowledge of its invariant measure, the same authors were able to attain analogous result for \eqref{NS} in \cite{DD02}. 

The third example is the damping effect of the nonlinear term within the $\Phi_{d}^{4}$ model \eqref{Phi4}, but also the fact that the diffusion is a Laplacian, not a double Laplacian (see Section~\ref{Subsec:MW17b}). Mourrat and Weber in \cite{MW17a, MW17b} proved its global-in-time unique solution theory in both cases $x \in \mathbb{R}^{2}$ and $x \in \mathbb{T}^{3}$ (also \cite{GH19} by Gubinelli and Hofmanov\'{a} and \cite{HM18} by Hairer and Matetski). We elaborate on their approach briefly in Section~\ref{Subsec:MW17b}. 

Very recently, Hairer and Rosati \cite{HR24} provided a second proof of the global-in-time unique solution theory for the 2D Navier-Stokes equations forced by STWN \eqref{NS} without relying on its invariant measure and therefore starting from any initial data, not necessarily w.r.t. its invariant measure. We elaborate on its approach in Section~\ref{Subsec:HR24}. 

\subsection{Direct motivation of our work}
Let us describe the motivation of our work to study the global-in-time unique solution theory of the Cahn-Hilliard equation \eqref{CH:Hairer}. Hereafter, for simplicity we refer to 
\begin{equation}\label{Singular:CH}
\partial_{t} \Psi + \Delta^{2} \Psi = \Delta (\Psi^{3}) + \divergence \xi 
\end{equation}
as the Cahn-Hilliard equation instead of \eqref{CH:Hairer} because the renormalization procedure will create a linear term anyhow, and it will be clear from our proof that the addition of a linear term creates no significant difficulties. 
\begin{remark}  
We could potentially consider $\Psi^{n}$ for any odd natural number $n$ as \cite{DD03} because our Theorem~\ref{Thm:2.1} is in the 2D case. However, to simplify our computations, we chose $n=3$ in  \eqref{Singular:CH}.      
\end{remark}

Recently, the technique of convex integration has seen rapid developments in both deterministic case and adaptations to SPDEs; we refer respectively to \cite{BV19a, BV19b, DS09, DS13, I18} and \cite{BFH20, CFF19, HZZ19a} and references therein. Of relevance to our discussion, we highlight that Hofmanov\'{a}, Zhu, and Zhu \cite{HZZ23} provided global-in-time non-unique solution theory for the 3D Navier-Stokes equations forced by STWN \eqref{NS} (also \cite{LZ23, LZ24} in the 2D case). Additionally, Hofmanov\'{a}, Luo, Zhu, and Zhu \cite{HLZZ23} provided global-in-time non-unique solution theory for the 2D SQG equations 
\begin{equation}\label{SQG}
\partial_{t} \theta + u \cdot \nabla  \theta + \Lambda^{\gamma}\theta = \Lambda^{\alpha} \xi, \qquad u = \mathcal{R}^{\bot} \theta, \qquad \alpha \in [0, 1),  \, \gamma \in \left[0, \frac{3}{2} \right)
\end{equation} 
where $\mathcal{R}$ represents the Riesz transform vector and $\xi$ is STWN (see \cite{HZZ22a} in case the noise is white only in space). 

Interestingly, Dong and the author \cite{DY26} considered deterministic heat equation with power damping of odd power
\begin{equation}\label{Heat:damped}
\partial_{t} \Phi + \Phi^{n} = \Delta \Phi, 
\end{equation}
solved by $\Phi: \mathbb{R}_{+} \times \mathbb{T}^{d} \to \mathbb{R}$ and proved that any weak solution $\Phi \in L_{t,x}^{n}$ is unique. The integrability of $L_{t,x}^{n}$ is the minimum requirement for the classical weak formulation to hold.  Consequently, without reaching below $L_{t,x}^{n}$ which would then require some renormalization procedure to handle the ill-defined product, application of convex integration technique on \eqref{Heat:damped} and therefore \eqref{Phi4} seems difficult, considering that the presence of noise has never made application of convex integration easier in all the examples in the literature thus far. 

The proof of the uniqueness of $\Phi \in L_{t,x}^{n}$ that solves \eqref{Heat:damped} weakly in \cite{DY26} mainly consists of only the $L^{1}(\mathbb{T}^{d})$-estimate of the difference of two potential weak solutions, being mindful of the low integrability $L_{t,x}^{n}$, and the crucial ingredient for its success is the fact that the nonlinear term is damping. 

We notice that the nonlinear term of the Cahn-Hilliard equation \eqref{Singular:CH} also has damping effect, e.g., any $L^{p}(\mathbb{T}^{d})$-estimate, $p \geq 2$, leads to 
\begin{equation}\label{CH:damping}
-\int_{\mathbb{T}^{d}} \Delta (\Psi^{3}) \lvert \Psi \rvert^{p-2} \Psi dx = 3(p-1) \int_{\mathbb{T}^{3}} \lvert \Psi \rvert^{p} \lvert \nabla \Psi \rvert^{2} dx.
\end{equation}
This naively suggests that the proof of \cite{DY26} may be extended to the deterministic Cahn-Hilliard equation
\begin{equation}\label{CH:deterministic}
\partial_{t} \Psi + \Delta^{2} \Psi = \Delta (\Psi^{3}). 
\end{equation}
This seems non-trivial at the time of writing this manuscript because while minimum integrability necessary for the weak formulation of \eqref{CH:deterministic} to hold is $L_{t,x}^{3}$ identically to \eqref{CH:damping} with $n=3$, solutions with such low integrability cannot handle the Laplacian within the nonlinear term during $L^{1}(\mathbb{T}^{d})$-estimate. 

There is a well-known trick for the Cahn-Hilliard equation that comes to mind here. We can apply $(-\Delta)^{-1}$ to \eqref{CH:deterministic} to deduce the inverse Laplacian formulation of 
\begin{equation}\label{CH:Laplace:inverse}
\partial_{t} (-\Delta)^{-1} \Psi + \Psi^{3} = \Delta \Psi. 
\end{equation}
The nonlinear term $\Psi^{3}$ and diffusion $\Delta \Psi$ are both identical to \eqref{Heat:damped} raising hope for the needed $L^{1}(\mathbb{T}^{d})$-estimate; alas, $\partial_{t}(-\Delta)^{-1} \Psi$ destroys such hope. 

In conclusion, while a naive attempt of convex integration technique on $\Phi_{d}^{4}$ model has been shown to be difficult in \cite{DY26}, it seems non-trivial to repeat the same argument for the Cahn-Hilliard equation \eqref{CH:deterministic}. This leads us instead to aim for the global-in-time unique solution theory, the difficulties of which we discuss next.  

\subsection{Why the approach of \texorpdfstring{\cite{MW17b}} \, is difficult for \eqref{Singular:CH}}\label{Subsec:MW17b}
We briefly sketch the approach of \cite{MW17b} on the $\Phi_{3}^{4}$ model \eqref{Phi4}. The authors apply Da Prato-Debussche trick multiples times and split the solution to $\Phi = \scalebox{0.18}{\begin{tikzpicture}
\draw[black, thick] (0.5,1) -- (0.5,-0.5);
\filldraw[black] (0.5,1) circle (6pt); 
\end{tikzpicture}
} - \scalebox{0.18}{\begin{tikzpicture}
\draw[black, thick] (-0.7,0.9) -- (0,0);
\draw[black, thick] (0.7,0.9) -- (0,0);
\draw[black, thick] (0,0.9) -- (0,0);
\draw[black, thick] (0,0) -- (0,-1);
\filldraw[black] (-0.7,0.9) circle (6pt); 
\filldraw[black] (0.7,0.9) circle (6pt);
\filldraw[black] (0,0.9) circle (6pt);
\end{tikzpicture}
} + u$ where $\scalebox{0.18}{\begin{tikzpicture}
\draw[black, thick] (0.5,1) -- (0.5,-0.5);
\filldraw[black] (0.5,1) circle (6pt); 
\end{tikzpicture}
}$ solves heat equation forced by STWN $\xi$ and $\scalebox{0.18}{\begin{tikzpicture}
\draw[black, thick] (-0.7,0.9) -- (0,0);
\draw[black, thick] (0.7,0.9) -- (0,0);
\draw[black, thick] (0,0.9) -- (0,0);
\draw[black, thick] (0,0) -- (0,-1);
\filldraw[black] (-0.7,0.9) circle (6pt); 
\filldraw[black] (0.7,0.9) circle (6pt);
\filldraw[black] (0,0.9) circle (6pt);
\end{tikzpicture}
}$ is a solution to another heat equation with rough force. They further split $u = v+ w$ and carefully analyze terms in both equations of $v$ and $w$ by applying the theory of paracontrolled distributions that involve commutators (see \cite[Equations (1.19) and (1.22)]{MW17b}). The rest of the proof includes various estimates and crucially the $L^{p}(\mathbb{T}^{3})$-estimate of $w$: e.g., for an even integer $p\geq 24$, they obtain the following identity: 
\begin{align}
&\frac{1}{3p-2} \left( \lVert w(t) \rVert_{L^{3p-2}}^{3p-2}  - \lVert w(0) \rVert_{L^{3p-2}}^{3p-2} \right) + (3p-3) \int_{0}^{t} \lVert \, \lvert \nabla w \rvert^{2} w^{3p-4} (s) \rVert_{L^{1}} ds  \nonumber \\
&+ \int_{0}^{t} \lVert w(s) \rVert_{L^{3p}}^{3p} ds = \int_{0}^{t} \left\langle \tilde{G}(v,w) + cv, w^{3p-3} \right\rangle (s) ds\label{MW17b:damp}
\end{align}
in \cite[Proposition 6.2]{MW17b} where we refer interested readers to \cite[Equations (1.19) and (6.2)]{MW17b} for the precise definition of $\tilde{G}$. We can explicitly see the advantage of the good sign of the damping term in $\int_{0}^{t} \lVert w(s) \rVert_{L^{3p}}^{3p} ds$ in \eqref{MW17b:damp}. This seems to make adaptation of the approach of \cite{MW17b} to \eqref{Singular:CH} hopeless because it is well-known that $\Delta^{2}$ does not allow $L_{x}^{p}$-estimate (see Lemma~\ref{Lem:Max:Prin} that is valid only up to $\Lambda^{2} = -\Delta$). Similarly, \cite[Proposition 7.7]{MW17a} has $L^{p}_{x}$-estimate and  \cite[Lemma 2.12]{GH19} consists of a maximum principle that also seems to fail for \eqref{Singular:CH} due to $\Delta^{2}$. In conclusion, the proof of the global-in-time unique solution theory for $\Phi_{d}^{4}, d \in \{2,3\}$ in \cite{GH19, MW17a, MW17b} rely not only the fact that the nonlinear term has damping effect but that the diffusion consists of only a Laplacian, a property that is absent in \eqref{Singular:CH}.  

\subsection{Why the approach of \texorpdfstring{\cite{HR24}} \, is difficult for \eqref{Singular:CH}}\label{Subsec:HR24}
With the approach of \cite{GH19, MW17a, MW17b} out of options, we turn to the approach of \cite{HR24} which does not rely on $L_{x}^{p}$-estimate. This approach has some flexibility and led to some extensions such as the 1D Burgers' equation forced by $\Lambda^{\frac{1}{2}} \xi$ in \cite{Y25d} (recall \eqref{Burgers'}), 2D MHD system forced by STWN in \cite{Y25e}, and 3D Navier-Stokes equations forced by STWN but diffusion $-\Delta u$ replaced by $\Lambda^{\frac{5}{2}} u$ in \cite{Y25f} by the author. 

Considering \eqref{NS} in the 2D case, Hairer and Rosati apply Da Prato-Debussche tricks multiple times. Specifically, with $A^{\otimes 2} \triangleq A \otimes A$ and $A \otimes_{s}B \triangleq \frac{1}{2} (A \otimes B + B \otimes A)$, they consider 
\begin{subequations}
\begin{align}
& \partial_{t}X = \Delta X + \mathbb{P}_{L} \mathbb{P}_{\neq 0} \xi, \qquad X(0) = 0,\label{HR:Eq:X} \\
& \partial_{t} Y + \mathbb{P}_{L} \divergence \left( 2 X \otimes_{s} Y + X^{\otimes 2} \right) = \Delta Y,  \qquad Y(0) = 0,    \label{HR:Eq:Y} 
\end{align}
\end{subequations}
so that $v \triangleq u- X$ and $w \triangleq v - Y$ satisfy 
\begin{subequations}
\begin{align}
& \partial_{t} v + \mathbb{P}_{L} \divergence ( v+X)^{\otimes 2} = \Delta v, \label{HR:Eq:v} \\
& \partial_{t} w + \mathbb{P}_{L} \divergence \left( w^{\otimes 2} + 2 (X+Y) \otimes_{s} w + Y^{\otimes 2} \right) = \Delta w,  \label{HR:Eq:w}
\end{align}
\end{subequations}
and thus defining the solution $u$ boils down to that of $w$. Denoting the low and high Fourier projections with threshold $\lambda$ respectively by $\mathcal{L}_{\lambda}$ and $\mathcal{H}_{\lambda}$ (see Definition~\ref{Def:proj}), the authors further define the paracontrolled ansatz through 
\begin{subequations}\label{HR:para:ans}
\begin{align}
& ( \partial_{t} - \Delta) Q = 2 X, \qquad Q(0) = 0, \label{HR:para:ans:a}\\
&  w = - \mathbb{P}_{L} \divergence ( w  \circlesign{\prec}_{s} Q) + w^{\sharp}, \\ 
& Q^{\mathcal{H}} \triangleq \mathcal{H}_{\lambda} Q, \qquad w^{\mathcal{H}} \triangleq - \mathbb{P}_{L} \divergence (w \circlesign{\prec}_{s} Q^{\mathcal{H}}), \qquad w^{\mathcal{L}} \triangleq w - w^{\mathcal{H}}, \label{HR:para:ans:b}
\end{align}
\end{subequations}
where $f  \circlesign{\prec}_{s} g \triangleq \sum_{i \geq -1} S_{i-1} f  \otimes_{s} \Delta_{i} g$ and $S_{i-1}$ is low-frequency cut-off operator (see Section~\ref{Prelim:Besov} for details). The significance of this is that upon deriving the equation of $\partial_{t} w^{\mathcal{L}} = \partial_{t} (w - w^{\mathcal{H}})$, $\mathbb{P}_{L} \divergence (2 w\circlesign{\prec}_{s} \mathcal{H}_{\lambda} X)$ will be produced from  
\begin{equation*}
w \circlesign{\prec}_{s} \partial_{t} Q^{\mathcal{H}} \overset{\eqref{HR:para:ans:b}}{=} w \circlesign{\prec}_{s} \mathcal{H}_{\lambda} \left( (\partial_{t} Q - 2X)+ 2  X \right) \overset{\eqref{HR:para:ans:a}}{=} w \circlesign{\prec}_{s} \Delta Q^{\mathcal{H}} + 2 w \circlesign{\prec}_{s} \mathcal{H}_{\lambda} X. 
\end{equation*}
Then, when we substitute $\partial_{t} w$ by that of \eqref{HR:Eq:w}, we obtain 
\begin{align}
\partial_{t} w^{\mathcal{L}} =& \Delta w^{\mathcal{L}} - \mathbb{P}_{L} \divergence \Bigg( w^{\otimes 2} + 2(X+Y) \otimes_{s} w + Y^{\otimes 2}  \nonumber \\
& \qquad  \qquad - \partial_{t} w \circlesign{\prec}_{s} Q^{\mathcal{H}} - w \circlesign{\prec}_{s} \Delta Q^{\mathcal{H}} - 2 w \circlesign{\prec}_{s} \mathcal{H}_{\lambda} X + \Delta (w \circlesign{\prec}_{s} Q^{\mathcal{H}}) \Bigg).  \label{Explain}
\end{align}
The term $X \otimes_{s} w = (\mathcal{L}_{\lambda} X + \mathcal{H}_{\lambda} X) \otimes_{s} (w^{\mathcal{L}} + w^{\mathcal{H}})$ is singular, but upon the $L^{2}(\mathbb{T}^{2})$-estimate of $w^{\mathcal{L}}$, $- w\circlesign{\prec}_{s} \mathcal{H}_{\lambda} X = -(w^{\mathcal{L}} + w^{\mathcal{H}}) \circlesign{\prec}_{s} \mathcal{H}_{\lambda} X$ in \eqref{Explain} will have a counter effect, specifically
\begin{subequations}\label{Cancel:out}
\begin{align}
& \mathcal{H}_{\lambda} X \otimes_{s} w^{\mathcal{L}} - w^{\mathcal{L}} \circlesign{\prec}_{s} \mathcal{H}_{\lambda} X = \mathcal{H}_{\lambda} X \circlesign{\preccurlyeq}_{s} w^{\mathcal{L}}, \\
& \mathcal{H}_{\lambda} X \otimes_{s} w^{\mathcal{H}} - w^{\mathcal{H}} \circlesign{\prec}_{s} \mathcal{H}_{\lambda} X = \mathcal{H}_{\lambda} X \circlesign{\preccurlyeq}_{s} w^{\mathcal{H}},
\end{align}
\end{subequations}
where 
\begin{equation*}
f  \circlesign{\preccurlyeq}_{s} g \triangleq \sum_{i \geq -1} S_{i-1} f  \otimes_{s} \Delta_{i} g + \sum_{i\geq -1} \sum_{j: \lvert j\rvert \leq 1} \Delta_{i} f \otimes_{s} \Delta_{i+j} g 
\end{equation*} 
(see Section~\ref{Prelim:Besov} for details) so that 
\begin{align*}
&\int_{\mathbb{T}^{2}} \mathbb{P}_{L} \divergence \left( X \otimes_{s} w - w \circlesign{\prec}_{s} \mathcal{H}_{\lambda} X \right) \cdot w^{\mathcal{L}} dx \\
=& \int_{\mathbb{T}^{2}} \divergence \Bigg( \mathcal{H}_{\lambda} X \circlesign{\preccurlyeq}_{s} w^{\mathcal{L}} + \mathcal{H}_{\lambda} X \circlesign{\preccurlyeq}_{s} w^{\mathcal{H}} + \mathcal{L}_{\lambda} X \otimes_{s} w^{\mathcal{L}} + \mathcal{L}_{\lambda} X \otimes_{s} w^{\mathcal{H}} \Bigg) \cdot w^{\mathcal{L}} dx. 
\end{align*}
While $\mathcal{H}_{\lambda} X$ is singular, it is never on the high frequency side relative to $w^{\mathcal{L}}$ or $w^{\mathcal{H}}$. For the term $\mathcal{L}_{\lambda} X \otimes_{s} w^{\mathcal{L}} \cdot w^{\mathcal{L}}$, Hairer and Rosati applied the Anderson Hamiltonian result from \cite{AC15} to deduce a logarithmically diverging constant and that is sufficient to close this estimate, analogously to the phenomenon of logarithmically supercritical deterministic Navier-Stokes equations studied by Tao \cite{T09}. With such $L^{2}(\mathbb{T}^{2})$-estimate on $w^{\mathcal{L}}$, the authors extended the local-in-time unique mild solution starting from initial data in $\mathscr{C}^{-1+\kappa}(\mathbb{T}^{2})$ for any $\kappa > 0$, globally in time. Additionally, they constructed a high-low (HL) weak solution to \eqref{NS} globally in time by Galerkin approximation type procedure on $\{w_{n}\}_{n}$, a family of solutions arising from $X^{n} = \mathcal{L}_{n} X$, and proved its path-wise uniqueness (see \cite[Definition 7.1]{HR24} for its specific definition). We also mention that  Hairer and Zhao \cite{HZ25} subsequently provided ergodicity result based on the framework of \cite{HR24}, although the hypothesis of \cite{HZ25} disallows forcing by STWN. 

In an attempt to adapt such an argument for the 2D Cahn-Hilliard equation \eqref{Singular:CH}, analogously to \eqref{HR:Eq:X} we define 
\begin{equation}\label{Eq:X}
\partial_{t} X + \Delta^{2} X = \divergence \xi, \qquad X(0) = 0. 
\end{equation}
Analogously to \eqref{Reg:Psi}, we know $\mathbb{P}$-a.s. 
\begin{equation}\label{Reg:X}
X = \int_{0}^{\cdot} e^{-(\cdot - s) \Delta^{2}} \divergence \xi(s) ds \in C_{t} \mathscr{C}^{-\kappa}(\mathbb{T}^{2}) \qquad \text{for any } \kappa >  0. 
\end{equation}
Analogously to \eqref{HR:Eq:v} we want to consider $v = \Psi - X$ such that $\partial_{t} v + \Delta^{2} v = \Delta (v+X)^{3}$. Here, because such $X^{2}$ or $X^{3}$ are ill-defined, we consider instead $X^{n} = \mathcal{L}_{n} X$ for $n \in \mathbb{N}$ and define 
\begin{equation}\label{Def:diamond}
(X^{n})^{\diamondsuit 2} \triangleq (X^{n})^{2} - \mathbb{E} [ (X^{n})^{2} ], \qquad (X^{n})^{\diamondsuit 3} \triangleq (X^{n})^{3} - 3 \mathbb{E} [ (X^{n})^{2} ] X^{n},
\end{equation}
which are both standard Wick products $:\!(X^{n})^{2}\!:$ and $:\! (X^{n})^{3}\!:$ (e.g. \cite{J97}, \cite[Example 2.1]{Y21}), so that for any $T > 0$ fixed, all $p \in [1,\infty)$, for any $\kappa > 0$, 
\begin{subequations}\label{Renorm}
\begin{align}
& X^{n} \to X \qquad \qquad \, \, \text{ in } L^{p} (\Omega; C([0,T]; \mathscr{C}^{-\kappa}(\mathbb{T}^{2})),  \label{Conv:X:1}\\
& (X^{n})^{\diamondsuit 2} \to X^{\diamondsuit 2} \qquad \text{ in } L^{p} (\Omega; C([0,T]; \mathscr{C}^{-2\kappa} (\mathbb{T}^{2})), \label{Conv:X:2} \\
& (X^{n})^{\diamondsuit 3} \to X^{\diamondsuit 3} \qquad \text{ in } L^{p} (\Omega; C([0,T]; \mathscr{C}^{-3\kappa}(\mathbb{T}^{2})) \label{Conv:X:3}
\end{align}
\end{subequations}
as $n\to\infty$. One can find in \cite[Theorem 4.3]{CC18} in the case of the $\Phi_{3}^{4}$ model the statements that are similar to \eqref{Renorm}. However, \eqref{Eq:X} is a bit different from $\partial_{t}X = \Delta X = \xi$ in the case of the $\Phi_{3}^{4}$ model; thus, we provide details in Section~\ref{Sec:5}. In fact, the proof of \eqref{Renorm} has interesting aspects; we obtain logarithmic divergence term in \eqref{Bd:IV1}-\eqref{Bd:IV2}, and consequently double logarithmic divergence term in \eqref{Bd:V1}-\eqref{Bd:V2} due to the criticality of the 2D case. 

We now consider 
\begin{equation}\label{Eq:v}
\partial_{t} v + \Delta^{2} v = \Delta ( v^{3} + 3v^{2} X + 3v X^{\diamondsuit 2} + X^{\diamondsuit 3})
\end{equation}
and define analogously to \eqref{HR:Eq:Y} 
\begin{equation}\label{Eq:Y} 
\partial_{t} Y + \Delta^{2} Y  = \Delta \left( 3 Y  X^{\diamondsuit 2} + X^{\diamondsuit 3} \right), \qquad Y(0) = 0 
\end{equation}
so that $w \triangleq v - Y$ satisfies 
\begin{align}
&\partial_{t} w + \Delta^{2} w  = \Delta \Bigg( w^{3} + 3wY^{2} + 3w^{2} Y + Y^{3}  + 3 \left( w^{2} + 2w Y + Y^{2} \right)X  + 3w X^{\diamondsuit 2}\Bigg), \label{Eq:w}\\
& w(0) = \Psi^{\text{in}}.\nonumber
\end{align}

Because \eqref{Conv:X:2} implies $X^{\diamondsuit 2} \in C_{t} \mathscr{C}^{-2\kappa}(\mathbb{T}^{2})$, we see that $w \in C_{t} \mathscr{C}^{2-3\kappa}(\mathbb{T}^{2})$. Leaving aside the low frequency projection $\mathcal{L}_{\lambda}$ for this discussion, we see that the ill-defined products upon the $L^{2}(\mathbb{T}^{2})$-estimate of $w$ in \eqref{Eq:w} are not just 
\begin{equation}\label{Prod:1}
\int_{\mathbb{T}^{2}} \Delta( w X^{\diamondsuit 2}) wdx = \int_{\mathbb{T}^{2}}w \underbrace{X^{\diamondsuit 2}}_{\in\mathscr{C}^{0-} } \underbrace{\Delta w }_{\in\mathscr{C}^{0-} }dx,  \qquad \int_{\mathbb{T}^{2}}\underbrace{\Delta w}_{\in\mathscr{C}^{0-} } \underbrace{\Delta w}_{\in\mathscr{C}^{0-} }dx,  
\end{equation} 
but also 
\begin{equation}\label{Prod:2}
\int_{\mathbb{T}^{2}}\Delta(  w^{2} X) wdx = \int_{\mathbb{T}^{2}} w^{2} \underbrace{X}_{\in \mathscr{C}^{0-}} \underbrace{\Delta w}_{\in \mathscr{C}^{0-}} dx. 
\end{equation}
Thus, in contrast to \eqref{Cancel:out}, we need to cancel out the high frequency projections of not only $\mathcal{H}_{\lambda} X$ but also $\mathcal{H}_{\lambda} X^{\diamondsuit 2}$, suggesting the need for new ideas beyond \eqref{HR:para:ans}. 

\section{Main results}
\subsection{Statement of main results}
The following definition is partially inspired by that of HL weak solution to \eqref{NS} from \cite[Definition 7.1]{HR24}. 
\begin{define}\label{Def:2.1}
Suppose that $Y$ solves \eqref{Eq:Y} with $X$ defined by \eqref{Reg:X}. Given any $\Psi^{\text{in}} \in \dot{H}^{-1+\alpha}(\mathbb{T}^{2})$, $\alpha \in (0,1)$, that has zero mean, $v \in C([0,\infty); \mathcal{S}'(\mathbb{T}^{2}; \mathbb{R}))$ is an analytically weak solution to \eqref{Eq:v} if $w = v - Y$ satisfies for any $t > 0$, $\mathbb{P}$-a.s. 
\begin{align}
& \int_{\mathbb{T}^{2}} w(t,x) \psi(t,x) dx - \int_{\mathbb{T}^{2}} \Psi^{\text{in}}(x) \psi(0,x) dx - \int_{0}^{t} \int_{\mathbb{T}^{2}} w(s,x) \partial_{s} \psi(s,x) dx ds  \label{CH:weak} \\
& + \int_{0}^{t} \int_{\mathbb{T}^{2}} w(s,x)\Delta^{2} \psi(s,x) dx ds \nonumber \\
=& \int_{0}^{t} \int_{\mathbb{T}^{2}} \Bigg( w^{3} + 3w Y^{2} + 3w^{2} Y + Y^{3} + 3 \left( w^{2} + 2w Y + Y^{2} \right) X + 3w  X^{\diamondsuit 2} \Bigg)(s,x) \Delta \psi(s,x) dx ds \nonumber 
\end{align}
for all $\psi \in C_{c}^{\infty} ((-\infty, T)\times \mathbb{T}^{2})$. 
\end{define}
We now state our main result. 
\begin{theorem}\label{Thm:2.1}
Let $(\Omega, \mathcal{F}, \mathbb{P})$ be a probability space on which STWN $\xi$ satisfies Definition~\ref{Def:STWN}. Then there exists a null set $\mathcal{N} \subset \Omega$, specified in Proposition~\ref{Prop:3.1}, such that for any $\omega \in \Omega \setminus \mathcal{N}$, any $\kappa > 0$, any $\alpha \in (0,1)$, any $\Psi^{\text{in}} \in \dot{H}^{-1+\alpha}(\mathbb{T}^{2})$ that has zero mean, the following holds.
\begin{enumerate}
\item (Existence and regularity) There exists an analytically weak solution $v(\cdot, \cdot; \omega, \Psi^{\text{in}})$ to \eqref{Eq:v} according to Definition~\ref{Def:2.1} on $[0,\infty)$. Additionally, $w$ satisfies for all $T > 0$, $\mathbb{P}$-a.s.
\begin{equation*}
w \in L_{T}^{\infty} \dot{H}^{-1+\alpha}(\mathbb{T}^{2}) \cap L_{T}^{2} \dot{H}^{1+\alpha}(\mathbb{T}^{2}) \cap L_{T}^{4} L^{\frac{4}{1-\alpha}}(\mathbb{T}^{2}). 
\end{equation*} 
\item (Path-wise uniqueness) Additionally, assume that $\alpha \in (\frac{2}{3}, 1)$. Then, such a solution is unique path-wise. 
\end{enumerate} 
\end{theorem}

We make several remarks. 
\begin{remark}\hfill 
\begin{enumerate}
\item Our proof follows the setting given in Section~\ref{Subsec:HR24} and the key to our proof is the inverse Laplacian formulation \eqref{CH:Laplace:inverse}. As we saw, $L^{2}(\mathbb{T}^{2})$-estimate is necessary for the well-posedness in the mild formulation of \eqref{Eq:w} and seems challenging; however, it is not necessary for the weak formulation of \eqref{Eq:w}. Moreover, because the products such as $X^{\diamondsuit 2} \Delta w, \Delta w \Delta w$, and $X \Delta w$ in \eqref{Prod:1}-\eqref{Prod:2} become barely ill-defined at the level of $L^{2}(\mathbb{T}^{2})$-estimate, we are able to employ the $\dot{H}^{-1}(\mathbb{T}^{2})$-estimate upon a Galerkin approximation on \eqref{Eq:w} (see Proposition~\ref{Prop:3.2}). Unfortunately, it was known even in the deterministic case that such  $\dot{H}^{-1}(\mathbb{T}^{d})$-bound is not enough to deduce strong convergence in $L^{3}(0,T; L^{3}(\mathbb{T}^{d}))$ for $d \geq 2$, which is needed to take the limit in the nonlinear term (see the proof of \cite[Proposition 3.10.4]{Y26} and Remark~\ref{Rem:3.1}). Thus, we take one step further and employ not only the $\dot{H}^{-1}(\mathbb{T}^{2})$-estimate but also the $\dot{H}^{-1+\alpha}(\mathbb{T}^{2})$-estimate for $\alpha \in (0,1)$; because the products were barely ill-defined at the level of $L^{2}(\mathbb{T}^{2})$-estimate, this turned out to be possible (see Proposition~\ref{Prop:3.3}). To treat the damping term upon such $\dot{H}^{-1+\alpha}(\mathbb{T}^{2})$-estimate, we will rely on the positivity lemma of the fractional Laplacian (see Lemma~\ref{Lem:Max:Prin}), which is typically used on fractional diffusion term such as the SQG equations \eqref{SQG}. It turns out that such $\dot{H}^{-1+\alpha}(\mathbb{T}^{2})$-estimate for $\alpha \in (0,1)$ is sufficient to deduce not only the strong convergence in $L^{3}(0,T; L^{3}(\mathbb{T}^{2}))$ to pass the limit in the nonlinear term, but also prove path-wise uniqueness, which implies that the solution is probabilistically strong thanks to the classical Yamada-Watanabe theorem (e.g. \cite{C03}).

In previous works such as \cite{M19}, $\dot{H}^{-1}(\mathbb{T}^{d})$-estimate was utilized already, but we seem to be the first to utilize $\dot{H}^{-1 + \alpha}(\mathbb{T}^{d})$-estimate for $\alpha \in (0,1)$ for the Cahn-Hilliard equation. Our proof in the case of zero noise gives the same estimate in the deterministic case immediately and may be of interest to researchers on the deterministic Cahn-Hilliard equation. 

\item The additional requirement of $\alpha > \frac{2}{3}$ in Theorem~\ref{Thm:2.1} is only for the convenience, rooting from the Sobolev embedding $W^{\frac{2}{3} + \kappa, 3}(\mathbb{T}^{2}) \hookrightarrow L^{\infty} (\mathbb{T}^{2})$, and may be optimized better to reduce it below $\frac{2}{3}$. 
\end{enumerate}
\end{remark}

We mention that some other PDEs in fluid mechanics have similar transformations; e.g., the magnetic potential (e.g. \cite{W58}) of the MHD and Hall-MHD systems. Particularly, considering the momentum formulation of the SQG equations, which is an equation solved by $\Upsilon$ that satisfies $-\nabla^{\bot} \cdot \Upsilon = \theta$ where $\theta$ solves \eqref{SQG} (see \cite{BSV19, Y25a, Y25b} for its usage), we can readily extend Theorem~\ref{Thm:2.1} (1) to the singular SQG equations forced by $\xi$ and diffusion of full Laplacian, specifically
\begin{equation}\label{Singular:SQG}
\partial_{t} \theta + u \cdot \nabla  \theta = \Delta \theta +  \xi, \qquad u = \mathcal{R}^{\bot} \theta.
\end{equation}
Let us make this precise. The setting and proof are simpler than the case of the Cahn-Hilliard equation \eqref{Singular:CH}. 

Analogously to \eqref{Eq:X}, we consider 
\begin{equation}\label{Eq:X:SQG}
\partial_{t} X = \Delta X +\xi, \qquad X(0) = 0.   \end{equation}
It follows that $\mathbb{P}$-a.s. $X \in C_{t} \mathscr{C}^{-\kappa}(\mathbb{T}^{2})$ for any $\kappa > 0$. We then consider $X^{n} = \mathcal{L}_{n} X$ for $n \in \mathbb{N}$ and define 
\begin{equation}\label{Def:diamond:SQG}
( \mathcal{R}^{\bot} X^{n} ) \diamondsuit X^{n} \triangleq ( \mathcal{R}^{\bot} X^{n} )  X^{n} - \mathbb{E} [ ( \mathcal{R}^{\bot} X^{n} )  X^{n}],
\end{equation}
the standard Wick product $:\!( \mathcal{R}^{\bot} X^{n} ) X^{n}\!:$ so that for any $T > 0$ fixed, all $p \in [1,\infty)$, for any $\kappa > 0$, 
\begin{subequations}
\begin{align}
& X^{n} \to X, \qquad \qquad  \qquad \qquad \, \,\,\,\, \,\text{ in } L^{p} (\Omega; C([0,T]; \mathscr{C}^{-\kappa}(\mathbb{T}^{2})) \\
& ( \mathcal{R}^{\bot} X^{n} ) \diamondsuit X^{n} \to ( \mathcal{R}^{\bot} X) \diamondsuit X, \qquad \text{ in } L^{p} (\Omega; C([0,T]; \mathscr{C}^{-2\kappa}(\mathbb{T}^{2})). 
\end{align} 
\end{subequations}
Analogously to \eqref{Eq:v} we consider $v \triangleq \theta - X$ that solves 
\begin{equation}\label{Eq:v:SQG}
\partial_{t} v + \divergence \left( \mathcal{R}^{\bot} v v + \mathcal{R}^{\bot} v X + \mathcal{R}^{\bot} X v + \mathcal{R}^{\bot} X \diamondsuit X \right) = \Delta v. 
\end{equation}
We define analogously to \eqref{Eq:Y}
\begin{equation}\label{Eq:Y:SQG}
\partial_{t} Y + \divergence \left( \mathcal{R}^{\bot} Y X + \mathcal{R}^{\bot} X Y + \mathcal{R}^{\bot} X \diamondsuit X \right) = \Delta Y 
\end{equation}
so that $w \triangleq v - Y$ satisfies 
\begin{subequations}\label{Eq:W:SQG}
\begin{align}
& \partial_{t} w + \divergence \left(  \mathcal{R}^{\bot} w w + \mathcal{R}^{\bot} w Y + \mathcal{R}^{\bot} Y w + \mathcal{R}^{\bot} Y Y + \mathcal{R}^{\bot} w X + \mathcal{R}^{\bot} X w \right) =  \Delta w,  \\
& w(0) = \theta^{\text{in}}. 
\end{align}
\end{subequations}
We are now ready to define the solution to \eqref{Singular:SQG}.
\begin{define}\label{Def:Sol:Sing:SQG}
Suppose that $Y$ solves \eqref{Eq:Y:SQG} with $X$ defined by \eqref{Eq:X:SQG}. Given any $\theta^{\text{in}} \in \dot{H}^{-\frac{1}{2}}(\mathbb{T}^{2})$ that has zero mean, $v \in C([0,\infty); \mathcal{S}'(\mathbb{T}^{2}; \mathbb{R}))$ is an analytically weak solution to \eqref{Eq:v:SQG} if $w = v - Y$ satisfies for any $t > 0$, $\mathbb{P}$-a.s. 
\begin{align}
& \int_{\mathbb{T}^{2}} w(t,x) \psi(t,x) dx - \int_{\mathbb{T}^{2}} \theta^{\text{in}}(x) \psi(0,x) dx - \int_{0}^{t} \int_{\mathbb{T}^{2}} w(s,x) \partial_{s} \psi(s,x) dx ds \\
& - \int_{0}^{t} \int_{\mathbb{T}^{2}} w(s,x) \Delta \psi(s,x) dx ds =\int_{0}^{t} \int_{\mathbb{T}^{2}} \Bigg( \mathcal{R}^{\bot} w \cdot \nabla \psi w + \mathcal{R}^{\bot} w \cdot \nabla \psi Y  \nonumber \\
& \qquad \qquad \qquad \qquad + \mathcal{R}^{\bot} Y \cdot \nabla \psi w + \mathcal{R}^{\bot} Y \cdot \nabla \psi Y  + \mathcal{R}^{\bot} w \cdot \nabla \psi X + \mathcal{R}^{\bot} X \cdot \nabla \psi w \Bigg) dx ds  \nonumber 
\end{align}
for all $\psi \in C_{c}^{\infty} ((-\infty, T) \times \mathbb{T}^{2})$. 
\end{define}
We are now able to state the global-in-time existence of a solution to \eqref{Singular:SQG}. 

\begin{theorem}\label{Thm:SQG}
Let $(\Omega, \mathcal{F}, \mathbb{P})$ be a probability space on which STWN $\xi$ satisfies Definition~\ref{Def:STWN}. There exists a null set $\mathcal{N} \subset \Omega$, specified in Proposition~\ref{Prop:C.1}, such that for any $\omega \in \Omega \setminus \mathcal{N}$, any $\kappa > 0$, and any $\theta^{\text{in}} \in \dot{H}^{-\frac{1}{2}}(\mathbb{T}^{2})$ that has zero mean, the following holds.

There exists an analytically weak solution $v(\cdot, \cdot; \omega, \theta^{\text{in}})$ to \eqref{Eq:v:SQG} according to Definition~\ref{Def:Sol:Sing:SQG} on $[0,\infty)$. Additionally, $w$ satisfies for all $T > 0$, $\mathbb{P}$-a.s.
\begin{equation}
w \in L_{T}^{\infty} \dot{H}^{-\frac{1}{2}}(\mathbb{T}^{2}) \cap L_{T}^{2} \dot{H}^{\frac{1}{2}}(\mathbb{T}^{2}). 
\end{equation} 
\end{theorem}
The proof of Theorem~\ref{Thm:SQG} is more straight-forward than that of Theorem~\ref{Thm:2.1}(1) and thus we sketch its proof in the Appendix~\ref{Appendix:Pf:Thm:SQG}. The reason why we are unable to deduce path-wise uniqueness in Theorem~\ref{Thm:SQG} is explained in Remark~\ref{Rem:C.1}. 

\subsection{Remarks on the mild formulation of \eqref{Eq:w}} 
After this pre-print was initially posted on ArXiv on 08/30/2026, the author was informed by the authors of \cite{RYZ21} that they obtained a similar result. They did the Da Prato-Debussche trick once, deduced $\dot{H}^{-1}(\mathbb{T}^{2})$-bound, and proved path-wise uniqueness. Completely independently, inspired by the HL-weak solution framework of \cite{HR24}, we did Da Prato-Debussche trick twice, deduced not only $\dot{H}^{-1}(\mathbb{T}^{2})$- but also $\dot{H}^{-1+\alpha}(\mathbb{T}^{2})$-bounds for all $\alpha \in (0,1)$, proved strong convergence in $L^{3}(0, T; W^{\alpha - \kappa, 3} (\mathbb{T}^{2}))$ (see \eqref{Conv:c}), and consequently path-wise uniqueness. 

Additionally, \cite[Remark 4.5]{RYZ21} suggests without proof that local-in-time existence of a unique mild solution is possible with initial data in $\mathscr{C}^{-\frac{4}{3} +}$. Because we obtained $\dot{H}^{-1+\alpha}(\mathbb{T}^{2})$-estimate for any $\alpha \in (0,1)$ in Proposition~\ref{Prop:3.3} and $H^{-\frac{1}{3} + 2 \kappa} (\mathbb{T}^{2}) \hookrightarrow \mathscr{C}^{-\frac{4}{3} + \kappa} (\mathbb{T}^{2})$ for any $\kappa > 0$, this raises hope that our Proposition~\ref{Prop:3.3} will lead to the first result of global-in-time unique mild solution if the local-in-time solution theory can indeed be proven starting from $\mathscr{C}^{-\frac{4}{3} +}(\mathbb{T}^{2})$-initial data. Unfortunately, we were unsuccessful in doing so. Precisely, we were able to construct a unique local-in-time mild solution starting from initial data in $\mathscr{C}^{-1+\kappa}(\mathbb{T}^{2})$ for any $\kappa > 0$. For completeness of this discussion, we present this theorem below, which is inspired by \cite[Theorem 2.1]{MW17b}, and prove it in Section~\ref{Sec:6}. We hope this gap between the local-in-time unique mild solution theory starting from $\mathscr{C}^{-1+\kappa}(\mathbb{T}^{2})$ from Theorem~\ref{Thm:2.2} and $H^{-\kappa}(\mathbb{T}^{2})$-bound from Proposition~\ref{Prop:3.3} for any $\kappa > 0$, can be closed in the future. 

\begin{theorem}[Local-in-time existence of unique mild solution of \eqref{Eq:w}]\label{Thm:2.2}
Let $(\Omega, \mathcal{F}, \mathbb{P})$ be a probability space on which STWN $\xi$ satisfies Definition~\ref{Def:STWN}. Fix 
\begin{equation}\label{Def:kappa0:beta:gamma}
\kappa_{0} \in \left(\frac{1}{3}, \frac{1}{2} \right), \qquad \beta \in (0, 1), \qquad  \gamma \in \left(\beta, \frac{4}{3}\right).
\end{equation}
and then $\kappa > 0$ sufficiently small, e.g. 
\begin{equation}\label{Small:kappa}
10 \kappa < \min \left\{ \kappa_{0} - \frac{1}{3},  \gamma-\beta  \right\}.  \end{equation}
For $T \in (0,1]$, define 
\begin{subequations}
\begin{align}
\mathcal{X}_{T} \triangleq& \Bigg\{ f \in C([0,T]; \mathscr{C}^{-\frac{4}{3} + \kappa_{0}}(\mathbb{T}^{2})) \cap C((0,T]; \mathscr{C}^{\gamma} (\mathbb{T}^{2})) \cap C^{\frac{\beta}{4}} ((0, T]; L^{\infty} (\mathbb{T}^{2})): \nonumber \\
&  \qquad \qquad \qquad \qquad \qquad \qquad \qquad \qquad \qquad \qquad \qquad \lVert f \rVert_{\mathcal{X}_{T}} < \infty \Bigg\}, \label{Def:X:T}\\
\lVert f \rVert_{\mathcal{X}_{T}} \triangleq& \max \Bigg\{ \sup_{t \in [0,T]} \lVert f(t) \rVert_{\mathscr{C}^{-\frac{4}{3} + \kappa_{0}}}, \sup_{t\in (0, T]} t^{\frac{ \gamma + \frac{4}{3} - \kappa_{0} }{4}} \lVert f(t) \rVert_{\mathscr{C}^{\gamma}},  \nonumber \\
& \qquad \qquad \qquad \qquad \qquad \sup_{0 < s < t \leq T} s^{\frac{ \beta - (-\frac{4}{3} + \kappa_{0}) + \kappa}{4}} \frac{\lVert f(t) - f(s) \rVert_{L^{\infty}}}{\lvert t-s \rvert^{\frac{\beta}{4}}} \Bigg\},   \label{Norm:X:T}  
\end{align}
and
\end{subequations}
\begin{subequations}
\begin{align}
\bar{\mathcal{X}}_{T} \triangleq&  C([0,T]; \mathscr{C}^{\gamma} (\mathbb{T}^{2})) \cap C^{\frac{\beta}{4}} ([0,T]; L^{\infty} (\mathbb{T}^{2})), \label{Def:bar:X:T} \\
\lVert f \rVert_{\bar{\mathcal{X}}_{T}} \triangleq& \max \Bigg\{ \sup_{t \in [0,T]} \lVert f(t) \rVert_{\mathscr{C}^{\gamma}},  \sup_{0 \leq s < t \leq T}   \frac{\lVert f(t) - f(s) \rVert_{L^{\infty}}}{\lvert t-s \rvert^{\frac{\beta}{4}}} \Bigg\}.    \label{Norm:bar:X:T}  
\end{align}
\end{subequations}
Then there exists a null set $\mathcal{N} \subset \Omega$, specified in Proposition~\ref{Prop:3.1}, such that for any $\omega \in \Omega \setminus \mathcal{N}$, the following holds.

\begin{enumerate}
\item For all $\Psi^{\text{in}} \in \mathscr{C}^{-\frac{4}{3} + \kappa_{0}}(\mathbb{T}^{2})$ that is mean-zero, there exists $T^{\ast} \in (0, 1]$ such that \eqref{Eq:w} has a unique mild solution on $[0, T^{\ast})$ where $T^{\ast}$ can be chosen to be maximal in the sense that either $T^{\ast} = 1$ and $w \in \mathcal{X}_{1}$ or $\lim_{t\nearrow T^{\ast} < 1} \lVert w(t) \rVert_{\mathscr{C}^{-\frac{4}{3} + \kappa_{0}}} = \infty$, the case in which the restriction $w \rvert_{[0,T]}$ for $[0, T] \subset [0, T^{\ast})$ takes values in $\mathcal{X}_{T}$ and the choice of maximal existence time $T^{\ast}$ and $w$ with such a property is unique. 
\item If $\Psi^{\text{in}} \in \mathscr{C}^{\gamma} (\mathbb{T}^{2})$, then the solution $w$ in (1) is continuous at initial time in a sense that the statement of (1) continues to hold with $\mathcal{X}_{T}$ replaced by $\bar{\mathcal{X}}_{T}$. 
\end{enumerate}
\end{theorem}

\section{Proof of Theorem~\ref{Thm:2.1} (1): existence and regularity}
Our proof is partially inspired by that of the global existence of HL weak solution to \eqref{NS} in \cite[Section 7]{HR24}. 
\subsection{Estimates}
We define for any $t \in [0,\infty)$ and $\kappa > 0$,  
\begin{equation}\label{Def:Lt:kappa}
L_{t}^{\kappa, n} \triangleq 1 + \lVert X^{n} \rVert_{C_{t} \mathscr{C}^{-\kappa}} + \lVert (X^{n})^{\diamondsuit 2} \rVert_{C_{t} \mathscr{C}^{-2\kappa}} + \lVert Y^{n} \rVert_{C_{t}\mathscr{C}^{2-4\kappa}}, \qquad L_{t}^{\kappa} \triangleq \sup_{n\in\mathbb{N}} L_{t}^{\kappa,n}, 
\end{equation}
where $X^{n} \triangleq \mathcal{L}_{n}X$ and $X$ solves \eqref{Eq:X}, $(X^{n})^{\diamondsuit 2} \overset{\eqref{Def:diamond}}{=} (X^{n})^{2} - \mathbb{E} [ (X^{n})^{2}]$, and $Y^{n}$ solves \eqref{Eq:Y} with $X$ and $X^{\diamondsuit 2}$ replaced respectively by $X^{n}$ and $(X^{n})^{\diamondsuit 2}$. We will prove \eqref{Renorm} in Section~\ref{Sec:5} and the following is one of its consequences.   

\begin{proposition}\label{Prop:3.1}
Let $(\Omega, \mathcal{F}, \mathbb{P})$ be a probability space on which STWN $\xi$ satisfies Definition~\ref{Def:STWN}. Then there exists a null set $\mathcal{N} \subset \Omega$ such that 
\begin{equation}\label{Bd:Lt:kappa}
L_{t}^{\kappa}(\omega) < \infty \qquad \forall \, \omega \in \Omega \setminus \mathcal{N}, \forall \, t \geq 0, \forall \, \kappa > 0.    \end{equation}
\end{proposition}

We consider a smooth solution $w^{n}$ satisfying 
\begin{subequations}\label{Eq:Gal:w:eps}
\begin{align}
&\partial_{t} w^{n} + \Delta^{2} w^{n} = \Delta \Bigg( (w^{n})^{3} + 3w^{n} (Y^{n})^{2} + 3(w^{n})^{2} Y^{n} + (Y^{n})^{3}  \nonumber \\
& \qquad \qquad \qquad \qquad+ 3 \left( (w^{n})^{2} + 2w^{n} Y^{n} + (Y^{n})^{2} \right)X^{n}  + 3w^{n} (X^{n})^{\diamondsuit 2}\Bigg), \\
& w^{n}(0) = \mathcal{L}_{n} \Psi^{\text{in}}. 
\end{align}  
\end{subequations} 
We now employ multiple estimates and the following Gagliardo-Nirenberg inequalities will be utilized extensively: for any $j \in \mathbb{N}$ and all $\kappa > 0$ sufficiently small, specifically
\begin{equation}\label{kappa:small}
\kappa \in \left(0, \min \left\{ \frac{1-\alpha}{j}, \frac{1-\alpha}{4}, \frac{1}{4} + \frac{\alpha}{2}, \alpha \right\} \right),    
\end{equation}
\begin{subequations}\label{GN1} 
\begin{align}
& \lVert f \rVert_{L^{2}} \lesssim \lVert f \rVert_{\dot{H}^{-1}}^{\frac{1}{2}} \lVert f \rVert_{\dot{H}^{1}}^{\frac{1}{2}}, \qquad \qquad \qquad \,\,\,\,\, \,\lVert f \rVert_{H^{j\kappa}}\lesssim \lVert f \rVert_{\dot{H}^{-1}}^{\frac{1-j\kappa}{2}}\lVert f \rVert_{\dot{H}^{1}}^{\frac{1+j\kappa}{2}}, \label{GN1:a} \\
& \lVert f \rVert_{L^{2}} \lesssim \lVert f \rVert_{\dot{H}^{-1+\alpha}}^{\frac{1+\alpha}{2}} \lVert f \rVert_{\dot{H}^{1+\alpha}}^{\frac{1-\alpha}{2}}, \qquad \qquad \,\, \,\, \,\,\,\,\, \, \lVert f \rVert_{\dot{H}^{2\alpha}} \lesssim \lVert f \rVert_{\dot{H}^{-1+\alpha}}^{\frac{1-\alpha}{2}} \lVert f \rVert_{\dot{H}^{1+\alpha}}^{\frac{1+\alpha}{2}}, \label{GN1:b} \\
& \lVert f  \rVert_{\dot{H}^{2\alpha + j \kappa}} \lesssim \lVert f \rVert_{\dot{H}^{-1+\alpha}}^{\frac{1-\alpha-j\kappa}{2}} \lVert f \rVert_{\dot{H}^{1+\alpha}}^{\frac{1+\alpha+j\kappa}{2}}, \qquad \,\,\,\,\,\,\,\, \lVert f \rVert_{\dot{H}^{\frac{1}{2} + 2\kappa}} \lesssim \lVert f \rVert_{\dot{H}^{-1+\alpha}}^{\frac{ \frac{1}{2} + \alpha - 2 \kappa}{2}} \lVert f \rVert_{\dot{H}^{1+\alpha}}^{\frac{ \frac{3}{2} - \alpha + 2 \kappa}{2}}, \label{GN1:c} \\
& \lVert f \rVert_{\dot{H}^{1+\kappa}} \lesssim \lVert f \rVert_{\dot{H}^{1+\alpha}}^{\frac{2-\alpha + \kappa}{2}} \lVert f \rVert_{\dot{H}^{-1+\alpha}}^{\frac{\alpha - \kappa}{2}},  \qquad \qquad \lVert f \rVert_{\dot{H}^{j\kappa}} \lesssim \lVert f \rVert_{\dot{H}^{-1+\alpha}}^{\frac{1+ \alpha - j\kappa}{2}} \lVert f \rVert_{\dot{H}^{1+\alpha}}^{\frac{1- \alpha + j \kappa}{2}}. \label{GN1:d} 
\end{align}
\end{subequations}
We write $C(L_{t}^{\kappa})$ to denote a constant that depends on $L_{t}^{\kappa}$. 
\begin{proposition}\label{Prop:3.2}
Under the hypothesis of Theorem~\ref{Thm:2.1}, for all $\omega \in \Omega \setminus \mathcal{N}$, with $\mathcal{N}$ from Proposition~\ref{Prop:3.1}, the solution $w^{n}$ to \eqref{Eq:Gal:w:eps} satisfies  
\begin{equation}\label{Prop:3.2:main:claim} 
\sup_{t\in [0,T]} \lVert w^{n} (t) \rVert_{\dot{H}^{-1}}^{2} +\int_{0}^{T} \lVert w^{n} (t) \rVert_{\dot{H}^{1}}^{2} + \lVert w^{n} (t) \rVert_{L^{4}}^{4} dt \leq C 
\end{equation}
with the bound independent of $n$. 
\end{proposition}

\begin{proof}[Proof of Proposition~\ref{Prop:3.2}]
We take $L^{2}(\mathbb{T}^{2})$-inner products on \eqref{Eq:Gal:w:eps} with $(-\Delta)^{-1} w^{n}$ to obtain 
\begin{equation}\label{Prop:3.2:a}
\frac{1}{2} \partial_{t} \lVert w^{n} \rVert_{\dot{H}^{-1}}^{2} + \lVert w^{n} \rVert_{\dot{H}^{1}}^{2} + \lVert w^{n} \rVert_{L^{4}}^{4} = \sum_{k=1}^{7} \RomanI_{k}
\end{equation}
where 
\begin{subequations} 
\begin{align}
\RomanI_{1} \triangleq& - 3\int_{\mathbb{T}^{2}} w^{n} (Y^{n})^{2} w^{n} dx,  \qquad \, \, \,\RomanI_{2} \triangleq -3 \int_{\mathbb{T}^{2}} (w^{n})^{2} Y^{n} w^{n} dx,  \label{Def:I1:I2} \\
\RomanI_{3} \triangleq& - \int_{\mathbb{T}^{2}} (Y^{n})^{3} w^{n} dx, \qquad \qquad \, \, \RomanI_{4} \triangleq -3 \int_{\mathbb{T}^{2}} (w^{n})^{2}  X^{n} w^{n} dx,\label{Def:I3:I4} \\
\RomanI_{5} \triangleq& -6 \int_{\mathbb{T}^{2}} w^{n} Y^{n} X^{n} w^{n} dx, \qquad \, \, \, \RomanI_{6} \triangleq -3 \int_{\mathbb{T}^{2}} (Y^{n})^{2} X^{n} w^{n} dx, \label{Def:I5:I6}\\
\RomanI_{7} \triangleq& -3 \int_{\mathbb{T}^{2}} w^{n} (X^{n})^{\diamondsuit 2} w^{n} dx. \label{Def:I7} 
\end{align}
\end{subequations}
For $\RomanI_{1}$ from \eqref{Def:I1:I2}, by H$\ddot{\mathrm{o}}$lder's and Young's inequalities, 
\begin{equation}\label{Est:I1}
\RomanI_{1} \lesssim \lVert Y^{n} \rVert_{L^{\infty}}^{2} \lVert w^{n} \rVert_{L^{2}}^{2} 
\overset{\eqref{GN1:a} \eqref{Def:Lt:kappa}}{\lesssim} (L_{t}^{\kappa})^{2} \lVert w^{n} \rVert_{\dot{H}^{-1}} \lVert w^{n} \rVert_{\dot{H}^{1}}   \leq \frac{1}{32} \lVert w^{n} \rVert_{\dot{H}^{1}}^{2} + C(L_{t}^{\kappa}) \lVert w^{n} \rVert_{\dot{H}^{-1}}^{2}.    
\end{equation}

We can handle both $\RomanI_{2}$ and $\RomanI_{3}$ from \eqref{Def:I1:I2} and \eqref{Def:I3:I4} by Young's inequality as  
\begin{equation}\label{Est:I2:I3}
\RomanI_{2}+ \RomanI_{3} \leq \frac{1}{32} \lVert w^{n} \rVert_{L^{4}}^{4} + C \lVert Y^{n} \rVert_{L^{\infty}}^{4} \overset{\eqref{Def:Lt:kappa}}{\leq} \frac{1}{32} \lVert w^{n} \rVert_{L^{4}}^{4} + C(L_{t}^{\kappa}).    
\end{equation}

For $\RomanI_{4}$, we use the duality of $B_{\infty,\infty}^{-\kappa}$ and $B_{1,1}^{\kappa}$ and then \eqref{Est:MW17b:b} to deduce 
\begin{equation}\label{Prop:3.2:c}
\RomanI_{4}  \lesssim \lVert X^{n} \rVert_{B_{\infty,\infty}^{-\kappa}} \lVert (w^{n})^{3} \rVert_{B_{1,1}^{\kappa}} \overset{\eqref{Est:MW17b:b} \eqref{Def:Lt:kappa}}{\lesssim} L_{t}^{\kappa} \lVert (w^{n})^{2} \rVert_{L^{2}} \lVert w^{n} \rVert_{B_{2,1}^{\kappa}}.
\end{equation}
We apply H$\ddot{\mathrm{o}}$lder's and Young's inequalities to deduce from \eqref{Prop:3.2:c}
\begin{equation}\label{Est:I4}
\RomanI_{4}  \overset{\eqref{Embed:Besov:element}}{\lesssim} L_{t}^{\kappa} \lVert w^{n} \rVert_{L^{4}}^{2} \lVert w^{n} \rVert_{H^{2\kappa}} \overset{\eqref{GN1:a}}{\leq} \frac{1}{32} ( \lVert w^{n} \rVert_{L^{4}}^{4} + \lVert w^{n} \rVert_{\dot{H}^{1}}^{2}) + C(L_{t}^{\kappa}) \lVert w^{n} \rVert_{\dot{H}^{-1}}^{2}.    
\end{equation}

We use the duality of $B_{\infty,\infty}^{-\kappa}$ and $B_{1,1}^{\kappa}$ and Bony's decomposition \eqref{Bony:decomp} to treat $\RomanI_{5}$ as follows:
\begin{align} 
&\RomanI_{5} \overset{\eqref{Def:I5:I6}}{\lesssim} \lVert (w^{n})^{2} Y^{n}  \rVert_{B_{1,1}^{\kappa}} \lVert X^{n} \rVert_{B_{\infty,\infty}^{-\kappa}}   \nonumber  \\
\overset{\eqref{Def:Lt:kappa}  \eqref{Besov:prod:b}\eqref{Besov:prod:c}}{\lesssim}& L_{t}^{\kappa} \left( \lVert (w^{n})^{2} \rVert_{L^{1}} \lVert Y^{n}  \rVert_{B_{\infty,1}^{\kappa}} + \lVert (w^{n})^{2} \rVert_{B_{1,1}^{\kappa}} \lVert Y^{n} \rVert_{L^{\infty}} + \lVert (w^{n})^{2} \rVert_{B_{1,1}^{\kappa}} \lVert Y^{n} \rVert_{B_{\infty,1}^{0}} \right). \label{Prop:3.2:b}
\end{align}
Then we can continue from \eqref{Prop:3.2:b} to estimate by Young's inequality to conclude 
\begin{equation}\label{Est:I5} 
\RomanI_{5} \overset{\eqref{Est:MW17b:b} \eqref{Embed:Besov:element}\eqref{Def:Lt:kappa}}{\lesssim} L_{t}^{\kappa} \left( \lVert w^{n} \rVert_{L^{2}}^{2} L_{t}^{\kappa}  + \lVert w^{n} \rVert_{L^{2}} \lVert w^{n} \rVert_{H^{2\kappa}} L_{t}^{\kappa}  \right)  \overset{\eqref{GN1:a}}{\leq} \frac{1}{32} \lVert w^{n} \rVert_{\dot{H}^{1}}^{2} + C (L_{t}^{\kappa}) \lVert w^{n} \rVert_{\dot{H}^{-1}}^{2}. 
\end{equation}

To treat $\RomanI_{6}$, we rely on the duality of  $B_{\infty,\infty}^{-\kappa}$ and $B_{1,1}^{\kappa}$, Bony's decomposition \eqref{Bony:decomp} again to compute 
\begin{align}
\RomanI_{6} \overset{\eqref{Def:I5:I6}}{\lesssim}& \lVert (Y^{n})^{2} w^{n} \rVert_{B_{1,1}^{\kappa}} \lVert X^{n} \rVert_{B_{\infty,\infty}^{-\kappa}}  \label{Prop:3.2:d}\\
\overset{\eqref{Besov:prod:a}\eqref{Besov:prod:b}\eqref{Besov:prod:c}}{\lesssim}& L_{t}^{\kappa} \left( \lVert (Y^{n})^{2} \rVert_{B_{2,1}^{1+ 2\kappa}} \lVert w^{n} \rVert_{B_{2,1}^{-1-\kappa}} + \lVert (Y^{n})^{2} \rVert_{L^{2}} \lVert w^{n} \rVert_{B_{2,1}^{\kappa}} + \lVert (Y^{n})^{2} \rVert_{B_{2,1}^{\frac{\kappa}{2}}} \lVert w^{n} \rVert_{B_{2,1}^{\frac{\kappa}{2}}} \right). \nonumber 
\end{align}
We continue from \eqref{Prop:3.2:d} by Young's inequality to deduce 
\begin{equation}\label{Est:I6}
\RomanI_{6} \overset{\eqref{Est:MW17b:b}\eqref{Embed:Besov:element}\eqref{Def:Lt:kappa}}{\lesssim} (L_{t}^{\kappa})^{3} \lVert w^{n} \rVert_{\dot{H}^{-1}} + (L_{t}^{\kappa})^{3} \lVert w^{n} \rVert_{\dot{H}^{1}} \leq \frac{1}{32} \lVert w^{n} \rVert_{\dot{H}^{1}}^{2} + C ( L_{t}^{\kappa}) (1+ \lVert w^{n} \rVert_{\dot{H}^{-1}}^{2}).  
\end{equation}

Lastly, we work on $\RomanI_{7}$ from \eqref{Def:I7} by the duality of $B_{\infty,\infty}^{-2\kappa}$ and $B_{1,1}^{2\kappa}$ and Young's inequality
\begin{align}
\RomanI_{7}      \overset{\eqref{Def:Lt:kappa} \eqref{Est:MW17b:b}}{\lesssim}& L_{t}^{\kappa} \lVert w^{n} \rVert_{L^{2}}  \lVert w^{n} \rVert_{B_{2,1}^{2\kappa}} \label{Est:I7}\\
\overset{\eqref{Embed:Besov:element}\eqref{GN1:a}}{\lesssim}& L_{t}^{\kappa} \lVert w^{n} \rVert_{\dot{H}^{-1}}^{\frac{1}{2}} \lVert w^{n} \rVert_{\dot{H}^{1}}^{\frac{1}{2}} \lVert w^{n} \rVert_{\dot{H}^{-1}}^{\frac{1-3\kappa}{2}} \lVert w^{n} \rVert_{\dot{H}^{1}}^{\frac{1+3\kappa}{2}}  \leq \frac{1}{32} \lVert w^{n} \rVert_{\dot{H}^{1}}^{2} + C(L_{t}^{\kappa}) \lVert w^{n} \rVert_{\dot{H}^{-1}}^{2}. \nonumber 
\end{align}
Applying \eqref{Est:I1}, \eqref{Est:I2:I3}, \eqref{Est:I4}, \eqref{Est:I5}, \eqref{Est:I6}, and \eqref{Est:I7} to \eqref{Prop:3.2:a} gives us 
\begin{equation*}
\partial_{t} \lVert w^{n} \rVert_{\dot{H}^{-1}}^{2} + \lVert w^{n} \rVert_{\dot{H}^{1}}^{2} + \lVert w^{n} \rVert_{L^{4}}^{4} \lesssim C(L_{t}^{\kappa}) (1+ \lVert w^{n} \rVert_{\dot{H}^{-1}}^{2})
\end{equation*}
and completes the proof of Proposition~\ref{Prop:3.2}. 
\end{proof}

\begin{remark}\label{Rem:3.1}
To take the limit of the nonlinear term, we need strong convergence in $L^{3}(0,T; L^{3}(\mathbb{T}^{2}))$ and to deduce such from \eqref{Compact 1}, we need uniform bound in $L^{3}(0,T; V)$ where $V \Subset L^{3}(\mathbb{T}^{2})$, e.g. $W^{\epsilon,3}(\mathbb{T}^{2})$. Proposition~\ref{Prop:3.2} is not sufficient for this purpose because in general dimension, Gagliardo-Nirenberg inequality admits for $d \in \{1,2,3,4,5\}$ and $\epsilon \in (0, \frac{6-d}{6})$, 
\begin{equation}\label{Rem:3.1:a}
\lVert f \rVert_{L_{T}^{3}W^{\epsilon,3}}^{3} \lesssim \lVert f \rVert_{L_{T}^{\infty} H^{-1}}^{3(\frac{1-\epsilon}{2}) - \frac{d}{4}} \int_{0}^{T} \lVert f \rVert_{\dot{H}^{1}}^{\frac{3}{2} + \frac{6\epsilon + d}{4}} dt 
\end{equation}
where for any $\epsilon > 0$, $\frac{3}{2} + \frac{6\epsilon + 2}{4} > 2$, although Proposition~\ref{Prop:3.2} only gave us only $w^{n} \in L_{T}^{\infty} H^{-1} \cap L_{T}^{2} \dot{H}^{1}$ (see the proof of \cite[Proposition 3.10.4]{Y26}). 
\end{remark}
This motivates us to improve the $\dot{H}^{-1}(\mathbb{T}^{2})$-bound and we can indeed do so thanks to Lemma~\ref{Lem:Max:Prin} and Proposition~\ref{Prop:3.2} that we already have in hand.  

\begin{proposition}\label{Prop:3.3}
Under the hypothesis of Theorem~\ref{Thm:2.1}, for all $\omega \in \Omega \setminus \mathcal{N}$, with $\mathcal{N}$ from Proposition~\ref{Prop:3.1}, the solution $w^{n}$ to \eqref{Eq:Gal:w:eps} satisfies  
\begin{equation}\label{Claim:Prop:3.3}
\sup_{t\in [0,T]} \lVert w^{n} (t) \rVert_{\dot{H}^{-1+\alpha}}^{2} +\int_{0}^{T} \lVert w^{n} (t) \rVert_{\dot{H}^{1+\alpha}}^{2} + \lVert w^{n} (t) \rVert_{L^{\frac{4}{1-\alpha}}}^{4} dt \leq C 
\end{equation}
with the bound independent of $n$. 
\end{proposition}

\begin{proof}[Proof of Proposition~\ref{Prop:3.3}]
We take $L^{2}(\mathbb{T}^{2})$-inner products on \eqref{Eq:Gal:w:eps} with $(-\Delta)^{-1} \Lambda^{2\alpha} w^{n}$ and pay attention to the diffusive term to see that by Lemma~\ref{Lem:Max:Prin} there exists a universal Sobolev constant $C_{S} > 0$ such that 
\begin{align*}
\int_{\mathbb{T}^{2}} (w^{n})^{3} \Lambda^{2\alpha} w^{n} dx \geq& \frac{1}{2} \lVert \Lambda^{\alpha} \lvert w^{n} \rvert^{2} \rVert_{L^{2}}^{2} \\
\geq& \frac{1}{4} \lVert \Lambda^{\alpha} \lvert w^{n} \rvert^{2} \rVert_{L^{2}}^{2} + C_{S} \lVert \, \lvert w^{n} \rvert^{2} \, \rVert_{L^{\frac{2}{1-\alpha}}}^{2}  = \frac{1}{4} \lVert \Lambda^{\alpha} \lvert w^{n} \rvert^{2} \rVert_{L^{2}}^{2}  + C_{S} \lVert w^{n} \rVert_{L^{\frac{4}{1-\alpha}}}^{4}.  
\end{align*}
This leads to 
\begin{equation}\label{Prop:3.3:d}
\frac{1}{2} \partial_{t} \lVert w^{n} \rVert_{\dot{H}^{-1+\alpha}}^{2} + \lVert w^{n} \rVert_{\dot{H}^{1+\alpha}}^{2} + \frac{1}{4} \lVert \Lambda^{\alpha} \lvert w^{n} \rvert^{2} \rVert_{L^{2}}^{2}  + C_{S} \lVert w^{n} \rVert_{L^{\frac{4}{1-\alpha}}}^{4} \leq  \sum_{j=1}^{7} \RomanII_{j} 
\end{equation}
where 
\begin{subequations}
\begin{align}
\RomanII_{1} \triangleq& - 3\int_{\mathbb{T}^{2}} w^{n} (Y^{n})^{2} \Lambda^{2\alpha} w^{n} dx,  \qquad \,\,\, \RomanII_{2} \triangleq -3 \int_{\mathbb{T}^{2}} (w^{n})^{2} Y^{n} \Lambda^{2\alpha}w^{n} dx, \label{Def:II1:II2}\\
\RomanII_{3} \triangleq& - \int_{\mathbb{T}^{2}} (Y^{n})^{3} \Lambda^{2\alpha}w^{n} dx, \qquad  \qquad \, \,  \RomanII_{4} \triangleq -3 \int_{\mathbb{T}^{2}} (w^{n})^{2} X^{n}  \Lambda^{2\alpha}w^{n} dx,\label{Def:II3:II4} \\
\RomanII_{5} \triangleq& -6 \int_{\mathbb{T}^{2}} w^{n} Y^{n} X^{n} \Lambda^{2\alpha}w^{n} dx, \qquad \,\,\,\,\RomanII_{6} \triangleq -3 \int_{\mathbb{T}^{2}} (Y^{n})^{2} X^{n} \Lambda^{2\alpha}w^{n} dx, \label{Def:II5:II6}\\
\RomanII_{7} \triangleq& -3 \int_{\mathbb{T}^{2}} w^{n} (X^{n})^{\diamondsuit 2} \Lambda^{2\alpha}w^{n} dx. \label{Def:II7} 
\end{align}
\end{subequations}

We estimate $\RomanII_{1}$ by H$\ddot{\mathrm{o}}$lder's and Young's inequalities 
\begin{align}
\RomanII_{1} \overset{\eqref{Def:II1:II2}}{\lesssim}& \lVert Y^{n} \rVert_{L^{\infty}}^{2} \lVert w^{n} \rVert_{L^{2}} \lVert \Lambda^{2\alpha} w^{n} \rVert_{L^{2}}  \nonumber \\
\overset{\eqref{GN1:b}\eqref{Def:Lt:kappa}}{\lesssim}& (L_{t}^{\kappa})^{2} \lVert w^{n} \rVert_{\dot{H}^{-1+\alpha}}\lVert w^{n} \rVert_{\dot{H}^{1+\alpha}} \leq \frac{1}{32} \lVert w^{n} \rVert_{\dot{H}^{1+\alpha}}^{2} + C(L_{t}^{\kappa}) \lVert w^{n} \rVert_{\dot{H}^{-1+\alpha}}^{2}. \label{Est:II:1}
\end{align}

We estimate $\RomanII_{2}$ by H$\ddot{\mathrm{o}}$lder's and Young's inequalities, and Sobolev embeddings of $\dot{H}^{\alpha} (\mathbb{T}^{2}) \hookrightarrow L^{\frac{2}{1-\alpha}} (\mathbb{T}^{2})$ and $\dot{H}^{1-\alpha}(\mathbb{T}^{2}) \hookrightarrow L^{\frac{2}{\alpha}}(\mathbb{T}^{2})$, 
\begin{align}
\RomanII_{2} \overset{\eqref{Def:II1:II2}\eqref{Classical product estimate}}{\lesssim}& \left( \lVert \Lambda^{\alpha} \lvert w^{n} \rvert^{2} \rVert_{L^{2}} \lVert Y^{n} \rVert_{L^{\infty}} + \lVert (w^{n})^{2} \rVert_{L^{\frac{2}{1-\alpha}}} \lVert \Lambda^{\alpha} Y^{n} \rVert_{L^{\frac{2}{\alpha}}} \right) \lVert \Lambda^{\alpha} w^{n} \rVert_{L^{2}} \nonumber \\
\overset{\eqref{Def:Lt:kappa}}{\lesssim}&  \lVert \Lambda^{\alpha} \lvert w^{n} \rvert^{2} \rVert_{L^{2}} L_{t}^{\kappa} \lVert \Lambda^{\alpha} w^{n} \rVert_{L^{2}} \leq \frac{1}{32} \lVert \Lambda^{\alpha} \lvert w^{n} \rvert^{2} \rVert_{L^{2}}^{2} + C (L_{t}^{\kappa})\lVert \Lambda^{\alpha} w^{n} \rVert_{L^{2}}^{2}. \label{Est:II:2}
\end{align}

Concerning $\RomanII_{3}$, we apply H$\ddot{\mathrm{o}}$lder's and Young's inequalities to deduce 
\begin{align}
\RomanII_{3} \overset{\eqref{Def:II3:II4}}{\lesssim}& \lVert (Y^{n})^{3}\rVert_{B_{2,2}^{2\alpha}} \lVert w^{n} \rVert_{L^{2}}  \label{Est:II:3} \\
\overset{\eqref{Est:MW17b:b}\eqref{GN1:b}}{\lesssim}& \lVert (Y^{n})^{2} \rVert_{L^{\infty}} \lVert Y^{n} \rVert_{B_{2,2}^{2\alpha}} \lVert w^{n} \rVert_{\dot{H}^{-1+\alpha}}^{\frac{1+\alpha}{2}} \lVert w^{n} \rVert_{\dot{H}^{1+\alpha}}^{\frac{1-\alpha}{2}} \leq  \frac{1}{32} \lVert w^{n} \rVert_{\dot{H}^{1+\alpha}}^{2} + C (L_{t}^{\kappa}) (1+ \lVert w^{n} \rVert_{\dot{H}^{-1+\alpha}}^{2}).   \nonumber 
\end{align}

For $\RomanII_{4}$, we rely on the duality of $B_{\infty,\infty}^{-\kappa}$ and $B_{1,1}^{\kappa}$ to compute 
\begin{align}
\RomanII_{4} \overset{\eqref{Def:II3:II4}}{\lesssim}& \lVert X^{n} \rVert_{B_{\infty,\infty}^{-\kappa}} \lVert (w^{n})^{2} \Lambda^{2\alpha} w^{n} \rVert_{B_{1,1}^{\kappa}} \nonumber  \\
\overset{\eqref{Def:Lt:kappa} \eqref{Est:MW17b:a} \eqref{Embed:Besov:element} \eqref{Est:MW17b:b}}{\lesssim}&  L_{t}^{\kappa} \left( \lVert w^{n} \rVert_{L^{4}}^{2} \lVert \Lambda^{2\alpha + 2 \kappa} w^{n} \rVert_{L^{2}} + \lVert w^{n} \rVert_{L^{4}} \lVert w^{n} \rVert_{ \dot{H}^{\frac{1}{2} + 2 \kappa}} \lVert \Lambda^{2\alpha} w^{n} \rVert_{L^{2}} \right). \label{Prop:3.3:a}
\end{align}
Then we rely on \eqref{GN1:c} and Young's inequality to conclude from \eqref{Prop:3.3:a} by 
\begin{align}
\RomanII_{4}  \lesssim&  L_{t}^{\kappa} \left( \lVert w^{n} \rVert_{L^{4}}^{2}  \lVert w^{n} \rVert_{\dot{H}^{-1+\alpha}}^{\frac{1-\alpha-2\kappa}{2}} \lVert w^{n} \rVert_{\dot{H}^{1+\alpha}}^{\frac{1+\alpha+2\kappa}{2}} + \lVert w^{n} \rVert_{L^{4}} \lVert w^{n} \rVert_{\dot{H}^{-1+\alpha}}^{\frac{ \frac{3}{2}  - 2\kappa}{2}} \lVert w^{n} \rVert_{\dot{H}^{1+\alpha}}^{\frac{\frac{5}{2}  + 2 \kappa}{2}}   \right) \nonumber \\
\leq& \frac{1}{32} \lVert \Lambda^{1+\alpha} w^{n} \rVert_{L^{2}}^{2} + C (L_{t}^{\kappa}) (1+ \lVert w^{n} \rVert_{L^{4}}^{4}) (1+ \lVert w^{n} \rVert_{\dot{H}^{-1+\alpha}}^{2}). \label{Est:II:4}
\end{align}

Concerning $\RomanII_{5}$,  we rely on the duality of $B_{\infty,\infty}^{-\kappa}$ and $B_{1,1}^{\kappa}$ and H$\ddot{\mathrm{o}}$lder's inequality to compute 
\begin{align}
\RomanII_{5} \overset{\eqref{Def:II5:II6}}{\lesssim}& \lVert X^{n} \rVert_{B_{\infty,\infty}^{-\kappa}} \lVert w^{n} Y^{n} \Lambda^{2\alpha}w^{n} \rVert_{B_{1,1}^{\kappa}} \label{Prop:3.3:b}\\
\overset{\eqref{Def:Lt:kappa} \eqref{Est:MW17b:a}\eqref{Embed:Besov:element} \eqref{GN1:b} }{\lesssim}& L_{t}^{\kappa} \Bigg( \lVert w^{n} \rVert_{L^{4}} \lVert Y^{n} \rVert_{L^{4}} \lVert \Lambda^{2\alpha + 2\kappa} w^{n} \rVert_{L^{2}}  \nonumber \\
& \qquad + \left( \lVert w^{n} \rVert_{L^{4}} \lVert Y^{n} \rVert_{B_{4,1}^{\kappa}} + \lVert w^{n} \rVert_{B_{4,1}^{\kappa}} \lVert Y^{n} \rVert_{L^{4}} \right) \lVert w^{n} \rVert_{\dot{H}^{-1+\alpha}}^{\frac{1-\alpha}{2}} \lVert w^{n} \rVert_{\dot{H}^{1+\alpha}}^{\frac{1+\alpha}{2}} \Bigg). \nonumber 
\end{align}
We continue from \eqref{Prop:3.3:b} by Young's inequality to conclude 
\begin{align}
\RomanII_{5} \overset{\eqref{GN1:c} \eqref{Def:Lt:kappa}}{\lesssim}&  (L_{t}^{\kappa})^{2} \left( \lVert w^{n} \rVert_{L^{4}} \lVert w^{n} \rVert_{\dot{H}^{-1+\alpha}}^{\frac{1-\alpha- 2 \kappa}{2}} \lVert w^{n} \rVert_{\dot{H}^{1+\alpha}}^{\frac{1+ \alpha + 2 \kappa}{2}} + \lVert w^{n} \rVert_{\dot{H}^{-1+\alpha}}^{\frac{ \frac{3}{2} - 2 \kappa}{2}} \lVert w^{n} \rVert_{\dot{H}^{1+\alpha}}^{\frac{ \frac{5}{2} + 2 \kappa}{2}} \right)  \nonumber \\
\leq& \frac{1}{32} \lVert w^{n} \rVert_{\dot{H}^{1+\alpha}}^{2} + C(L_{t}^{\kappa}) ( 1+ \lVert w^{n} \rVert_{L^{4}}^{2} ) \lVert w^{n} \rVert_{\dot{H}^{-1+\alpha}}^{2}.\label{Est:II:5}
\end{align}

Concerning $\RomanII_{6}$, we use the duality of $B_{\infty,\infty}^{-\kappa}$ and $B_{1,1}^{\kappa}$ and H$\ddot{\mathrm{o}}$lder's and Young's inequalities to estimate 
\begin{align}
\RomanII_{6} \overset{\eqref{Def:II5:II6}}{\lesssim}& \lVert X^{n} \rVert_{B_{\infty,\infty}^{-\kappa}} \lVert (Y^{n})^{2} \Lambda^{2\alpha} w^{n} \rVert_{B_{1,1}^{\kappa}}  \nonumber\\
\overset{\eqref{Def:Lt:kappa}  \eqref{Est:MW17b:a}\eqref{Est:MW17b:b}}{\lesssim}& L_{t}^{\kappa} ( \lVert Y^{n} \rVert_{L^{4}}^{2}\lVert \Lambda^{2\alpha} w^{n} \rVert_{B_{2,1}^{\kappa}} + \lVert Y^{n} \rVert_{L^{4}} \lVert Y^{n} \rVert_{B_{4,1}^{\kappa}} \lVert \Lambda^{2\alpha} w^{n} \rVert_{L^{2}} ) \nonumber \\
\overset{\eqref{Def:Lt:kappa} \eqref{Embed:Besov:element}\eqref{GN1:c}}{\lesssim}& (L_{t}^{\kappa})^{3} \lVert w^{n} \rVert_{\dot{H}^{-1+\alpha}}^{\frac{1-\alpha - 2 \kappa}{2}} \lVert w^{n} \rVert_{\dot{H}^{1+\alpha}}^{\frac{1+\alpha+2\kappa}{2}} \leq \frac{1}{32} \lVert w^{n} \rVert_{\dot{H}^{1+\alpha}}^{2} + C(L_{t}^{\kappa}) (1+ \lVert w^{n} \rVert_{\dot{H}^{-1+\alpha}}^{2}).\label{Est:II:6}
\end{align}

Finally, concerning $\RomanII_{7}$, we use the duality of $B_{\infty,\infty}^{-2\kappa}$ and $B_{1,1}^{2\kappa}$ and Young's inequality to estimate 
\begin{align}
\RomanII_{7} \overset{\eqref{Def:II7}}{\lesssim}& \lVert (X^{n})^{\diamondsuit 2} \rVert_{B_{\infty,\infty}^{-2\kappa}} \lVert w^{n} \Lambda^{2\alpha} w^{n} \rVert_{B_{1,1}^{2\kappa}}  \nonumber \\
\overset{\eqref{Def:Lt:kappa} \eqref{Est:MW17b:a}\eqref{Embed:Besov:element}}{\lesssim}& L_{t}^{\kappa} ( \lVert w^{n} \rVert_{L^{2}} \lVert   w^{n} \rVert_{H^{2\alpha + 3\kappa}} + \lVert w^{n} \rVert_{H^{3\kappa}} \lVert \Lambda^{2\alpha} w^{n} \rVert_{L^{2}}) \nonumber \\ 
\overset{\eqref{GN1:b} \eqref{GN1:c} \eqref{GN1:d}}{\lesssim}& L_{t}^{\kappa} \lVert w^{n} \rVert_{\dot{H}^{-1+\alpha}}^{1-\frac{3\kappa}{2}} \lVert w^{n} \rVert_{\dot{H}^{1+\alpha}}^{1+ \frac{3\kappa}{2}} \leq \frac{1}{32} \lVert w^{n} \rVert_{\dot{H}^{1+\alpha}}^{2} + C(L_{t}^{\kappa}) \lVert w^{n} \rVert_{\dot{H}^{-1+\alpha}}^{2}. \label{Est:II:7}
\end{align}
Applying \eqref{Est:II:1}, \eqref{Est:II:2}, \eqref{Est:II:3}, \eqref{Est:II:4}, \eqref{Est:II:5}, \eqref{Est:II:6}, and \eqref{Est:II:7} to \eqref{Prop:3.3:d} gives us 
\begin{align*}
&\partial_{t} \lVert w^{n} \rVert_{\dot{H}^{-1+\alpha}}^{2} + \lVert w^{n} \rVert_{\dot{H}^{1+\alpha}}^{2} + \frac{1}{4} \lVert \Lambda^{\alpha} \lvert w^{n} \rvert^{2} \rVert_{L^{2}}^{2}  + C_{S} \lVert w^{n} \rVert_{L^{\frac{4}{1-\alpha}}}^{4}  \\
\lesssim& C(L_{t}^{\kappa}) (1+ \lVert w^{n} \rVert_{\dot{H}^{-1+\alpha}}^{2}) (1+ \lVert \Lambda^{\alpha} w^{n} \rVert_{L^{2}}^{2} + \lVert w^{n} \rVert_{L^{4}}^{4}). 
\end{align*}
Considering \eqref{Prop:3.2:main:claim}, we see that this   completes the proof of Proposition~\ref{Prop:3.3}. 
\end{proof}

\subsection{Taking the limit}
In order to deduce strong convergence in $L^{3}(0,T; L^{3}(\mathbb{T}^{2}))$ via \eqref{Compact 1}, we need an additional estimate of $\partial_{t}w^{n}$. 

\begin{proposition}\label{Prop:3.4}
Under the hypothesis of Theorem~\ref{Thm:2.1}, for all $\omega \in \Omega \setminus \mathcal{N}$, with $\mathcal{N}$ from Proposition~\ref{Prop:3.1}, the solution $w^{n}$ to \eqref{Eq:Gal:w:eps} satisfies  
\begin{equation}\label{Est:time:derivative}
\int_{0}^{T} \lVert \partial_{t} w^{n} \rVert_{H^{\alpha -3}}^{\frac{4}{3}} dt \leq C    
\end{equation} 
with the bound independent of $n$. 
\end{proposition}

\begin{proof}[Proof of Proposition~\ref{Prop:3.4}]
From \eqref{Eq:Gal:w:eps} we have 
\begin{equation}\label{Prop:3.4:a}
\lVert \partial_{t} w^{n} \rVert_{H^{\alpha -3}} \leq \sum_{j=1}^{5} \RomanIII_{j} 
\end{equation}
where 
\begin{subequations}
\begin{align}
\RomanIII_{1} \triangleq& \lVert \Delta w^{n} \rVert_{H^{\alpha -1}}, \qquad  \qquad \qquad \qquad \qquad \qquad \RomanIII_{2} \triangleq \lVert (w^{n} )^{3} \rVert_{H^{\alpha -1}}, \qquad \label{Def:III:1:2} \\
\RomanIII_{3} \triangleq& \lVert 3w^{n}(Y^{n})^{2} + 3 (w^{n})^{2} Y^{n} + (Y^{n})^{3} \rVert_{H^{\alpha -1}}, \label{Def:III:3} \\
\RomanIII_{4} \triangleq& \lVert 3((w^{n})^{2} + 2w^{n} Y^{n} + (Y^{n})^{2}) X^{n} \rVert_{H^{\alpha -1}}, \qquad \,\,\,\,\,\,\, \RomanIII_{5} \triangleq \lVert 3w^{n} (X^{n})^{\diamondsuit 2} \rVert_{H^{\alpha -1}}. \label{Def:III:4:5}
\end{align}
\end{subequations}

For $\RomanIII_{1}$ in \eqref{Def:III:1:2}, it is immediate that 
\begin{equation}\label{Bd:III1}
\int_{0}^{T} \RomanIII_{1}^{2} dt \leq C    
\end{equation} 
thanks to Proposition~\ref{Prop:3.3}. 

For $\RomanIII_{2}$ in \eqref{Def:III:1:2}, using the Sobolev embedding of $L^{\frac{2}{2-\alpha}}(\mathbb{T}^{2}) \hookrightarrow \dot{H}^{\alpha-1} (\mathbb{T}^{2})$ and H$\ddot{\mathrm{o}}$lder's inequality we deduce 
\begin{align*}
\lVert (w^{n})^{3} \rVert_{\dot{H}^{\alpha -1}} \lesssim  \left( \int_{\mathbb{T}^{2}} \lvert w^{n} \rvert^{\frac{6}{2-\alpha}} dx \right)^{\frac{2-\alpha}{2}} 
\lesssim \lVert w^{n} \rVert_{L^{\frac{4}{1-\alpha}}}^{3} 
\end{align*}
which shows that 
\begin{equation}\label{Bd:III2}
\int_{0}^{T} \RomanIII_{2}^{\frac{4}{3}} dt \lesssim \int_{0}^{T} \lVert w^{n} \rVert_{L^{\frac{4}{1-\alpha}}}^{4} dt \leq C    
\end{equation}
thanks to Proposition~\ref{Prop:3.3}. 

For $\RomanIII_{3}$ in \eqref{Def:III:3}, we rely on the Sobolev embedding of $L^{\frac{2}{2-\alpha}}(\mathbb{T}^{2}) \hookrightarrow \dot{H}^{\alpha-1} (\mathbb{T}^{2})$ again and H$\ddot{\mathrm{o}}$lder's inequality to deduce 
\begin{align*}
\lVert 3w^{n} (Y^{n})^{2} + 3(w^{n})^{2} Y^{n} + (Y^{n})^{3} \rVert_{H^{\alpha -1}} \overset{\eqref{Def:Lt:kappa}}{\lesssim} (L_{t}^{\kappa})^{3} \left( 1+ \lVert w^{n} \rVert_{L^{\frac{4}{1-\alpha}}}^{2} \right), 
\end{align*}
which shows that 
\begin{equation}\label{Bd:III3}
\int_{0}^{T} \RomanIII_{3}^{2} dt \leq C    
\end{equation}
thanks to Proposition~\ref{Prop:3.3}. 

Concerning $\RomanIII_{4}$ in \eqref{Def:III:4:5}, we can estimate 
\begin{align*}
& \lVert \left( (w^{n})^{2} + 2 w^{n} Y^{n} + (Y^{n})^{2} \right) X^{n} \rVert_{H^{\alpha -1}} \\
\overset{\eqref{Def:Lt:kappa} \eqref{Est:MW17b:a}}{\lesssim}& L_{t}^{\kappa}  \sup_{\lVert \phi \rVert_{H^{1-\alpha}}  \leq 1} \left( \lVert (w^{n})^{2} + 2w^{n} Y^{n} + (Y^{n})^{2} \rVert_{L^{2}} \lVert \phi \rVert_{B_{2,1}^{\kappa}} + \lVert (w^{n})^{2} + 2w^{n} Y^{n} + (Y^{n})^{2} \rVert_{B_{2,1}^{\kappa}} \lVert \phi \rVert_{L^{2}} \right) \\
\overset{\eqref{Embed:Besov:element}\eqref{Def:Lt:kappa}}{\lesssim}& L_{t}^{\kappa} \left( \lVert w^{n} \rVert_{L^{4}} \lVert w^{n} \rVert_{\dot{H}^{\frac{1}{2} + 2 \kappa}} + L_{t}^{\kappa} \lVert w^{n} \rVert_{\dot{H}^{\frac{1}{2} + 2 \kappa}} + (L_{t}^{\kappa})^{2} \right), 
\end{align*}
which shows that 
\begin{equation}\label{Bd:III4}
\lVert \RomanIII_{4} \rVert_{L_{T}^{\frac{4}{3}} H^{\alpha -1}} \lesssim L_{T}^{\kappa} [ \lVert w^{n} \rVert_{L_{T}^{4}L^{4}} \lVert w^{n} \rVert_{L_{T}^{2} \dot{H}^{\frac{1}{2} + 2 \kappa}} + L_{t}^{\kappa} \lVert w^{n} \rVert_{L_{T}^{2} \dot{H}^{\frac{1}{2} + 2 \kappa}} + (L_{t}^{\kappa})^{2} ] \lesssim 1 
\end{equation}
thanks to Proposition~\ref{Prop:3.2}. 

Finally, concerning $\RomanIII_{5}$ in \eqref{Def:III:4:5}, we estimate 
\begin{align*}
& \lVert w^{n} (X^{n})^{\diamondsuit 2} \rVert_{H^{\alpha-1}} \lesssim \sup_{\lVert \phi \rVert_{H^{1-\alpha}} \leq 1} \lVert (X^{n})^{\diamondsuit 2} \rVert_{B_{\infty,\infty}^{-\kappa}} \lVert w^{n} \phi \rVert_{B_{1,1}^{\kappa}} \\
\overset{\eqref{Def:Lt:kappa} \eqref{Est:MW17b:a}\eqref{Embed:Besov:element}}{\lesssim}& \sup_{\lVert \phi \rVert_{H^{1-\alpha}} \leq 1} L_{t}^{\kappa} \left( \lVert w^{n} \rVert_{L^{2}} \lVert \phi \rVert_{H^{2\kappa}} + \lVert w^{n} \rVert_{H^{2\kappa}} \lVert \phi \rVert_{L^{2}} \right)  \lesssim  L_{t}^{\kappa} \lVert w^{n} \rVert_{H^{2\kappa}}
\end{align*}
which implies by Proposition~\ref{Prop:3.3} (or even Proposition~\ref{Prop:3.2}) that 
\begin{equation}\label{Bd:III5}
\int_{0}^{T} \RomanIII_{5}^{2} dt \leq C.    
\end{equation}   

We apply \eqref{Bd:III1}, \eqref{Bd:III2}, \eqref{Bd:III3}, \eqref{Bd:III4}, and \eqref{Bd:III5} to \eqref{Prop:3.4:a} to conclude \eqref{Est:time:derivative} 
\end{proof}

Computation analogous to \eqref{Rem:3.1:a} via Gagliardo-Nirenberg inequality and Proposition~\ref{Prop:3.3} shows 
\begin{equation}\label{W:alpha:3}
\lVert w^{n} \rVert_{L_{T}^{3}W^{\alpha,3}}^{3}   \lesssim \lVert w^{n} \rVert_{L_{T}^{\infty} \dot{H}^{-1+\alpha}} \int_{0}^{T} \lVert w^{n} \rVert_{\dot{H}^{1+\alpha}}^{2} dt \leq C.     
\end{equation}
Weak compactness and Banach-Alaoglu theorem applied to Proposition~\ref{Prop:3.3} and an application of \eqref{Compact 1} with \eqref{Est:time:derivative} and \eqref{W:alpha:3} give us, after relabeling subsequence by $w^{n}$ if necessary for simplicity, for all $\kappa > 0$ sufficiently small (specifically \eqref{kappa:small}), a limit $w$ such that   
\begin{subequations}
\begin{align}
&w^{n} \overset{w^{\ast}}{\rightharpoonup} w \text{ in } L^{\infty} (0, T; \dot{H}^{-1+\alpha}(\mathbb{T}^{2})), \label{Conv:a} \\
& w^{n} \overset{w}{\rightharpoonup} w \text{ in } L^{2}(0, T; \dot{H}^{1+\alpha}(\mathbb{T}^{2})), \label{Conv:b} \\
&w^{n} \to w \text{ in } L^{3}(0,T; W^{\alpha - \kappa, 3}(\mathbb{T}^{2})). \label{Conv:c}
\end{align}
\end{subequations}  
The weak formulation of \eqref{Eq:Gal:w:eps} for all $\psi \in C_{c}^{\infty} ((-\infty, T) \times \mathbb{T}^{2})$ is 
\begin{align}
& \int_{\mathbb{T}^{2}} w^{n}(t,x) \psi(t,x) dx - \int_{\mathbb{T}^{2}} \mathcal{L}_{n} \Psi^{\text{in}}(x) \psi(0,x) dx - \int_{0}^{t} \int_{\mathbb{T}^{2}} w^{n} (s,x) \partial_{s} \psi(s,x) dx ds  \nonumber \\
& + \int_{0}^{t} \int_{\mathbb{T}^{2}} w^{n} (s,x)\Delta^{2} \psi(s,x) dx ds = \int_{0}^{t} \int_{\mathbb{T}^{2}} \Bigg[ (w^{n})^{3} + 3w^{n} (Y^{n})^{2} + 3(w^{n})^{2} Y^{n} + (Y^{n})^{3}  \nonumber \\
& \qquad \qquad \qquad + 3 \left( (w^{n})^{2} + 2w^{n} Y^{n} + (Y^{n})^{2}\right) X^{n} + 3w^{n}  (X^{n})^{\diamondsuit 2} \Bigg] \Delta \psi(s,x) dx ds. \label{Weak:w:n}
\end{align}
We have sufficient tools now to take the limit $n\to\infty$ to conclude that \eqref{CH:weak} is satisfied $\mathbb{P}$-a.s., e.g. 
\begin{subequations}
\begin{align}
& \left\lvert \int_{0}^{t} \int_{\mathbb{T}^{2}} [(w^{n})^{3} - w^{3} ] \Delta \psi dx ds \right\rvert  \nonumber \\
=&  \left\lvert \int_{0}^{t} \int_{\mathbb{T}^{2}} [(w^{n} - w)(w^{n})^{2} + w(w^{n} - w) w^{n} + w^{2} (w^{n} - w) ] \Delta \psi dx ds \right\rvert  \overset{\eqref{Conv:c}}{\to} 0, \\
& \left\lvert \int_{0}^{t} \int_{\mathbb{T}^{2}} [ ( w^{n})^{2} X^{n} - w^{2} X ] \Delta \psi dx ds \right\rvert  \nonumber \\
=& \left\lvert \int_{0}^{t} \int_{\mathbb{T}^{2}} [ (w^{n} - w) w^{n} X^{n} + w (w^{n} - w) X^{n} + w^{2} (X^{n} - X)] \Delta \psi dx ds \right\rvert \overset{\eqref{Conv:c} \eqref{Conv:X:1}}{\to} 0 
\end{align}
\end{subequations}
as $n\to\infty$. 

\section{Proof of Theorem~\ref{Thm:2.1} (2): path-wise uniqueness}
For convenience of proof, we follow the approach of \cite{DY26} here. We assume that $w_{1}$ and $w_{2}$ are both weak solutions satisfying \eqref{CH:weak} from same initial data $\Psi^{\text{in}}$. Since $\psi \in C_{c}^{\infty} ((-\infty, T)\times \mathbb{T}^{2})$ satisfies $\psi(T) = 0$, we have for both $j \in \{1,2\}$, $\mathbb{P}$-a.s. 
\begin{align*}
&  - \int_{\mathbb{T}^{2}} \Psi^{\text{in}}(x) \psi(0,x) dx - \int_{0}^{T} \int_{\mathbb{T}^{2}} w_{j}(s,x) \partial_{s} \psi(s,x) dx ds  \nonumber \\
& + \int_{0}^{T} \int_{\mathbb{T}^{2}} w_{j}(s,x)\Delta^{2} \psi(s,x) dx ds \nonumber \\
=& \int_{0}^{T} \int_{\mathbb{T}^{2}} \Bigg[ w_{j}^{3} + 3w_{j} Y^{2} + 3w_{j}^{2} Y + Y^{3} + 3 \left( w_{j}^{2} + 2w_{j} Y + Y^{2} \right) X + 3w_{j}  X^{\diamondsuit 2} \Bigg](s,x) \Delta \psi(s,x) dx ds.  
\end{align*}
We extend 
\begin{equation}\label{Extend:wj}
w_{j} \equiv 0 \text{ on } (-\infty, 0).   
\end{equation}
We denote the Steklov average by 
\begin{equation}\label{Def:Stek:ave}
w_{j,h}(t,x) = \frac{1}{h} \int_{t-h}^{t} w_{j}(s,x) ds 
\end{equation}
and rely on Lemma~\ref{Lemma on Steklov}(3) to justify working on 
\begin{align*}
&  - \int_{\mathbb{T}^{2}} \Psi^{\text{in}}(x) \psi(0,x) dx - \int_{0}^{T} \int_{\mathbb{T}^{2}} w_{j,h}(s,x) \partial_{s} \psi(s,x) dx ds  \nonumber \\
& + \int_{0}^{T} \int_{\mathbb{T}^{2}} w_{j,h}(s,x)\Delta^{2} \psi(s,x) dx ds \nonumber \\
=& \int_{0}^{T} \int_{\mathbb{T}^{2}} \Bigg[ w_{j,h}^{3} + 3w_{j,h} Y^{2} + 3w_{j,h}^{2} Y + Y^{3} + 3 \left( w_{j,h}^{2} + 2w_{j,h} Y + Y^{2} \right) X + 3w_{j,h}  X^{\diamondsuit 2} \Bigg] \Delta \psi(s,x) dx ds.  
\end{align*}
Now we integrate by parts in time derivative, use the facts that $\psi(T) = 0$ and 
\begin{equation}\label{adv:Stek}
w_{j,h}(0,x) \overset{\eqref{Def:Stek:ave}}{=} \frac{1}{h} \int_{-h}^{0} w_{j}(s,x) ds \overset{\eqref{Extend:wj}}{=} 0    
\end{equation}
to compute 
\begin{align*}
&  - \int_{\mathbb{T}^{2}} \Psi^{\text{in}}(x) \psi(0,x) dx + \int_{0}^{T} \int_{\mathbb{T}^{2}} \partial_{s} w_{j,h} \psi(s,x) dx ds  \\
&+ \int_{0}^{T} \int_{\mathbb{T}^{2}} w_{j,h}(s,x)\Delta^{2} \psi(s,x) dx ds \nonumber \\
=& \int_{0}^{T} \int_{\mathbb{T}^{2}} \Bigg[ w_{j,h}^{3} + 3w_{j,h} Y^{2} + 3w_{j,h}^{2} Y + Y^{3} + 3 \left( w_{j,h}^{2} + 2w_{j,h} Y + Y^{2} \right) X + 3w_{j,h}  X^{\diamondsuit 2} \Bigg] \Delta \psi(s,x) dx ds.  
\end{align*}
Thus, for all $\Upsilon \in C_{0}^{\infty} (\mathbb{R}; C^{\infty} (\mathbb{T}^{2}))$ and all $t \in (0,T)$, 
\begin{align}
&  - \int_{\mathbb{T}^{2}} \Psi^{\text{in}}(x) \Upsilon(0,x) dx + \int_{0}^{t} \int_{\mathbb{T}^{2}} \partial_{s} w_{j,h} \Upsilon(s,x) dx ds \label{Unique:id} \\
& + \int_{0}^{t} \int_{\mathbb{T}^{2}} w_{j,h}(s,x)\Delta^{2} \Upsilon(s,x) dx ds \nonumber \\
=& \int_{0}^{t} \int_{\mathbb{T}^{2}} \Bigg[ w_{j,h}^{3} + 3w_{j,h} Y^{2} + 3w_{j,h}^{2} Y + Y^{3} + 3 ( w_{j,h}^{2} + 2w_{j,h} Y + Y^{2}) X + 3w_{j,h}  X^{\diamondsuit 2} \Bigg] \Delta \Upsilon(s,x) dx ds.   \nonumber 
\end{align}
We define 
\begin{equation}\label{Def:rho:h}
\rho \triangleq w_{1} - w_{2}, \qquad \text{so that} \qquad \rho_{h} \overset{\eqref{Def:Stek:ave}}{=} w_{1,h} - w_{2,h}, \qquad \rho_{h}^{n} = \mathcal{L}_{n} \rho_{h}     
\end{equation}
and subtract this identity \eqref{Unique:id} for $w_{2,h}$ from that of $w_{1,h}$ to deduce 
\begin{align*}
&  \int_{0}^{t} \int_{\mathbb{T}^{2}} \partial_{s} \rho_{h} \Upsilon(s,x) dx ds + \int_{0}^{t} \int_{\mathbb{T}^{2}} \rho_{h}(s,x)\Delta^{2} \Upsilon(s,x) dx ds \nonumber \\
=& \int_{0}^{t} \int_{\mathbb{T}^{2}} \Bigg[ \rho_{h} \sum_{l=0}^{2} w_{1,h}^{2-l} w_{2,h}^{l} + 3 \rho_{h} Y^{2}  + 3\rho_{h} (w_{1,h} + w_{2,h}) Y  \\
& \qquad \qquad  + 3 ( \rho_{h} (w_{1,h} + w_{2,h}) + 2\rho_{h} Y) X + 3\rho_{h}  X^{\diamondsuit 2} \Bigg] \Delta \Upsilon(s,x) dx ds.  
\end{align*}
Thanks to Lemma~\ref{Lemma on Steklov} we can take $\Upsilon = \mathcal{L}_{n} (-\Delta)^{-1} \rho_{h}^{n}$ to deduce 
\begin{align*}
&  \int_{0}^{t} \int_{\mathbb{T}^{2}} \partial_{s} \rho_{h} \mathcal{L}_{n} (-\Delta)^{-1} \rho_{h}^{n} dx ds - \int_{0}^{t} \int_{\mathbb{T}^{2}} \rho_{h}(s,x)\Delta \mathcal{L}_{n}  \rho_{h}^{n} dx ds \nonumber \\
=& -\int_{0}^{t} \int_{\mathbb{T}^{2}} \Bigg[ \rho_{h} \sum_{l=0}^{2} w_{1,h}^{2-l} w_{2,h}^{l} + 3 \rho_{h} Y^{2}  + 3\rho_{h} (w_{1,h} + w_{2,h}) Y  \\
& \qquad \qquad  + 3 ( \rho_{h} (w_{1,h} + w_{2,h}) + 2\rho_{h} Y) X + 3\rho_{h}  X^{\diamondsuit 2} \Bigg] \mathcal{L}_{n}  \rho_{h}^{n} dx ds.  
\end{align*}
The first term is written as 
\begin{align*}
\int_{0}^{t} \int_{\mathbb{T}^{2}} \partial_{s} \rho_{h} \mathcal{L}_{n} (-\Delta)^{-1} \rho_{h}^{n} dx ds  = \frac{1}{2} \lVert \rho_{h}^{n} (t) \rVert_{\dot{H}^{-1}}^{2} 
\end{align*}
because $\lVert \rho_{h}^{n} (0) \rVert_{\dot{H}^{-1}}^{2} = 0$ thanks to \eqref{adv:Stek}. This leads us to 
\begin{align*}
\frac{1}{2} \lVert \rho_{h}^{n} (t) \rVert_{\dot{H}^{-1}}^{2} + \int_{0}^{t} \lVert  \rho_{h}^{n} \rVert_{\dot{H}^{1}}^{2} ds  = & -\int_{0}^{t} \int_{\mathbb{T}^{2}} \mathcal{L}_{n}\Bigg[ \rho_{h} \sum_{l=0}^{2} w_{1,h}^{2-l} w_{2,h}^{l} + 3 \rho_{h} Y^{2}  + 3\rho_{h} (w_{1,h} + w_{2,h}) Y  \\
& \qquad + 3 ( \rho_{h} (w_{1,h} + w_{2,h}) + 2\rho_{h} Y) X + 3\rho_{h}  X^{\diamondsuit 2} \Bigg]   \rho_{h}^{n} dx ds.  
\end{align*}
Thus, we now split this as follows:
\begin{equation}\label{Unique:c}
\frac{1}{2} \lVert \rho_{h}^{n} (t) \rVert_{\dot{H}^{-1}}^{2} + \int_{0}^{t} \lVert  \rho_{h}^{n} \rVert_{\dot{H}^{1}}^{2} ds  = \sum_{j=1}^{6} G_{j,1} + G_{j,2} + G_{j,3} 
\end{equation}
where 
\begin{subequations}
\begin{align}
G_{1,1} \triangleq& - \int_{0}^{t} \int_{\mathbb{T}^{2}} \left( \mathcal{L}_{n} \left[ \rho_{h} \sum_{l=0}^{2} w_{1,h}^{2-l} w_{2,h}^{l} \right] - \rho_{h} \sum_{l=0}^{2} w_{1,h}^{2-l} w_{2,h}^{l}  \right) \rho_{h}^{n}dxds,  \label{Def:G11}\\
G_{1,2} \triangleq& - \int_{0}^{t}\int_{\mathbb{T}^{2}} (\rho_{h} - \rho_{h}^{n}) \sum_{l=0}^{2} w_{1,h}^{2-l} w_{2,h}^{l}\rho_{h}^{n} dx ds, \label{Def:G12}\\
G_{1,3} \triangleq& - \int_{0}^{t} \int_{\mathbb{T}^{2}} (\rho_{h}^{n})^{2} \sum_{l=0}^{2} w_{1,h}^{2-l} w_{2,h}^{l} dx ds, \label{Def:G13}
\end{align}
\end{subequations}
\begin{subequations}
\begin{align}
G_{2,1} \triangleq& -3 \int_{0}^{t} \int_{\mathbb{T}^{2}} \left( \mathcal{L}_{n} [ \rho_{h} Y^{2} ] - \rho_{h} Y^{2} \right) \rho_{h}^{n} dxds, \label{Def:G21}\\
G_{2,2} \triangleq& -3\int_{0}^{t} \int_{\mathbb{T}^{2}} (\rho_{h} - \rho_{h}^{n}) Y^{2} \rho_{h}^{n} dx ds, \label{Def:G22}\\
G_{2,3} \triangleq& -3 \int_{0}^{t} \int_{\mathbb{T}^{2}} (\rho_{h}^{n})^{2} Y^{2} dxds, \label{Def:G23}
\end{align}
\end{subequations}
\begin{subequations}
\begin{align}
G_{3,1} \triangleq& -3 \int_{0}^{t} \int_{\mathbb{T}^{2}} \left( \mathcal{L}_{n} [ \rho_{h} (w_{1,h} + w_{2.h} ) Y ] - \rho_{h} (w_{1,h} + w_{2,h}) Y \right) \rho_{h}^{n} dxds, \label{Def:G31}\\
G_{3,2} \triangleq& -3 \int_{0}^{t} \int_{\mathbb{T}^{2}} (\rho_{h} - \rho_{h}^{n})(w_{1,h} + w_{2,h}) Y \rho_{h}^{n} dxds, \label{Def:G32}\\
G_{3,3} \triangleq& -3 \int_{0}^{t} \int_{\mathbb{T}^{2}} \rho_{h}^{n} (w_{1,h} + w_{2,h}) Y \rho_{h}^{n} dx ds, \label{Def:G33}
\end{align}
\end{subequations}
\begin{subequations}
\begin{align}
G_{4,1} \triangleq& -3\int_{0}^{t}\int_{\mathbb{T}^{2}} \left( \mathcal{L}_{n} [ \rho_{h}(w_{1,h} + w_{2,h} ) X] - \rho_{h} (w_{1,h} + w_{2,h} ) X  \right) \rho_{h}^{n} dxds, \label{Def:G41} \\
G_{4,2} \triangleq& -3\int_{0}^{t}\int_{\mathbb{T}^{2}} (\rho_{h} - \rho_{h}^{n})(w_{1,h} + w_{2,h}) X \rho_{h}^{n} dxds, \label{Def:G42}\\
G_{4,3} \triangleq& -3\int_{0}^{t}\int_{\mathbb{T}^{2}} \rho_{h}^{n} (w_{1,h} + w_{2,h})X \rho_{h}^{n} dxds, \label{Def:G43}
\end{align}
\end{subequations}
\begin{subequations}
\begin{align}
G_{5,1} \triangleq& -6\int_{0}^{t}\int_{\mathbb{T}^{2}} \left( \mathcal{L}_{n} [\rho_{h} Y X] - \rho_{h} Y X \right) \rho_{h}^{n} dxds,  \label{Def:G51}\\
G_{5,2} \triangleq& -6\int_{0}^{t}\int_{\mathbb{T}^{2}} (\rho_{h} - \rho_{h}^{n}) Y X \rho_{h}^{n} dxds, \label{Def:G52}\\
G_{5,3} \triangleq& -6\int_{0}^{t}\int_{\mathbb{T}^{2}} \rho_{h}^{n} Y X \rho_{h}^{n}  dxds, \label{Def:G53}
\end{align}
\end{subequations}
\begin{subequations}
\begin{align}
G_{6,1} \triangleq&  -3\int_{0}^{t}\int_{\mathbb{T}^{2}}  \left( \mathcal{L}_{n} [ \rho_{h} X^{\diamondsuit 2} ] - \rho_{h} X^{\diamondsuit 2}  \right) \rho_{h}^{n}dxds, \label{Def:G61}\\
G_{6,2} \triangleq&  -3\int_{0}^{t}\int_{\mathbb{T}^{2}}  (\rho_{h} - \rho_{h}^{n} )X^{\diamondsuit 2} \rho_{h}^{n} dxds, \label{Def:G62}\\
G_{6,3} \triangleq&  -3\int_{0}^{t}\int_{\mathbb{T}^{2}} \rho_{h}^{n} X^{\diamondsuit 2} \rho_{h}^{n}  dxds. \label{Def:G63}
\end{align}
\end{subequations}

\begin{proposition}\label{Prop:3.5}
Under the hypothesis of Theorem~\ref{Thm:2.1}, for all $\omega \in \Omega \setminus \mathcal{N}$, with $\mathcal{N}$ from Proposition~\ref{Prop:3.1}, the following holds. For all $n \in \mathbb{N}$, $\rho_{h}^{n}$ defined by \eqref{Def:rho:h} satisfies for all $t \in [0,T]$, 
\begin{align}
& \lVert \rho_{h}^{n}(t) \rVert_{\dot{H}^{-1}}^{2} + \int_{0}^{t} \left( \lVert  \rho_{h}^{n} (s) \rVert_{ \dot{H}^{1}}^{2} + \int_{\mathbb{T}^{2}} (\rho_{h}^{n})^{2}(s) [ w_{1,h}^{2} + w_{2,h}^{2} ](s) dx \right) ds + 6 \lVert \rho_{h}^{n} Y \rVert_{L_{t,x}^{2}}^{2} \nonumber \\
& \qquad \qquad \leq C(L_{t}^{\kappa}) \sum_{j=1}^{2} \int_{0}^{t} \lVert \rho_{h}^{n} \rVert_{\dot{H}^{-1}}^{2} (1+ \lVert w_{j,h} \rVert_{\dot{H}^{1+\alpha}}^{2}) ds.\label{Unique:d}
\end{align}
\end{proposition}

\begin{proof}[Proof of Proposition~\ref{Prop:3.5}]
The idea is that all of $G_{1,1}, ..., G_{6,1}$ vanish by Lebesgue's dominated convergence theorem and $G_{1,2}, ..., G_{6,2}$ can be directly bounded to show that they vanish as $n\to\infty$ thanks to \eqref{Conv:c} for $\alpha \in (\frac{2}{3},1)$. We leave computations for $G_{k,1}$ and $G_{k,2}$ for all $k \in \{1,\hdots, 6\}$ in the Appendix~\ref{Append:B}.

Let us now give estimates of $G_{k,3}$ for $k\in \{1,\hdots, 6\}$ where $G_{1,3}$ particularly makes a careful use of its sign. First, considering the sign carefully, from \eqref{Def:G13} 
\begin{equation*}
G_{1,3} = - \int_{0}^{t} \int_{\mathbb{T}^{2}} (\rho_{h}^{n})^{2} \left( w_{1,h}^{2} + w_{1,h} w_{2,h} + w_{2,h}^{2} \right)dx ds
\end{equation*}
where 
\begin{equation*}
- \int_{0}^{t} \int_{\mathbb{T}^{2}} (\rho_{h}^{n})^{2}  w_{1,h} w_{2,h} dx ds \leq \frac{1}{2} \int_{0}^{t} \int_{\mathbb{T}^{2}} (\rho_{h}^{n})^{2} \left(  w_{1,h}^{2} + w_{2,h}^{2} \right) dx ds 
\end{equation*}
so that 
\begin{equation}\label{Bd:G13}
G_{1,3} \leq - \frac{1}{2} \int_{0}^{t} \int_{\mathbb{T}^{2}} (\rho_{h}^{n})^{2} \left( w_{1,h}^{2} + w_{2,h}^{2} \right)dx ds.
\end{equation}

Concerning $G_{2,3}$, it is apparently already non-negative. Concerning $G_{3,3}$ from \eqref{Def:G33},  by the Sobolev embedding of $H^{1+\alpha} (\mathbb{T}^{2}) \hookrightarrow L^{\infty} (\mathbb{T}^{2})$, H$\ddot{\mathrm{o}}$lder' and Young's inequalities, 
\begin{align}
\lvert G_{3,3} \rvert \overset{\eqref{Def:G33}}{\lesssim}& \int_{0}^{t} \lVert \rho_{h}^{n} \rVert_{L^{2}}^{2} \lVert Y \rVert_{L^{\infty}} \lVert w_{j,h} \rVert_{L^{\infty}} ds \overset{\eqref{GN1:a}}{\lesssim} \sum_{j=1}^{2} \lVert Y \rVert_{L_{T,x}^{\infty}} \int_{0}^{t} \lVert \rho_{h}^{n} \rVert_{\dot{H}^{-1}} \lVert \rho_{h}^{n} \rVert_{\dot{H}^{1}} \lVert w_{j,h} \rVert_{\dot{H}^{1+\alpha}} ds   \nonumber\\
& \qquad \qquad \overset{\eqref{Def:Lt:kappa}}{\leq} \frac{1}{32} \int_{0}^{t} \lVert \rho_{h}^{n} \rVert_{\dot{H}^{1}}^{2} ds + C (L_{t}^{\kappa})\sum_{j=1}^{2} \int_{0}^{t} \lVert \rho_{h}^{n}\rVert_{\dot{H}^{-1}}^{2} \lVert w_{j,h} \rVert_{\dot{H}^{1+\alpha}}^{2} ds. \label{Bd:G33}
\end{align}

For $G_{4,3}$, we first rely on the duality of $B_{\infty,\infty}^{-\kappa}$ and $B_{1,1}^{\kappa}$ to estimate 
\begin{equation}\label{Unique:a}
\lvert G_{4,3} \rvert \lesssim \sum_{j=1}^{2}\int_{0}^{t} \lVert  (\rho_{h}^{n})^{2}  w_{j,h} \rVert_{B_{1,1}^{\kappa}} \lVert X \rVert_{B_{\infty,\infty}^{-\kappa}} ds.     
\end{equation}
Then, by H$\ddot{\mathrm{o}}$lder's and Young's inequalities, along with the Sobolev embedding of $H^{1+\kappa}(\mathbb{T}^{2}) \hookrightarrow L^{\infty} (\mathbb{T}^{2})$, we continue to estimate from \eqref{Unique:a} by  
\begin{align}
\lvert G_{4,3} \rvert &\overset{\eqref{Def:G43}\eqref{Def:Lt:kappa} \eqref{Est:MW17b:a}}{\lesssim} L_{t}^{\kappa} \sum_{j=1}^{2}  \int_{0}^{t} \lVert (\rho_{h}^{n})^{2} \rVert_{L^{1}} \lVert w_{j,h} \rVert_{B_{\infty,1}^{\kappa}} + \lVert (\rho_{h}^{n})^{2} \rVert_{B_{1,1}^{\kappa}} \lVert w_{j,h} \rVert_{L^{\infty}} ds   \nonumber \\
\overset{\eqref{Est:MW17b:b}\eqref{Embed:Besov:element}}{\lesssim}& L_{t}^{\kappa} \sum_{j=1}^{2} \int_{0}^{t}  \lVert \rho_{h}^{n} \rVert_{L^{2}} \lVert \rho_{h}^{n} \rVert_{H^{2\kappa}} \lVert w_{j,h} \rVert_{H^{1+ \kappa}}  \nonumber\\
\overset{\eqref{GN1:a} \eqref{GN1:d}}{\lesssim}& L_{t}^{\kappa} \sum_{j=1}^{2} \int_{0}^{t} \lVert \rho_{h}^{n} \rVert_{\dot{H}^{-1}}^{1-\kappa} \lVert \rho_{h}^{n} \rVert_{\dot{H}^{1}}^{1+ \kappa} \lVert w_{j,h} \rVert_{H^{1+\alpha}}^{\frac{2-\alpha + \kappa}{2}} \lVert w_{j,h} \rVert_{\dot{H}^{-1+\alpha}}^{\frac{\alpha - \kappa}{2}} ds\nonumber \\
\leq& \frac{1}{32} \int_{0}^{t} \lVert \rho_{h}^{n} \rVert_{\dot{H}^{1}}^{2} + C(L_{t}^{\kappa}) \sum_{j=1}^{2}  \int_{0}^{t} \lVert \rho_{h}^{n} \rVert_{\dot{H}^{-1}}^{2} (1+ \lVert w_{j,h} \rVert_{\dot{H}^{1+ \alpha}}^{2})ds. \label{Bd:G43}
\end{align} 

For $G_{5,3}$ from \eqref{Def:G53}, we can estimate 
\begin{align*}
\lvert G_{5,3} \rvert \lesssim \int_{0}^{t} \lVert X \rVert_{B_{\infty,\infty}^{-\kappa}}  \lVert (\rho_{h}^{n})^{2} Y \rVert_{B_{1,1}^{\kappa}} ds. 
\end{align*}
This is now identical to the estimate of \eqref{Unique:a} with ``$w_{j,h}$'' replaced by $Y$ so that an identical procedure to \eqref{Bd:G43} gives us 
\begin{equation}\label{Bd:G53}
\lvert G_{5,3} \rvert \lesssim L_{t}^{\kappa} \int_{0}^{t} \lVert \rho_{h}^{n} \rVert_{\dot{H}^{-1}}^{1-\kappa} \lVert \rho_{h}^{n} \rVert_{\dot{H}^{1}}^{1+ \kappa} \lVert Y \rVert_{H^{1+ \kappa}} ds  \overset{\eqref{Def:Lt:kappa}}{\leq} \frac{1}{32} \int_{0}^{t} \lVert \rho_{h}^{n} \rVert_{\dot{H}^{1}}^{2} + C(L_{t}^{\kappa}) \int_{0}^{t} \lVert \rho_{h}^{n} \rVert_{\dot{H}^{-1}}^{2} ds. 
\end{equation}

Finally, we estimate $G_{6,3}$ from \eqref{Def:G63} by 
\begin{align}
\lvert G_{6,3} \rvert  \overset{\eqref{Def:Lt:kappa}  \eqref{Embed:Besov:element}}{\lesssim}& L_{t}^{\kappa} \int_{0}^{t} \lVert \rho_{h}^{n} \rVert_{L^{2}} \lVert \rho_{h}^{n} \rVert_{H^{3\kappa}} ds  \nonumber\\
\overset{\eqref{GN1:a}}{\lesssim}&  L_{t}^{\kappa} \int_{0}^{t} \lVert \rho_{h}^{n} \rVert_{\dot{H}^{-1}}^{1- \frac{3\kappa}{2}} \lVert \rho_{h}^{n} \rVert_{\dot{H}^{1}}^{1 + \frac{3\kappa}{2}}ds \leq \frac{1}{32} \int_{0}^{t} \lVert \rho_{h}^{n} \rVert_{\dot{H}^{1}}^{2} ds + C (L_{t}^{\kappa}) \int_{0}^{t} \lVert \rho_{h}^{n} \rVert_{\dot{H}^{-1}}^{2} ds. \label{Bd:G63}
\end{align}
Considering \eqref{Bd:G13}, \eqref{Bd:G33}, \eqref{Bd:G43}, \eqref{Bd:G53}, and \eqref{Bd:G63} into \eqref{Unique:c}, not forgetting the contribution from $G_{2,3}$ from \eqref{Def:G23}, we obtain \eqref{Unique:d}. This completes the proof of Proposition~\ref{Prop:3.5}. 
\end{proof}

Thanks to \eqref{Claim:Prop:3.3}, Gronwall's inequality applied on \eqref{Unique:d} shows $\lVert \rho_{h}^{n}(t) \rVert_{\dot{H}^{-1} } = 0$ because $\rho_{h}^{n}(0) = 0$ due to \eqref{adv:Stek}. Taking $n\to\infty$ shows 
$\lVert \rho_{h}(t) \rVert_{\dot{H}^{-1}} = 0$. Therefore, $\lVert \rho_{h}\rVert_{L^{1} ( 0,T; \dot{H}^{-1}(\mathbb{T}^{2}))} = 0$. Now Lemma~\ref{Lemma on Steklov} (3) shows that taking $h\searrow 0$ gives us $\lVert \rho\rVert_{L^{1} ( 0,T; \dot{H}^{-1}(\mathbb{T}^{2}))} = 0$. Thus, for almost every (a.e.) $t$, $\lVert \rho(t) \rVert_{\dot{H}^{-1}} = 0$, allowing us to conclude our proof of the path-wise uniqueness and hence Theorem~\ref{Thm:2.1}. 

\section{Renormalization: proof of \eqref{Renorm}}\label{Sec:5}
\subsection{Proof of \eqref{Conv:X:1}}
We define $\{\beta_{l}(m)\}_{l \in \{1,2\}, m\in\mathbb{Z}^{2}}$ to be a family of $\mathbb{C}$-valued two-sided Brownian motions such that $\hat{\xi}_{l}(k) = \partial_{t} \beta_{l}(k)$ and 
\begin{equation}\label{correlation}
\mathbb{E} [\partial_{t} \beta_{l}(t,m) \partial_{t} \beta_{l'}(s,m') ] = \delta(t-s) 1_{\{ m = -m'\}} 1_{ \{l=  l'\} }. 
\end{equation} 
We denote for $i,j \in \{1,2\}$ distinct, 
\begin{equation}\label{Notations:three}
e_{m}(x) \triangleq e^{i2\pi m \cdot x}, \qquad k_{ij} \triangleq k_{i} + k_{j}, \qquad \psi_{0}(k_{1}, k_{2}) \triangleq \sum_{c,d: \lvert c-d \rvert \leq 1} \rho_{c}(k_{1}) \rho_{d}(k_{2}).   
\end{equation}
We can directly compute 
\begin{equation}\label{Renorm:1}
\mathbb{E} [ \lvert \Delta_{m} X(t) \rvert^{2} ] =\sum_{k\in\mathbb{Z}^{2} \setminus \{0\}} \rho_{m}(k)^{2}   \left( \frac{1-e^{-2t \lvert k \rvert^{4}}}{2 \lvert k \rvert^{2}} \right) \lesssim 1.     
\end{equation}
It follows by Gaussian hypercontractivity theorem (e.g. \cite[Theorem 3.50]{J97}) that for all $p \in [2 \vee (\frac{4}{\kappa}), \infty)$, 
\begin{equation}\label{Renorm:2}
\mathbb{E} [ \lVert X(t) \rVert_{\mathscr{C}^{-\kappa}}^{p} ] \lesssim \mathbb{E} [ \lVert X(t) \rVert_{B_{p,p}^{-\frac{\kappa}{2}}}^{p}] \lesssim  \sum_{m\geq -1} 2^{-m(\frac{\kappa}{2}) p} \int_{\mathbb{T}^{2}} \mathbb{E} [ \lVert \Delta_{m} X(t) \rVert_{L_{\omega}^{2}}^{p} dx 
\overset{\eqref{Renorm:1}}{\lesssim}   1.     
\end{equation}
Now, we compute 
\begin{align}
X(t) -X(r) =& \sum_{k\in\mathbb{Z}^{2}\setminus\{0\}} \left[ \int_{0}^{r} ( e^{-(t-s) \lvert k \rvert^{4}} - e^{-(r-s) \lvert k \rvert^{4}}) ik\cdot \hat{\xi} (k,s) e_{k}(x) \right] \nonumber \\
&+ \sum_{k\in\mathbb{Z}^{2}\setminus\{0\}} \int_{r}^{t} e^{-(t-s) \lvert k \rvert^{4}} ik\cdot  \hat{\xi}(k,s) ds e_{k}(x). \label{Diff:explicit} 
\end{align}
We can bound for all $r \in [0, t]$, 
\begin{subequations}\label{Bd:simple}
\begin{align}
& \lvert e^{-(t-s) \lvert k \rvert^{4}} - e^{-(r-s) \lvert k \rvert^{4}} \rvert \lesssim e^{-(r-s) \lvert k \rvert^{4}} [ ( \lvert k \rvert^{4} \lvert t-r \rvert ) \wedge 1 ], \label{2nd chaos 13} \\
&  \int_{r}^{t} e^{-2 (t-s) \lvert k \rvert^{4}} ds  \lesssim \lvert k \rvert^{-4} [ (\lvert k \rvert^{4} \lvert t-r \rvert ) \wedge 1].\label{2nd chaos 14}
\end{align}
\end{subequations} 
Thus, we now compute for any $\delta \in [0,1]$ and $r \in [0,t]$, 
\begin{align*}
\mathbb{E} [ \lvert \Delta_{j} X(t,x) - \Delta_{j} X(r,x) \rvert^{2} ] \leq& \sum_{k\in\mathbb{Z}^{2}\setminus\{0\}} \int_{0}^{r} ( e^{(t-s) \lvert k \rvert^{4}} - e^{-(r-s) \lvert k \rvert^{4}})^{2} \rho_{j}(k)^{2} \lvert k \rvert^{2} ds   \\
& + \sum_{k\in\mathbb{Z}^{2}\setminus\{0\}} \int_{r}^{t} e^{-2(t-s) \lvert k \rvert^{4}} \rho_{j}(k)^{2} \lvert k \rvert^{2} ds  \lesssim 2^{j4 \delta } \lvert t-r \vert^{\delta}.
\end{align*}  
Then we can compute by the Gaussian hypercontractivity theorem for any $p \in (2 \vee (\frac{4}{\kappa}), \infty)$, $r \in [0,t]$, $\Theta \in [-2-\kappa, -\kappa]$, and $\delta \in [0,1]$,  
\begin{align}
\mathbb{E} [ \lVert X(t) - X(r) \rVert_{\mathscr{C}^{\Theta}}^{p} ]  \lesssim \mathbb{E} [ \lVert X(t) - X(r) \rVert_{B_{p,p}^{\Theta + \frac{\kappa}{2}}}^{p} ] \lesssim \sum_{j \geq -1} 2^{j ( \Theta + \frac{\kappa}{2} + 2 \delta)p} \lvert t-r \rvert^{\frac{p\delta}{2}},  \label{2nd chaos 15}
\end{align}
from which Kolmogorov's test (e.g. \cite[Theorem 3.3]{DZ14}) gives us $X\in C_{T}^{(\frac{1}{2} - \kappa) \delta} \mathscr{C}^{ - \kappa - 2 \delta }(\mathbb{T}^{2})$ $\mathbb{P}$-a.s. for any $\delta \in [0,1]$. 

\subsection{Proof of \eqref{Conv:X:2}}
Recall $(X^{n})^{\diamondsuit 2} = (X^{n})^{2} - \mathbb{E} [ (X^{n})^{2} ]$ from \eqref{Def:diamond}.  We compute 
\begin{equation}\label{Fourier:X:n}
\mathcal{F} (X_{t}^{n}) (k) = 1_{k\neq 0} \sum_{i_{1} =1}^{2} \int_{0}^{t} e^{-(t-s) \lvert k \rvert^{4}} i k_{i_{1}} \mathfrak{l} (n^{-1} k) \hat{\xi}^{i_{1}}(s,k) ds, 
\end{equation} 
which leads to 
\begin{equation}\label{Def:C:n}
\mathbb{E} [ \lvert X_{t}^{n}\rvert^{2}] = \mathbb{E} [X_{t}^{n} \circ X_{t}^{n}]  = C_{n}(t) \qquad  \text{where} \qquad   C_{n}(t) \triangleq \sum_{k \in \mathbb{Z}^{2}\setminus \{0\}}  \frac{1-e^{-2 t \lvert k \rvert^{4}}}{2 \lvert k \rvert^{2}} \mathfrak{l}(n^{-1} k)^{2}.    
\end{equation}
We compute the second chaos as 
\begin{equation}\label{2nd chaos 1}
\mathbb{E} [ \lvert \Delta_{m} (X_{t}^{n} \circ X_{t}^{n} - C_{n}(t)) \rvert^{2} ] = \mathbb{E} [ \lvert \Delta_{m} (X_{t}^{n} \circ X_{t}^{n}) \rvert^{2} ] - \Delta_{-1} C_{n}(t)^{2}.
\end{equation} 
We focus on the first term. We compute recalling the notation of $\psi_{0}$ from \eqref{Notations:three}, 
\begin{align} 
& \mathbb{E} [ \lvert \Delta_{m} ( X_{t}^{n} \circ X_{t}^{n})(t,x) \rvert^{2} ] \label{2nd chaos 2}\\ 
=&  \sum_{k,k', k'', k'''\in\mathbb{Z}^{2}\setminus \{0\}}  e_{k+k'}(x)  \overline{e_{k'' + k'''}(x)}\rho_{m}(k+k')\rho_{m}(k''+k''')   \psi_{0} (k, k') \psi_{0} (k'', k''')  \nonumber \\ 
& \times  \sum_{i_{1},j_{1}, i_{1}', j_{1}' = 1}^{2}  \mathfrak{l} ( n^{-1} k) \mathfrak{l} ( n^{-1} k')   \mathfrak{l} ( n^{-1} k'') \mathfrak{l} ( n^{-1} k''') k_{i_{1}} k_{j_{1}}' k_{i_{1}'}'' k_{j_{1}'}'''  \nonumber \\ 
& \times  \mathbb{E} \Bigg[  \int_{0}^{t} e^{-(t-s) \lvert k \rvert^{4}} d\beta^{i_{1}}(s,k) \int_{0}^{t} e^{-(t-s') \lvert k' \rvert^{4}} d\beta^{j_{1}}(s', k')   \nonumber \\
& \qquad \times   \int_{0}^{t} e^{-(t-s'') \lvert k'' \rvert^{4}} \overline{d\beta^{i_{1}'}(s'',k'')} \int_{0}^{t} e^{-(t-s''') \lvert k''' \rvert^{4}} \overline{d\beta^{j_{1}'}(s''', k''')} \Bigg) \Bigg]. \nonumber
\end{align}
Using \eqref{2nd chaos 8}, we can compute 
\begin{align}
&\mathbb{E} [ \lvert \Delta_{m} (X_{t}^{n} \circ X_{t}^{n} - C_{n}(t))(x) \rvert^{2} ] \label{2nd chaos 5} \\
=& \frac{1}{2} \sum_{k,k' \in \mathbb{Z}^{2} \setminus \{0\}}   \rho_{m}(k+k')^{2} \psi_{0}(k,k')^{2}    \mathfrak{l}(n^{-1} k)^{2} \mathfrak{l}(n^{-1} k')^{2} \left( \frac{1-e^{-2 t \lvert k \rvert^{4}}}{\lvert k \rvert^{2}}\right) \left( \frac{1-e^{-2 t \lvert k' \rvert^{4}}}{\lvert k' \rvert^{2}}\right). \nonumber
\end{align}
Considering the frequency restrictions allows us to deduce $m \lesssim c, \lvert k' \rvert \approx 2^{d} \approx 2^{c} \gtrsim 2^{m}$ so that 
\begin{equation}\label{Est:res}
( \psi_{0}(k,k'))^{2} \lesssim \lvert k \rvert^{\frac{\delta}{2}} \lvert k' \rvert^{\frac{\delta}{2}} 2^{-m\delta}.
\end{equation}
Thus, we can apply Lemma~\ref{Lem:A.8} to conclude for any $\delta \in (0, 1)$, 
\begin{equation}\label{Est:res:b}
\mathbb{E} [ \lvert \Delta_{m} (X_{t}^{n} \circ X_{t}^{n} - C_{n}(t)) \rvert^{2} ]   \overset{\eqref{Est:res}}{\lesssim} 2^{-m\delta} \sum_{n \in \mathbb{Z}^{2}: \lvert n \rvert \approx 2^{m}} \frac{1}{\lvert n \rvert^{2-\delta}}  \approx  1.
\end{equation}
We can compute for $p \in (2 \vee (\frac{4}{\kappa}), \infty)$, using Gaussian hypercontractivity theorem 
\begin{align}\label{Benefits:modified:eq:zq:a:alt} 
\mathbb{E} [ \lVert X_{t}^{n} \circ X_{t}^{n} - C_{n}(t) \rVert_{\mathscr{C}^{-2\kappa}}^{p} ] \lesssim \mathbb{E} [ \lVert X_{t}^{n} \circ X_{t}^{n} - C_{n}(t) \rVert_{B_{p,p}^{-\kappa }}^{p} ]  \overset{\eqref{Est:res:b}}{\lesssim} \sum_{m\geq -1}  2^{- m \kappa p} \approx  1. 
\end{align}
We omit further details.  

\subsection{Proof of \eqref{Conv:X:3}}
From \eqref{Def:C:n} and \eqref{Def:diamond}, we know $(X^{n})^{\diamondsuit 3} = (X^{n})^{3} - 3C_{n} X^{n}$. We write 
\begin{equation}\label{Renorm:3}
\mathcal{F} \left( (X_{t}^{n})^{\diamondsuit 3} \right) (k) = \sum_{ \substack{k_{1}, k_{2}, k_{3} \in \mathbb{Z}^{2} \setminus \{0\} \\ k_{1} + k_{2} + k_{3} = k}} : \hat{X}_{t}^{n}(k_{1}) \hat{X}_{t}^{n}(k_{2}) 
\hat{X}_{t}^{n}(k_{3}): 
\end{equation}
which leads to 
\begin{align}
\mathbb{E} [ \lvert \Delta_{m} (X_{t}^{n})^{\diamondsuit 3} (x) \rvert^{2} ] =& \sum_{k, k'\in\mathbb{Z}^{2}} e_{k-k'}(x) \rho_{m}(k) \rho_{m}(k')  \sum_{\substack{k_{1}, k_{2}, k_{3} \in \mathbb{Z}^{2} \setminus \{0\} \\ k_{1} + k_{2} + k_{3} = k}}\sum_{\substack{k_{1}', k_{2}', k_{3}' \in \mathbb{Z}^{2} \setminus \{0\} \\ k_{1}' + k_{2}' + k_{3}' = k'}}  \nonumber \\
& \times \mathbb{E} \left[: \hat{X}_{t}^{n}(k_{1}) \hat{X}_{t}^{n}(k_{2}) 
\hat{X}_{t}^{n}(k_{3}): : \hat{X}_{t}^{n}(k_{1}') \hat{X}_{t}^{n}(k_{2}') 
\hat{X}_{t}^{n}(k_{3}'): \right].\label{Renorm:5}
\end{align}
To compute $\mathbb{E} \left[: \hat{X}_{t}^{n}(k_{1}) \hat{X}_{t}^{n}(k_{2}) 
\hat{X}_{t}^{n}(k_{3}): \overline{: \hat{X}_{t}^{n}(k_{1}') \hat{X}_{t}^{n}(k_{2}') 
\hat{X}_{t}^{n}(k_{3}'):} \right]$, we use the fact that e.g. 
\begin{align*}
\mathbb{E} [ \hat{X}_{t}^{n}(k_{1}) \overline{ \hat{X}_{t}^{n} (k_{1}')} ] = -1_{k_{1} \neq 0} 1_{k_{1} = k_{1}'}  \mathfrak{l} (n^{-1} k_{1})^{2} \left( \frac{1- e^{-2t \lvert k_{1} \rvert^{4}}}{2 \lvert k_{1} \rvert^{2}} \right) 
\end{align*}
and rely on \eqref{Square:triple:Wick} to deduce 
\begin{equation}\label{Renorm:6}
\mathbb{E} [ \lvert \Delta_{m} (X_{s}^{n})^{\diamondsuit 3} (x) \rvert^{2} ] = - \sum_{k\in\mathbb{Z}^{2}} \rho_{m}(k)^{2} \sum_{ \substack{ k_{1},k_{2},k_{3} \in \mathbb{Z}^{2} \setminus \{0\} \\ k_{1} + k_{2} + k_{3} = k}} \prod_{j=1}^{3} \mathfrak{l}(n^{-1} k_{j})^{2} \left( \frac{1-e^{-2t \lvert k_{j} \rvert^{4}}}{\lvert k_{j} \rvert^{2}} \right). 
\end{equation}
Now we estimate 
\begin{align}\label{Renorm:7}
\mathbb{E} [ \lvert \Delta_{m} (X_{s}^{n})^{\diamondsuit 3} (x) \rvert^{2} ] \overset{\eqref{Renorm:6}}{\lesssim} \sum_{k \in\mathbb{Z}^{2}} \rho_{m}(k)^{2} \sum_{k_{2} \in \mathbb{Z}^{2} \setminus \{0\}} \frac{1}{\lvert k_{2} \rvert^{2}} \sum_{ \substack{k_{1} \in \mathbb{Z}^{2} \setminus \{0\} \\ k_{12} \neq k}} \frac{1}{\lvert k_{1} \rvert^{2} \lvert k - k_{1} - k_{2} \rvert^{2}}. 
\end{align}
In the same spirit of the proof of Lemma~\ref{Lem:A.8}, we further split the first sum as 
\begin{equation}\label{Split:IV}
\sum_{ \substack{k_{1} \in \mathbb{Z}^{2} \setminus \{0\} \\ k_{12} \neq k}} \frac{1}{\lvert k_{1} \rvert^{2} \lvert k - k_{1} - k_{2} \rvert^{2}}  = \sum_{j=1}^{3} \RomanIV_{j} 
\end{equation}
where 
\begin{subequations}
\begin{align}
\RomanIV_{1} \triangleq& \sum_{ \substack{k_{1} \in \mathbb{Z}^{2} \setminus \{0\}  \\ k_{12} \neq k \\ \lvert k_{1} \rvert \leq \frac{\lvert k - k_{2} \rvert}{2}}} \frac{1}{\lvert k_{1} \rvert^{2} \lvert k - k_{1} - k_{2} \rvert^{2}}, \qquad 
\RomanIV_{2} \triangleq \sum_{ \substack{k_{1} \in \mathbb{Z}^{2} \setminus \{0\} \\ k_{12} \neq k \\ \lvert k- k_{1} - k_{2} \rvert \leq \frac{ \lvert k - k_{2} \rvert}{2} }} \frac{1}{\lvert k_{1} \rvert^{2} \lvert k - k_{1} - k_{2} \rvert^{2}}, \label{Def:IV:2} \\
\RomanIV_{3} \triangleq& \sum_{ \substack{k_{1} \in \mathbb{Z}^{2} \setminus \{0\} \\ k_{12} \neq k \\ \lvert k_{1} \rvert > \frac{\lvert k - k_{2} \rvert}{2}, \lvert k - k_{1} - k_{2} \rvert > \frac{\lvert k - k_{2} \rvert}{2}}} \frac{1}{\lvert k_{1} \rvert^{2} \lvert k - k_{1} - k_{2} \rvert^{2}}. \label{Def:IV:3}
\end{align}
\end{subequations}
For $\RomanIV_{1}$ in \eqref{Def:IV:2}, since $\lvert k - k_{1} - k_{2} \rvert \geq  \frac{\lvert k - k_{2} \rvert}{2}$, we can estimate 
\begin{equation}\label{Bd:IV1}
\RomanIV_{1} \lesssim 1_{k \neq k_{2}} \frac{1}{\lvert k - k_{2} \rvert^{2}} \sum_{\substack{k_{1} \in \mathbb{Z}^{2} \setminus \{0\} \\ \lvert k_{1} \rvert \leq \frac{\lvert k-k_{2} \rvert}{2}}}  \frac{1}{\lvert k_{1} \rvert^{2}}  \lesssim  1_{k\neq k_{2}} \frac{\ln( \lvert k -k_{2} \rvert)}{\lvert k - k_{2} \rvert^{2}}.     
\end{equation}
For $\RomanIV_{2}$ in \eqref{Def:IV:2}, because $\lvert k_{1} \rvert \geq \frac{\lvert k - k_{2} \rvert}{2}$, we can estimate similarly  
\begin{equation}\label{Bd:IV2}
\RomanIV_{2} \overset{\eqref{Def:IV:2}}{\lesssim}  1_{k\neq k_{2}}  \frac{1}{\lvert k- k_{2} \rvert^{2}}\sum_{ \substack{k_{1} \in \mathbb{Z}^{2} \setminus \{0\} \\ k_{12} \neq k \\ \lvert k- k_{1} - k_{2} \rvert \leq \frac{ \lvert k - k_{2} \rvert}{2} }} \frac{1}{ \lvert k - k_{1} - k_{2} \rvert^{2}} \lesssim 1_{k\neq k_{2}} \frac{\ln (\lvert k-k_{2} \rvert)}{\lvert k- k_{2} \rvert^{2}}.     
\end{equation}
For $\RomanIV_{3}$ in \eqref{Def:IV:3}, $\lvert k - k_{1} - k_{2} \rvert \geq \frac{1}{4} \lvert k_{1} \rvert$ so that 
\begin{equation}\label{Bd:IV3}
\RomanIV_{3} \lesssim \sum_{ \substack{k_{1} \in \mathbb{Z}^{2} \setminus \{0\} \\ k_{12} \neq k \\ \lvert k_{1} \rvert > \frac{\lvert k - k_{2} \rvert}{2}, \lvert k - k_{1} - k_{2} \rvert > \frac{\lvert k - k_{2} \rvert}{2}}} \frac{1}{\lvert k_{1} \rvert^{4}} \lesssim  1_{k\neq k_{2}} \frac{1}{\lvert k - k_{2} \rvert^{2}}.     
\end{equation}
Applying \eqref{Bd:IV1}, \eqref{Bd:IV2}, and \eqref{Bd:IV3} to \eqref{Split:IV} and then \eqref{Renorm:7} gives us 
\begin{equation}\label{Renorm:10}
\mathbb{E} [ \lvert \Delta_{m} (X_{s}^{n})^{\diamondsuit 3} (x) \rvert^{2} ] \lesssim \sum_{k \in\mathbb{Z}^{2}} \rho_{m}(k)^{2} \sum_{k_{2} \in \mathbb{Z}^{2} \setminus \{0, k\}} \frac{1}{\lvert k_{2} \rvert^{2}} \frac{\ln (\lvert k-k_{2} \rvert)}{\lvert k- k_{2} \rvert^{2}}. 
\end{equation}
Now we split the sum in the upper bound as 
\begin{equation}\label{Split:V3}
\sum_{k_{2} \in \mathbb{Z}^{2} \setminus \{0, k\}} \frac{1}{\lvert k_{2} \rvert^{2}} \frac{\ln (\lvert k-k_{2} \rvert)}{\lvert k- k_{2} \rvert^{2}} = \sum_{j=1}^{3} \RomanV_{j}  
\end{equation}
\begin{subequations}
\begin{align}
\RomanV_{1} \triangleq& \sum_{ \substack{k_{2} \in \mathbb{Z}^{2} \setminus \{0, k\} \\ \lvert k_{2} \rvert \leq \frac{\lvert k \rvert}{2} }} \frac{1}{\lvert k_{2} \rvert^{2}} \frac{\ln (\lvert k-k_{2} \rvert)}{\lvert k- k_{2} \rvert^{2}}, \qquad 
\RomanV_{2} \triangleq \sum_{ \substack{k_{2} \in \mathbb{Z}^{2} \setminus \{0, k\} \\ \lvert k-k_{2} \rvert \leq \frac{\lvert k \rvert}{2} }} \frac{1}{\lvert k_{2} \rvert^{2}} \frac{\ln (\lvert k-k_{2} \rvert)}{\lvert k- k_{2} \rvert^{2}}, \label{Def:V:2}\\
\RomanV_{3} \triangleq& \sum_{ \substack{k_{2} \in \mathbb{Z}^{2} \setminus \{0, k\} \\ \lvert k_{2} \rvert > \frac{\lvert k \rvert}{2}, \lvert k-k_{2} \rvert > \frac{\lvert k \rvert}{2} }} \frac{1}{\lvert k_{2} \rvert^{2}} \frac{\ln (\lvert k-k_{2} \rvert)}{\lvert k- k_{2} \rvert^{2}} =   \RomanV_{3,1} + \RomanV_{3,2}, \label{Def:V:3}\\ 
\RomanV_{3,1} \triangleq& \sum_{ \substack{k_{2} \in \mathbb{Z}^{2} \setminus \{0, k\} \\ 2 \lvert k \rvert > \lvert k_{2} \rvert > \frac{\lvert k \rvert}{2}, \lvert k-k_{2} \rvert > \frac{\lvert k \rvert}{2} }} \frac{1}{\lvert k_{2} \rvert^{2}} \frac{\ln (\lvert k-k_{2} \rvert)}{\lvert k- k_{2} \rvert^{2}}, \hspace{4mm} \RomanV_{3,2} \triangleq \sum_{ \substack{k_{2} \in \mathbb{Z}^{2} \setminus \{0, k\} \\ \lvert k_{2} \rvert > 2 \lvert k \rvert, \lvert k-k_{2} \rvert > \frac{\lvert k \rvert}{2} }} \frac{1}{\lvert k_{2} \rvert^{2}} \frac{\ln (\lvert k-k_{2} \rvert)}{\lvert k- k_{2} \rvert^{2}}. \label{Def:V:31}
\end{align}
\end{subequations} 
Here, we split $\RomanV_{3}$ to $\RomanV_{3,1} + \RomanV_{3,2}$  because on $\RomanV_{3}$, we can obtain, similarly to $\RomanIV_{3}$, 
\begin{equation}\label{V3:adv}
\lvert k - k_{2} \rvert \geq \frac{1}{4} \lvert k_{2} \rvert,    
\end{equation}
but this alone is not enough to bound $\ln(\lvert k - k_{2} \rvert)$ within $\RomanV_{3}$ and thus we split further to $\RomanV_{3,1}$ on which $\lvert k_{2} \rvert \leq 2 \lvert k \rvert$ and $\RomanV_{3,2}$ on which $\lvert k_{2} \rvert > 2 \lvert k \rvert$.  

First, for $\RomanV_{1}$ in \eqref{Def:V:2}, $\frac{\lvert k \rvert}{2} \leq \lvert k-k_{2} \rvert \leq \frac{3\lvert k \rvert}{2}$ so that 
\begin{equation}\label{Bd:V1}
\RomanV_{1} \lesssim 1_{k\neq 0}\frac{ \ln(\lvert k \rvert)}{\lvert k \rvert^{2}} \sum_{\substack{k_{2} \in \mathbb{Z}^{2} \setminus \{0\} \\ \lvert k_{2} \rvert \leq \frac{\lvert k \rvert}{2}}} \frac{1}{\lvert k_{2} \rvert^{2}} \lesssim 1_{k\neq 0}\frac{ \ln(\lvert k \rvert)^{2}}{\lvert k \rvert^{2}}. 
\end{equation}
For $\RomanV_{2}$ in \eqref{Def:V:2}, we have $\lvert k - k_{2} \rvert \leq \frac{\lvert k \rvert}{2}$ and $\lvert k_{2} \rvert \geq \frac{\lvert k \rvert}{2}$ so that 
\begin{equation}\label{Bd:V2}
\RomanV_{2} \lesssim 1_{k\neq 0} \frac{ \ln(\lvert k \rvert)}{\lvert k \rvert^{2}} \sum_{ \substack{k_{2} \in \mathbb{Z}^{2} \setminus \{0, k \} \\ \lvert k - k_{2} \rvert \leq \frac{\lvert k \rvert}{2}}} \frac{1}{\lvert k - k_{2} \rvert^{2}}  \lesssim 1_{k\neq 0} \frac{ \ln( \lvert k \rvert)^{2}}{\lvert k \rvert^{2}}.
\end{equation} 
For $\RomanV_{3,1}$ in \eqref{Def:V:31}, besides \eqref{V3:adv}, we have $\lvert k - k_{2} \rvert \leq 3 \lvert k \rvert$ so that 
\begin{equation}\label{Bd:V31}
\RomanV_{3,1} \lesssim 1_{k\neq 0} \ln( \lvert k \rvert) \sum_{ \substack{k_{2} \in \mathbb{Z}^{2} \setminus \{0, k\} \\ 2 \lvert k \rvert > \lvert k_{2} \rvert > \frac{\lvert k \rvert}{2} }} \frac{1}{\lvert k_{2} \rvert^{4}} \lesssim 1_{k\neq 0} \frac{ \ln( \lvert k \rvert)}{\lvert k \rvert^{2}}. 
\end{equation}
For $\RomanV_{3,2}$ in \eqref{Def:V:31}, besides \eqref{V3:adv} we have $\lvert k-k_{2} \rvert  \leq \frac{3}{2} \lvert k_{2} \rvert $ so that 
\begin{equation}\label{Bd:V32}
\RomanV_{3,2} \lesssim \sum_{ \substack{k_{2} \in \mathbb{Z}^{2} \setminus \{0, k\} \\ \lvert k_{2} \rvert > 2 \lvert k \rvert, \lvert k-k_{2} \rvert > \frac{\lvert k \rvert}{2} }} \frac{1}{\lvert k_{2} \rvert^{2}} \frac{\ln (\lvert k_{2} \rvert)}{\lvert k_{2} \rvert^{2}} \lesssim  \frac{ \ln(e+ \lvert k \rvert)}{(\lvert k \rvert \vee 1)^{2}}.
\end{equation}
Combining \eqref{Bd:V31} and \eqref{Bd:V32} to \eqref{Def:V:3} and then applying the resulting inequality together with 
\eqref{Bd:V1} and \eqref{Bd:V2} to \eqref{Split:V3} gives us 
\begin{equation}\label{Renorm:11}
\sum_{k_{2} \in \mathbb{Z}^{2} \setminus \{0, k\}} \frac{1}{\lvert k_{2} \rvert^{2}} \frac{\ln (\lvert k-k_{2} \rvert)}{\lvert k- k_{2} \rvert^{2}} \lesssim 1_{k\neq 0} \frac{ \ln(e+  \lvert k \rvert)^{2}}{\lvert k \rvert^{2}}.
\end{equation}  
Finally, we apply \eqref{Renorm:11} to \eqref{Renorm:10} to conclude 
\begin{equation}\label{Renorm:12}
\mathbb{E} [ \lvert \Delta_{m} (X_{s}^{n})^{\diamondsuit 3} (x) \rvert^{2} ]  \lesssim   \sum_{k \in\mathbb{Z}^{2}} \rho_{m}(k)^{2} 1_{k\neq 0} \frac{ \ln( e+ \lvert k \rvert)^{2}}{\lvert k \rvert^{2}}  \approx m^{2}.
\end{equation}
Considering \eqref{Renorm:12} and applying Gaussian hypercontractivity theorem again gives us for $p \in (2 \vee (\frac{4}{\kappa}), \infty)$,
\begin{align*}
\mathbb{E} [ \lVert (X_{t}^{n})^{\diamondsuit 3} \rVert_{\mathscr{C}^{-3\kappa}}^{p} ]  \lesssim \mathbb{E} \left[ \lVert (X_{t}^{n})^{\diamondsuit 3} \rVert_{B_{p,p}^{-2\kappa}}^{p} \right] \lesssim \sum_{m\geq -1} 2^{-2 \kappa mp} \int_{\mathbb{T}^{2}} \lVert \Delta_{m} (X_{t}^{n})^{\diamondsuit 3} (x) \rVert_{L_{\omega}^{2}}^{p} dx \lesssim \approx 1.
\end{align*}
We omit further details. 

\section{Proof of Theorem~\ref{Thm:2.2}}\label{Sec:6}  
In this section we prove Theorem~\ref{Thm:2.2} following the approach of \cite[Theorem 2.1]{MW17b}. We define for $M > 0$, 
\begin{equation}\label{Def:G}
G(w) \triangleq \Delta \Bigg( w^{3} + 3wY^{2} + 3w^{2} Y + Y^{3}  + 3 \left( w^{2} + 2w Y + Y^{2} \right)X  + 3w X^{\diamondsuit 2}\Bigg)
\end{equation}
so that the fixed point map $\Upsilon$ for \eqref{Eq:w} can be defined as  
\begin{equation}\label{Def:Upsilon}
\Upsilon [w](t) \triangleq e^{-t \Delta^{2}} \Psi^{\text{in}} + \int_{0}^{t} e^{-(t-s) \Delta^{2}} G(w(s)) ds. 
\end{equation}
We recall $\mathcal{X}_{T}$ and $\bar{\mathcal{X}}_{T}$ in \eqref{Def:X:T}-\eqref{Def:bar:X:T} and define 
\begin{equation}\label{Def:X:T:M}
\mathcal{X}_{T,M} \triangleq \{ f \in \mathcal{X}_{T}: \lVert f \rVert_{\mathcal{X}_{T}} \leq M \},  \qquad \bar{\mathcal{X}}_{T,M} \triangleq \{ f \in \bar{\mathcal{X}}_{T}: \lVert f \rVert_{\bar{\mathcal{X}}_{T}} \leq M \}.
\end{equation}
The following proposition represents the crucial estimates. 
\begin{proposition}\label{Prop:6.1}
For all 
\begin{equation}\label{Def:M}
M \geq \max\{1, \lVert \Psi^{\text{in}} \rVert_{\mathscr{C}^{-\frac{4}{3} + \kappa_{0}}} \}, 
\end{equation}
$T \in (0,1]$, $w \in \mathcal{X}_{T,M}$ and $s \in [0,T]$, $\mathbb{P}$-a.s. $G(w(s))$ satisfies 
\begin{equation}\label{Prop:6.1:claim:1}
\lVert G(w(s)) \rVert_{\mathscr{C}^{-2-2\kappa}}  \lesssim M^{3} s^{- 3(\frac{\kappa+ \frac{4}{3} - \kappa_{0}}{4})}.    
\end{equation}
If 
\begin{equation}\label{Def:M:2}
M \geq \max\{1, \lVert \Psi^{\text{in}} \rVert_{\mathscr{C}^{\gamma}} \}
\end{equation}
and $w \in \bar{\mathcal{X}}_{T,M}$, then $\mathbb{P}$-a.s.
\begin{equation}\label{Prop:6.1:claim:2}
\lVert G(w(s)) \rVert_{\mathscr{C}^{-2-2\kappa}}  \lesssim M^{3}. \end{equation}
\end{proposition}

\begin{proof}[Proof of Proposition~\ref{Prop:6.1}]
We first prove \eqref{Prop:6.1:claim:1}. Because we may assume that $w \in \mathcal{X}_{T,M}$, we have for all $\lambda \in [-\frac{4}{3} + \kappa_{0}, \gamma]$, 
\begin{equation}\label{Est:w:lambda}
\lVert w(s) \rVert_{\mathscr{C}^{\lambda}} \lesssim M s^{- (\frac{\lambda+ \frac{4}{3} - \kappa_{0}}{4})}. 
\end{equation} 
Using \eqref{Est:w:lambda} we can estimate for any $\theta < -2$ and $2\kappa \leq -2 - \theta$ using \eqref{Est:w:lambda}, \eqref{Bd:Lt:kappa}, and \eqref{Est:MW17b:b}, with implicit constant depending on $L_{1}^{\kappa}$, 
\begin{subequations}
\begin{align}
& \lVert \Delta (w^{3}) (s) \rVert_{\mathscr{C}^{\theta}} \lesssim M^{3} s^{- 3(\frac{\kappa+ \frac{4}{3} - \kappa_{0}}{4})}, \qquad \lVert \Delta (wY^{2}) (s) \rVert_{\mathscr{C}^{\theta}}  \lesssim  Ms^{-(\frac{5\kappa + \frac{4}{3} - \kappa_{0}}{4})}, \\
& \lVert \Delta (w^{2} Y)(s) \rVert_{\mathscr{C}^{\theta}} \lesssim   M^{2} s^{- (\frac{\frac{\kappa}{2} + \frac{4}{3} - \kappa_{0}}{2})}, \qquad \lVert \Delta (Y^{3})(s) \rVert_{\mathscr{C}^{\theta}} \lesssim  1, \\
& \left\lVert \Delta  \left((w^{2} + 2wY^{2} + Y^{2}) X \right) (s) \right\rVert_{\mathscr{C}^{\theta}} \lesssim M^{2} s^{- (\frac{2\kappa + \frac{4}{3} - \kappa_{0}}{2})}  + M s^{- (\frac{2 \kappa + \frac{4}{3} - \kappa_{0}}{4})} + 1,  \\
&\lVert \Delta ( w X^{\diamondsuit 2})(s) \rVert_{\mathscr{C}^{\theta}} \lesssim  M s^{- (\frac{3\kappa+ \frac{4}{3} - \kappa_{0}}{4})},  
\end{align}
\end{subequations}
which results in \eqref{Prop:6.1:claim:1}, considering \eqref{Def:M}. 

The proof of \eqref{Prop:6.1:claim:2} follows from same computations in a simpler way because the hypothesis that $w \in \bar{\mathcal{X}}_{T,M}$ and the definition of $\bar{\mathcal{X}}_{T,M}$ imply that $\lVert w(s) \rVert_{\mathscr{C}^{\gamma}} \leq M$.  
\end{proof}

We first show that $\Upsilon$ maps $\mathcal{X}_{T,M}$ into itself, and the same result follows for $\bar{\mathcal{X}}_{T,M}$. We can estimate by Lemma~\ref{Lemma:new:1} 
\begin{align}
\lVert \Upsilon[w](t) \rVert_{\mathscr{C}^{-\frac{4}{3} + \kappa_{0}}}  \overset{\eqref{Def:Upsilon} \eqref{Prop:6.1:claim:1}}{\lesssim}& \lVert \Psi^{\text{in}} \rVert_{\mathscr{C}^{-\frac{4}{3} + \kappa_{0}}} + \int_{0}^{t} (t-s)^{- \frac{1}{4} [ \frac{2}{3} + 2\kappa + \kappa_{0}]}M^{3}  s^{- 3(\frac{\kappa+ \frac{4}{3} - \kappa_{0}}{4})} ds    \nonumber \\ 
& \qquad \qquad \qquad \overset{\eqref{Def:kappa0:beta:gamma} \eqref{Small:kappa}}{\lesssim}  \lVert \Psi^{\text{in}} \rVert_{\mathscr{C}^{-\frac{4}{3} + \kappa_{0}}} + M^{3} t^{\kappa}. \label{Est:1st}
\end{align}
We can also estimate by Lemma~\ref{Lemma:new:1}, \eqref{Def:Upsilon}, \eqref{Prop:6.1:claim:1}, and \eqref{Def:kappa0:beta:gamma}
\begin{equation}\label{Est:2nd}
\lVert \Upsilon[w](t) \rVert_{\mathscr{C}^{\gamma}} \lesssim t^{-\left( \frac{\gamma + \frac{4}{3} - \kappa_{0}}{4} \right)} \lVert \Psi^{\text{in}} \rVert_{\mathscr{C}^{-\frac{4}{3} + \kappa_{0}}}  + M^{3} t^{-\left( \frac{\gamma + \frac{4}{3} - \kappa_{0} -\kappa}{4} \right)}.    
\end{equation}
Third, we can write from \eqref{Def:Upsilon} 
\begin{align}
& \Upsilon[w](t) - \Upsilon[w](s) \label{Rewrite:L:infty}  \\
=& \left( e^{-(t-s) \Delta^{2}} - \Id \right) e^{-s\Delta^{2}} \Psi^{\text{in}} + \left(e^{-(t-s) \Delta^{2}} - \Id \right) \int_{0}^{s} e^{-(s-r) \Delta^{2}} G(w(r)) dr + \int_{s}^{t} e^{-(t-r) \Delta^{2}} G(w(r)) dr. \nonumber 
\end{align}
Then we can rely on Lemmas~\ref{Lemma:new:1} and \ref{Lemma:new:2} to estimate 
\begin{align}
& \lVert \Upsilon[w](t) -\Upsilon[w](s) \rVert_{L^{\infty}} \nonumber  \\
\overset{\eqref{Rewrite:L:infty} \eqref{Prop:6.1:claim:1}}{\lesssim}& (t-s)^{\frac{\beta}{4}} s^{-( \frac{\kappa + \beta + \frac{4}{3} - \kappa_{0}}{4})} \lVert \Psi^{\text{in}} \rVert_{\mathscr{C}^{-\frac{4}{3} + \kappa_{0}}} + (t-s)^{\frac{\beta}{4}} \int_{0}^{s} (s-r)^{-(\frac{2\kappa + \beta + 2}{4})} M^{3} r^{- 3(\frac{\kappa+ \frac{4}{3} - \kappa_{0}}{4})}dr  \nonumber \\
&+ \int_{s}^{t} (t-r)^{-(\frac{2+2\kappa}{4})} M^{3} (r-s)^{-\frac{2}{3} - \frac{3\kappa}{4} + \frac{\kappa_{0}}{2} + \frac{\beta}{4}} s^{-(\frac{\beta + \frac{4}{3} - \kappa_{0}}{4})} dr  \nonumber \\
\lesssim& (t-s)^{\frac{\beta}{4}} s^{-( \frac{\kappa + \beta + \frac{4}{3} - \kappa_{0}}{4})} \lVert \Psi^{\text{in}} \rVert_{\mathscr{C}^{-\frac{4}{3} + \kappa_{0}}} + M^{3} (t-s)^{\frac{\beta}{4}} s^{-( \frac{\beta + \frac{4}{3} - \kappa_{0}}{4})} \label{Est:3rd}
\end{align}
where we used the inequality of $r^{-1- \frac{3}{4} (\kappa - \kappa_{0})} \leq (r-s)^{-\frac{2}{3} - \frac{3\kappa}{4} + \frac{\kappa_{0}}{2} + \frac{\beta}{4}} s^{-(\frac{\beta + \frac{4}{3} - \kappa_{0}}{4})}$. Therefore, there exists $C^{\ast} \geq 1$ such that for all $T \leq 1, w \in \mathcal{X}_{T,M}$, $\Psi^{\text{in}} \in \mathscr{C}^{-\frac{4}{3} + \kappa_{0}}$, $M$ that satisfies \eqref{Def:M}, we have by applying \eqref{Est:1st}, \eqref{Est:2nd}, and \eqref{Est:3rd} to \eqref{Norm:X:T}
\begin{equation*}
\lVert \Upsilon [w] \rVert_{\mathcal{X}_{T}} \leq C^{\ast}  \max\{ \lVert \Psi^{\text{in}} \rVert_{\mathscr{C}^{-\frac{4}{3} + \kappa_{0}}}, M^{3} T^{\frac{\kappa}{4}} \}. 
\end{equation*}
Thus, if we choose $M = C^{\ast} \max \{1, \lVert \Psi^{\text{in}} \rVert_{\mathscr{C}^{-\frac{4}{3} + \kappa_{0}}} \}$ and $T = (C^{\ast} M^{3})^{-\frac{4}{\kappa}}$, then $\lVert \Upsilon [w]\rVert_{\mathcal{X}_{T}} \leq M$ so that $\Upsilon$ maps $\mathcal{X}_{T,M}$ into itself. Similarly, we can obtain 
\begin{equation*}
\lVert \Upsilon[w_{1}] - \Upsilon[w_{2}] \rVert_{\mathcal{X}_{T}} \leq C^{\ast} T^{\frac{\kappa}{4}} M^{3} \lVert w_{1} - w_{2} \rVert_{\mathcal{X}_{T}} 
\end{equation*}
so that $\Upsilon$ is a contraction on $\mathcal{X}_{T,M}$. Finally, if $\Psi^{\text{in}} \in \mathscr{C}^{\gamma}(\mathbb{T}^{2})$ and $M$ satisfies \eqref{Def:M:2}, then using \eqref{Prop:6.1:claim:2} and \eqref{Small:kappa}, we can show similarly to \eqref{Est:1st}, \eqref{Est:2nd}, and \eqref{Est:3rd}, for some $C^{\ast} \geq 1$, 
\begin{equation*}
\lVert \Upsilon[w]\rVert_{\bar{\mathcal{X}}_{T}} \leq C^{\ast} \max \{ \lVert \Psi^{\text{in}} \rVert_{\mathscr{C}^{\gamma}}, T^{\kappa} \}.
\end{equation*}
Thus, an identical argument shows that $\Upsilon$ maps $\bar{\mathcal{X}}_{T,M}$ into itself and it is a contraction on $\bar{\mathcal{X}}_{T,M}$. Therefore, for all $\Psi^{\text{in}}  \in \mathscr{C}^{-\frac{4}{3} + \kappa_{0}}(\mathbb{T}^{2})$, there exists $T_{1} \in (0,1]$ such that there exists a unique mild solution on $(0, T_{1}]$. Then, $\lVert w(T_{1}) \rVert_{\mathscr{C}^{\gamma}} < \infty$ so that we can repeat up to $T_{1} + T_{2}$ for some $T_{2} > 0$. Relying on the fact that $\Upsilon$ is a contraction on $\bar{\mathcal{X}}_{\tilde{T}_{2}, M}$ for some $\tilde{T}_{2} \in (0, T_{2}]$ guarantees the continuity at $T_{1}$ of the solution starting from  $t = T_{1}$, which by uniqueness gives the continuity at $T_{1}$ of the solution starting from $t=0$. By induction we can either get a solution on $[0,1]$ or $[0, T^{\ast})$ with $T^{\ast} = \sum_{k=1}^{\infty} T_{k} \leq 1$ where $T^{\ast} < 1$ occurs only if $\lim_{t\nearrow T^{\ast}} \lVert w(t) \rVert_{\mathscr{C}^{-\frac{4}{3} + \kappa_{0}}} = \infty$. This completes the proof of Theorem~\ref{Thm:2.2}. 

\appendix

\section{Preliminaries}
\subsection{Preliminaries on fractional Laplacian, Steklov average, and compact embedding}
We start with a definition of low and frequency projections that we used extensively. 
 \begin{define}\label{Def:proj} 
\rm{(\hspace{1sp}\cite[Definition 4.1]{HR24})} Let $\mathfrak{h}: \hspace{1mm} [0,\infty) \to [0,\infty)$ be a smooth function such that 
\begin{equation*}
\mathfrak{h}(r) \triangleq  
\begin{cases}
1 & \text{ if } r \geq 1, \\
0 & \text{ if } r \leq \frac{1}{2}, 
\end{cases} 
\hspace{5mm} \mathfrak{l} \triangleq 1- \mathfrak{h}.
\end{equation*} 
Then, we consider for any $\lambda > 0$
\begin{equation*}
\check{\mathfrak{h}}_{\lambda}(x) \triangleq \mathcal{F}^{-1} \left( \mathfrak{h} \left( \frac{ \lvert \cdot \rvert}{\lambda} \right) \right) (x), \hspace{5mm} \check{\mathfrak{l}}_{\lambda}(x) \triangleq \mathcal{F}^{-1} \left( \mathfrak{l} \left(\frac{ \lvert \cdot \rvert}{\lambda} \right) \right)(x) 
\end{equation*} 
and then the projections onto high and low frequencies respectively by 
\begin{equation*}
\mathcal{H}_{\lambda}: \hspace{1mm} \mathcal{S}' \to \mathcal{S}' \text{ by } \mathcal{H}_{\lambda} f \triangleq \check{\mathfrak{h}}_{\lambda} \ast f \text{ and } \mathcal{L}_{\lambda}: \hspace{1mm} \mathcal{S}' \to \mathcal{S} \text{ by } \mathcal{L}_{\lambda} f \triangleq f - \mathcal{H}_{\lambda} f = \check{\mathfrak{l}}_{\lambda} \ast f.  
\end{equation*} 
\end{define} 

The following lower bound on the fractional Laplacian has found many applications, especially in the study of SQG equations. 
\begin{lemma}\label{Lem:Max:Prin}
(\cite{CC04} and Lemma 3.3 \cite{J05}) Let $r \in [0, 1],$ $V = \mathbb{T}^{2}$ or $V = \mathbb{R}^{2}$, and $f, \Lambda^{2r}f \in L^{p}(V), p \geq 2$. Then  
\begin{equation*}
2\int_{V} \lvert\Lambda^{r}\lvert f\rvert^{\frac{p}{2}}\rvert^{2}  dx \leq p \int_{V}\lvert f\rvert^{p-2}f\Lambda^{2r}f dx.
\end{equation*}
\end{lemma}

We collect the definition and useful properties of Steklov average for convenience. 
\begin{lemma}\label{Lemma on Steklov}
Let $I \subset \mathbb{R}$ be any interval, $E$ any measurable set over $\mathbb{T}^{d}$, $1 \leq q, r \leq \infty$ and $h > 0$. Given $v \in L^{r} (I, L^{q}(E))$, we define Steklov average 
\begin{equation*}
v_{h} (t, \cdot) \triangleq \frac{1}{h} \int_{t-h}^{t} \tilde{v} (s, \cdot) ds, \qquad \text{where } \tilde{v}(t, \cdot) \triangleq 
\begin{cases}
v(t, \cdot) &\text{ if } t \in I, \\
0 & \text{ if } t \in \mathbb{R} \setminus I.
\end{cases}
\end{equation*}
Then 
\begin{enumerate}
\item (e.g. \cite[p. 9]{BD25}) $\lVert v_{h} \rVert_{L^{q}(I; L^{q}(E))} \leq \lVert v \rVert_{L^{q}(I; L^{q}(E))}$ and $(\partial_{t} v_{h})(t,x) = \frac{\tilde{v}(t,x) - \tilde{v}(t-h,x)}{h}$. 
\item (e.g. \cite[Lemma 2.4]{CDG17}) $v_{h} \in C(I, L^{q}(E)) \cap L^{\infty} (I, L^{q}(E))$. 
\item (e.g. \cite[Lemma 2.5]{CDG17}) $v_{h} \to v$ in $L^{r} (I, L^{q}(E))$ as $h\searrow 0$. 
\end{enumerate} 
\end{lemma} 

We used the following well-known compact embedding theorem. 
\begin{lemma}\label{Compact inclusions}
{\rm (\cite[Lemma 4 (i), (ii), and (iv)]{S90})} Let $X, Y$, and $E$ be Banach spaces such that $X \Subset E \subset Y$, $q \in [1, \infty], s \in (0,1)$, and $T > 0$. Then the following compact embeddings hold: 
\begin{subequations}
\begin{align}
&L^{q} (0, T; X) \cap \{ v: \partial_{t} v \in L^{1} (0, T; Y) \} \Subset L^{q} (0, T; E),  \label{Compact 1}\\
&L^{\infty}(0, T; X) \cap \{v: \partial_{t} v \in L^{r}(0, T; Y)\} \Subset C([0,T]; E) \hspace{2mm} \forall \, r \in (1, \infty]. \label{Compact 2}  
\end{align}
\end{subequations}
\end{lemma}
The following product estimate is classical, e.g. \cite[Lemma 3.6.1 (a)]{Y26}. 
\begin{lemma}
Let $f \in W^{\delta, p_{1}}(\mathbb{T}^{d}) \cap L^{q_{2}}(\mathbb{T}^{d})$ and $g \in W^{\delta, p_{2}}(\mathbb{T}^{d}) \cap L^{q_{1}}(\mathbb{T}^{d})$ where $\delta \geq 0, 1 < p_{k} < \infty, 1 < q_{k} \leq \infty, \frac{1}{p_{k}} + \frac{1}{q_{k}} = \frac{1}{p}, k \in \{ 1, 2 \}$. Then 
\begin{equation}\label{Classical product estimate}
\lVert fg\rVert_{\dot{W}^{\delta, p}} \lesssim \lVert f\rVert_{\dot{W}^{\delta, p_{1}}}\lVert g\rVert_{L^{q_{1}}} + \lVert f\rVert_{L^{q_{2}}} \lVert g\rVert_{\dot{W}^{\delta, p_{2}}}.
\end{equation} 
\end{lemma}

\subsection{Preliminaries on Besov space and Bony's paraproducts}\label{Prelim:Besov}
We briefly recall the standard definitions and properties of Besov spaces and Bony's paraproducts. We let $\chi$ and $\rho$ be smooth functions with compact support on $\mathbb{R}^{d}$ that are non-negative, and radial such that the support of $\chi$ is contained in a ball while that of $\rho$ in an annulus and 
\begin{align*}
& \chi(\xi) + \sum_{j\geq 0} \rho(2^{-j} \xi) = 1 \hspace{3mm} \forall \hspace{1mm} \xi \in \mathbb{R}^{d}, \\
&\supp (\chi) \cap  \supp (\rho(2^{-j} \cdot )) = \emptyset \hspace{1mm} \forall \hspace{1mm} j \in\mathbb{N}, \hspace{2mm} \supp (\rho(2^{-i}\cdot)) \cap \supp (\rho (2^{-j} \cdot)) = \emptyset \hspace{1mm} \text{ if } \lvert i-j \rvert > 1.
\end{align*}
We define $\rho_{j}(\cdot) \triangleq \rho(2^{-j} \cdot)$ and the Littlewood-Paley operators $\Delta_{j}$ for $j \in \mathbb{N}_{0} \cup \{-1\}$ by 
\begin{equation*} 
\Delta_{j} f \triangleq 
\begin{cases}
\mathcal{F}^{-1} (\chi) \ast f & \text{ if } j = -1, \\
\mathcal{F}^{-1} (\rho_{j}) \ast f & \text{ if } j \in \mathbb{N}_{0},
\end{cases}
\end{equation*}
and inhomogeneous Besov spaces $B_{p,q}^{s} \triangleq \{f \in \mathcal{S}': \hspace{1mm} \lVert f \rVert_{B_{p,q}^{s}} < \infty \}$ where 
\begin{equation*}
\lVert f \rVert_{B_{p,q}^{s}} \triangleq  \lVert 2^{sm} \lVert \Delta_{m} f \rVert_{L_{x}^{p}} \rVert_{l_{m}^{q}} \hspace{3mm} \forall \hspace{1mm} p, q \in [1,\infty], s \in \mathbb{R}. 
\end{equation*} 
We define the low-frequency cut-off operator $S_{i} f \triangleq \sum_{-1  \leq j \leq i-1} \Delta_{j} f$ and Bony's paraproducts and resonant  respectively as 
\begin{equation*}
f \prec g  \triangleq \sum_{i\geq -1} S_{i-1} f \Delta_{i} g  \qquad  \text{and} \qquad  f  \circ g  \triangleq \sum_{i\geq -1} \sum_{j: \lvert j\rvert \leq 1} \Delta_{i} f  \Delta_{i+j} g  
\end{equation*}
so that 
\begin{equation}\label{Bony:decomp}
f g  = f  \prec g  + f  \succ g  + f  \circ g
\end{equation}
where $f  \succ g  = g  \prec f $ (see \cite[Sections 2.6.1 and 2.8.1]{BCD11}).
The Bony's estimates admit the following useful estimates and corollary. 

\begin{lemma}[Bony's estimates]\label{Regularity of Bony's paraproducts}\hfill
Let $\alpha, \beta \in \mathbb{R}$ and $p, p_{1}, p_{2}, q \in [1,\infty]$ where $\frac{1}{p} = \frac{1}{p_{1}} + \frac{1}{p_{2}}$. Then 
\begin{subequations}
\begin{align}
& \lVert f \prec g \rVert_{B_{p,q}^{\alpha + \beta}} \lesssim \lVert f \rVert_{B_{p_{1}, q}^{\alpha}} \lVert g \rVert_{B_{p_{2},q}^{\beta}} \hspace{2mm} \forall \, f \in B_{p_{1},q}^{\alpha}, g \in B_{p_{2},q}^{\beta} \text{ if } \alpha < 0, \label{Besov:prod:a} \\
& \lVert f \prec g \rVert_{B_{p,q}^{\beta}} \lesssim \lVert f \rVert_{L^{p_{1}}} \lVert g \rVert_{B_{p_{2},q}^{\beta}} \hspace{4mm} \,  \forall  \, f \in L^{p_{1}}, g \in B_{p_{2},q}^{\beta}, \label{Besov:prod:b}  \\
& \lVert f \circ g \rVert_{B_{p,q}^{\alpha+\beta}} \lesssim \lVert f \rVert_{B_{p_{1},q}^{\alpha}} \lVert g \rVert_{B_{p_{2},q}^{\beta}} \hspace{3mm}  \forall \, f \in B_{p_{1},q}^{\alpha}, g \in B_{p_{2},q}^{\beta} \text{ if } \alpha + \beta > 0 \label{Besov:prod:c} 
\end{align}
\end{subequations}
(\hspace{1sp}\cite[Proposition A.7]{MW17b}). Special cases include the following (\cite[Proposition 3.1]{AC15}) : 
\begin{subequations}
\begin{align}
& \lVert f \prec g \rVert_{H^{\alpha + \beta}} \lesssim_{\alpha, \beta} \lVert f \rVert_{H^{\alpha}} \lVert g \rVert_{\mathscr{C}^{\beta}} \hspace{5mm}  \forall \, f \in H^{\alpha}, g \in \mathscr{C}^{\beta}  \text{ if } \alpha < 0,  \label{Sobolev products c}   \\
& \lVert f \circ g \rVert_{H^{\alpha+ \beta}} \lesssim_{\alpha, \beta} \lVert f \rVert_{H^{\alpha}} \lVert g \rVert_{\mathscr{C}^{\beta}}  \hspace{6mm} \forall \, f \in H^{\alpha}, g \in \mathscr{C}^{\beta}  \text{ if } \alpha + \beta > 0.  \label{Sobolev products e} 
\end{align} 
\end{subequations} 
\end{lemma}  

The following is an immediate corollary, and its detailed proof can be found in \cite[Corollary 1.2]{Y21}. 
\begin{corollary}\label{Bony's threshold}
A product $fg$ is well-defined for $f \in \mathscr{C}^{\alpha}, g \in \mathscr{C}^{\beta}$ if 
\begin{equation*}
\alpha + \beta > 0; 
\end{equation*}
more precisely, if $\alpha + \beta > 0$, then $\lVert f g\rVert_{\mathscr{C}^{\text{min} \{\alpha, \beta\}}} \lesssim \lVert f \rVert_{\mathscr{C}^{\alpha}} \lVert g \rVert_{\mathscr{C}^{\beta}}.$
\end{corollary} 

\begin{lemma}
\rm{(\cite[Proposition A.7 and Corollary A.8]{MW17b})} Let $\alpha > 0$, $r \in \mathbb{N}$, $p, q, p_{1}, p_{2}, p_{3}, p_{4} \in [1,\infty]$ such that $\frac{1}{p} = \frac{1}{p_{1}} + \frac{1}{p_{2}} = \frac{1}{p_{3}} + \frac{1}{p_{4}}$. Then  
\begin{subequations}
\begin{align}
& \lVert fg \rVert_{B_{p,q}^{\alpha}} \lesssim \lVert f \rVert_{L^{p_{1}}} \lVert g \rVert_{B_{p_{2},q}^{\alpha}} + \lVert f \rVert_{B_{p_{3},q}^{\alpha}} \lVert g \rVert_{L^{p_{4}}}, \label{Est:MW17b:a}\\
& \lVert f^{r+1} \rVert_{B_{p,q}^{\alpha}} \lesssim \lVert f^{r} \rVert_{L^{p_{1}}} \lVert f \rVert_{B_{p_{2},q}^{\alpha}}. \label{Est:MW17b:b}
\end{align}
\end{subequations}
\end{lemma}
Throughout the proof, we also used the fact that for any $p \in [1,\infty], s \in \mathbb{R}$, and $\epsilon > 0$, 
\begin{equation}\label{Embed:Besov:element}
B_{p,\infty}^{s+\epsilon}  \hookrightarrow B_{p,1}^{s} 
\end{equation}
(see e.g. \cite[Lemma 3.2.15]{Y26}). 

\begin{lemma}\label{Lemma:new:1}
(\rm{\cite[Proposition A.13]{MW17b}, \cite[Lemma A.7]{GIP15}, \cite[Proposition 5]{MW17a}})
Let $p \in [1,\infty], \sigma > 0$, and $(P_{t,\sigma})_{t\geq 0}$ be the semigroup generated by $- (-\Delta)^{\sigma}$. Then, for all $\alpha  \in \mathbb{R}, \delta \geq 0$, there exists $C  > 0$ uniformly for all $t > 0$, for all smooth functions $u$ that are are mean-zero, 
\begin{equation}
\lVert P_{t,\sigma} u \rVert_{B_{p,\infty}^{\alpha + \delta}} \leq C t^{-\frac{\delta}{2\sigma}} \lVert u \rVert_{B_{p,\infty}^{\alpha}}, \qquad \lVert P_{t,\sigma} u \rVert_{B_{p,\infty}^{\delta}} \leq C t^{-\frac{\delta}{2\sigma}} \lVert u \rVert_{L^{p}}. 
\end{equation}
\end{lemma}

\begin{lemma}\label{Lemma:new:2}
(\rm{\cite[Proposition A.13]{MW17b}, \cite[Lemma A.8]{GIP15}, \cite[Lemma 4 and Proposition 6]{MW17a}}) Let $p, q\in [1,\infty]$, $\sigma > 0$, and $(P_{t,\sigma})_{t\geq 0}$ be the semigroup generated by $- (-\Delta)^{\sigma}$. Let $\alpha, \beta \in \mathbb{R}$ such that $0 \leq \beta - \alpha \leq 2 \sigma$. Then there exists $C > 0$,  uniformly of $t > 0$, such that 
\begin{equation}
\lVert (P_{t,\sigma} - \Id) u \rVert_{B_{p,q}^{\alpha}} \leq C t^{\frac{\beta - \alpha}{2\sigma}} \lVert u \Vert_{B_{p,q}^{\beta}}. 
\end{equation}
\end{lemma}

\subsection{Preliminaries on Wick products and renormalization}\label{Prelim:renorm}
To handle the quadruple product of the STWN, we recall (e.g. \cite[Example 2.1]{Y21}) that if $\xi_{1}, \xi_{2}, \xi_{3}, \xi_{4}$ are any elements of Gaussian Hilbert space, then 
\begin{equation}\label{2nd chaos 8}
\mathbb{E} [ \xi_{1} \xi_{2} \xi_{3} \xi_{4}] = \mathbb{E} [\xi_{2} \xi_{3}] \mathbb{E} [\xi_{1} \xi_{4}] + \mathbb{E} [\xi_{2} \xi_{4}] \mathbb{E} [\xi_{1} \xi_{3}] + \mathbb{E} [\xi_{3} \xi_{4}] \mathbb{E} [\xi_{1} \xi_{2}].
\end{equation} 
Another useful fact (e.g. \cite[Example 2.2]{Y21}) is the following: if $\{ \xi_{ij}\}_{1 \leq i \leq k, 1 \leq j \leq l_{i}}$ are centered jointly normal variables with $k \geq 0$ and $l_{1}, \hdots, l_{k} \geq 0$ and $: \xi_{ij}...\xi_{kl}:$ denotes a Wick product, then 
\begin{align} 
& \mathbb{E}  [ : \xi_{11} \xi_{12} \xi_{13}: : \xi_{21} \xi_{22} \xi_{23}:] \label{Square:triple:Wick}\\
=& \mathbb{E} [ \xi_{11} \xi_{21}] \mathbb{E} [ \xi_{12} \xi_{22}] \mathbb{E} [ \xi_{13} \xi_{23}]+ \mathbb{E} [ \xi_{11} \xi_{21}] \mathbb{E} [ \xi_{12} \xi_{23}] \mathbb{E} [ \xi_{13} \xi_{22}] \nonumber\\
+& \mathbb{E} [ \xi_{11} \xi_{22}] \mathbb{E} [ \xi_{12} \xi_{21}] \mathbb{E} [ \xi_{13} \xi_{23}] + \mathbb{E} [ \xi_{11} \xi_{22}] \mathbb{E} [ \xi_{12} \xi_{23}] \mathbb{E} [ \xi_{13} \xi_{21}] \nonumber\\
+ &\mathbb{E} [ \xi_{11} \xi_{23}] \mathbb{E} [ \xi_{12} \xi_{21}] \mathbb{E} [ \xi_{13} \xi_{22}] + \mathbb{E} [ \xi_{11} \xi_{23}] \mathbb{E} [ \xi_{12} \xi_{22}] \mathbb{E} [ \xi_{13} \xi_{21}]. \nonumber
\end{align} 

The following result has been found to be useful for computations on renormalizations. 
\begin{lemma}\label{Lem:A.8}
\rm{(\cite[Lemma 3.10]{ZZ15})} Let $d \in \mathbb{N}$, $0 < l, m < d, l + m - d > 0$. Then 
\begin{equation*}
\sum_{\substack{ k_{1}, k_{2} \in \mathbb{Z}^{d} \setminus \{0\} \\ k_{1} + k_{2} = k}} \frac{1}{\lvert k_{1} \rvert^{l} \lvert k_{2} \rvert^{m}} \lesssim \frac{1}{\lvert k \rvert^{l+m-d}}. 
\end{equation*}
\end{lemma}

\section{Details of the proof of  Proposition~\ref{Prop:3.5}}\label{Append:B} 
Here, we give the computations for $G_{k,1}$ and $G_{k,2}$ for all $k \in \{1,\hdots, 6\}$. First, for $G_{1,1}$ from \eqref{Def:G11}, we can apply H$\ddot{\mathrm{o}}$lder's inequality and Young's inequality for convolution to deduce 
\begin{align*}
\lvert G_{1,1}\rvert \lesssim \left( \int_{0}^{t} \left\lVert \mathcal{L}_{n} \left[ \rho_{h} \sum_{l=0}^{2} w_{1,h}^{2-l} w_{2,h}^{l} \right] - \rho_{h} \sum_{l=0}^{2} w_{1,h}^{2-l} w_{2,h}^{l}  \right\rVert_{L^{\frac{4}{3}}(\mathbb{T}^{2})}^{\frac{4}{3}} ds \right)^{\frac{3}{4}} \lVert \rho_{h} \rVert_{L_{t,x}^{4}}. 
\end{align*} 
On the other hand, Young's inequality for convolution also gives 
\begin{align*}
&\left\lVert \mathcal{L}_{n} [ \rho_{h} \sum_{l=0}^{2} w_{1,h}^{2-l} w_{2,h}^{l} ] - \rho_{h} \sum_{l=0}^{2} w_{1,h}^{2-l} w_{2,h}^{l}  \right\rVert_{L^{\frac{4}{3}}(\mathbb{T}^{2})}^{\frac{4}{3}} \\
\lesssim& \left( \lVert \rho_{h} \rVert_{L^{4}(\mathbb{T}^{2})} \sum_{l=0}^{2} \lVert w_{1,h} \rVert_{L^{4}(\mathbb{T}^{2})}^{2-l} \lVert w_{2,h} \rVert_{L^{4}(\mathbb{T}^{2})}^{l} \right)^{\frac{4}{3}} \in L^{1}(0,t).  
\end{align*}
Therefore, for a.e. $s \in [0,t]$ we have $\lVert \mathcal{L}_{n} [ \rho_{h} \sum_{l=0}^{2} w_{1,h}^{2-l} w_{2,h}^{l} ] - \rho_{h} \sum_{l=0}^{2} w_{1,h}^{2-l} w_{2,h}^{l} \rVert_{L^{\frac{4}{3}}(\mathbb{T}^{2})} \to 0$ as $n\to\infty$ due to the standard property of mollifiers, implying that Lebesgue's dominated convergence theorem applies to show that $\lim_{n\to\infty} G_{1,1} = 0$. 

Next, for $G_{1,2}$ from \eqref{Def:G12}, we can estimate since $\alpha > \frac{2}{3}$, 
\begin{align*}
& \left\lvert \int_{0}^{t}\int_{\mathbb{T}^{2}} (\rho_{h} - \rho_{h}^{n})w_{1,h}^{2} \rho_{h}^{n} dx ds \right\rvert \\
\overset{\eqref{Classical product estimate}}{\lesssim}& \int_{0}^{t} \left( \lVert \rho_{h} - \rho_{h}^{n} \rVert_{W^{1-\alpha,\frac{2}{\alpha}}} \lVert w_{1,h}^{2} \rVert_{L^{\frac{2}{1-\alpha}}} + \lVert \rho_{h} - \rho_{h}^{n} \rVert_{L^{\infty}} \lVert w_{1,h}^{2} \rVert_{B_{2,2}^{1-\alpha}} \right) dt \lVert \rho_{h} \rVert_{L_{t}^{\infty} \dot{H}^{\alpha -1}} \\
\overset{\eqref{Claim:Prop:3.3}\eqref{Est:MW17b:b}}{\lesssim}&  \int_{0}^{t} \left( \lVert \rho_{h} - \rho_{h}^{n} \rVert_{ H^{2-2\alpha}} \lVert w_{1,h} \rVert_{L^{\frac{4}{1-\alpha}}}^{2} + \lVert \rho_{h} - \rho_{h}^{n} \rVert_{ W^{\frac{2}{3} + \kappa, 3} } \lVert w_{1,h} \rVert_{L^{\frac{4}{1-\alpha}}} \lVert w_{1,h} \rVert_{H^{-1+\alpha}}^{\frac{5\alpha-1}{4}} \lVert w_{1,h} \rVert_{\dot{H}^{1+\alpha}}^{\frac{5-5\alpha}{4}} \right) dt   \\
\overset{\eqref{Claim:Prop:3.3}}{\lesssim}& \left( \int_{0}^{t} \lVert \rho_{h} - \rho_{h}^{n} \rVert_{H^{2-2\alpha}}^{2} ds \right)^{\frac{1}{2}} + \left( \int_{0}^{t} \lVert \rho_{h} - \rho_{h}^{n} \rVert_{W^{\frac{2}{3} + \kappa,3}}^{3} ds \right)^{\frac{1}{3}}    \overset{\eqref{Conv:c}}{\to} 0 
\end{align*}
as $n\to\infty$. 

Next, for $G_{2,1}$ from \eqref{Def:G21}, the following estimate can be used to deduce that it vanishes as $n\to\infty$ similarly to $G_{1,1}$:
\begin{align*}
\int_{0}^{t} \lVert \mathcal{L}_{n} [ \rho_{h} Y^{2} ] - \rho_{h} Y^{2} \rVert_{L^{\frac{4}{3}}(\mathbb{T}^{2})}^{\frac{4}{3}} ds \lesssim \int_{0}^{t} \lVert \rho_{h} Y^{2} \rVert_{L^{\frac{4}{3}}(\mathbb{T}^{2})}^{\frac{4}{3}} ds \overset{\eqref{Def:Lt:kappa}}{\lesssim} C(L_{t}^{\kappa}) \int_{0}^{t} \lVert \rho_{h} \rVert_{L^{4} (\mathbb{T}^{2})}^{\frac{4}{3}} ds \lesssim 1. 
\end{align*}

Next, $G_{2,2}$ from \eqref{Def:G22} can be shown to vanish as $n\to\infty$ as follows by H$\ddot{\mathrm{o}}$lder's inequality:
\begin{equation*}
\lvert G_{2,2} \rvert \lesssim \lVert \rho_{h} - \rho_{h}^{n} \rVert_{L_{t,x}^{\frac{4}{3}}} \lVert Y \rVert_{L_{t,x}^{\infty}}^{2} \lVert \rho_{h} \rVert_{L_{t,x}^{4}} \overset{\eqref{Def:Lt:kappa} \eqref{Conv:c}}{\to} 0 
\end{equation*}
as $n\to\infty$. 

Next, for $G_{3,1}$ in \eqref{Def:G31}, the following computation can be used to show that $G_{3,1} \to 0$ as $n\to\infty$ similarly to $G_{1,1}$: by H$\ddot{\mathrm{o}}$lder's inequality, 
\begin{align*}
& \int_{0}^{t} \lVert \mathcal{L}_{n} [ \rho_{h} (w_{1,h} + w_{2,h} ) Y ] - \rho_{h} (w_{1,h} + w_{2,h} ) Y \rVert_{L^{\frac{4}{3}}(\mathbb{T}^{2})}^{\frac{4}{3}} ds \\
\overset{\eqref{Def:Lt:kappa}}{\lesssim}& C(L_{t}^{\kappa}) \left( \int_{0}^{t} \lVert \rho_{h} \rVert_{L^{4}(\mathbb{T}^{2})}^{2} ds \right)^{\frac{1}{2}} \sum_{j=1}^{2} \left( \int_{0}^{t} \lVert w_{j,h} \rVert_{L^{4}(\mathbb{T}^{2})}^{2} ds \right)^{\frac{1}{2}}\lesssim 1. 
\end{align*}

Next, $G_{3,2}$ from \eqref{Def:G32} is estimated by by H$\ddot{\mathrm{o}}$lder's inequality as 
\begin{align*}
\lvert G_{3,2} \rvert \overset{\eqref{Def:Lt:kappa}}{\lesssim} L_{t}^{\kappa}  \lVert \rho_{h} - \rho_{h}^{n} \rVert_{L_{t,x}^{3}} \sum_{j=1}^{2} \lVert w_{j,h} \rVert_{L_{t,x}^{4}}  \lVert \rho_{h} \rVert_{L_{t,x}^{4}}  \overset{\eqref{Conv:c}}{\to} 0 
\end{align*}
as $n\to\infty$. 

Next, the fact that $G_{4,1}$ from \eqref{Def:G41} vanishes as $n\to\infty$ can be deduced from the following estimates similarly to $G_{1,1}$: for $j \in \{1,2\}$, by \eqref{Est:MW17b:a}, \eqref{GN1:b}, and \eqref{GN1:d}, 
\begin{align*}
\lVert \rho_{h} w_{j,h} X \rVert_{H^{-1-\alpha}} &\lesssim L_{t}^{\kappa} \Bigg( \lVert \rho_{h} \rVert_{\dot{H}^{-1+\alpha}}^{\frac{1+\alpha}{2}} \lVert \rho_{h} \rVert_{\dot{H}^{1+\alpha}}^{\frac{1-\alpha}{2}} \lVert w_{j.h} \rVert_{\dot{H}^{-1+\alpha}}^{\frac{1+\alpha - 2 \kappa}{2}} \lVert w_{j,h} \rVert_{\dot{H}^{1+\alpha}}^{\frac{1- \alpha + 2 \kappa}{2}} \nonumber \\
& \qquad + \lVert \rho_{h} \rVert_{\dot{H}^{-1+\alpha}}^{\frac{1+\alpha-2\kappa}{2}} \lVert \rho_{h} \rVert_{\dot{H}^{1+\alpha}}^{\frac{1-\alpha+ 2 \kappa}{2}} \lVert w_{j.h} \rVert_{\dot{H}^{-1+\alpha}}^{\frac{1+\alpha}{2}} \lVert w_{j,h} \rVert_{\dot{H}^{1+\alpha}}^{\frac{1- \alpha}{2}} \Bigg),  
\end{align*}
which leads to 
\begin{align*}
& \int_{0}^{t} \lVert \mathcal{L}_{n} [ \rho_{h} (w_{1,h} + w_{2,h} ) X] -\rho_{h} (w_{1,h} + w_{2,h} ) X \rVert_{\dot{H}^{-1-\alpha}}^{2} ds \lesssim\int_{0}^{t} \lVert \rho_{h} [ w_{1,h} + w_{2,h} ] X \rVert_{\dot{H}^{-1-\alpha}}^{2} ds  \\
\lesssim& C(L_{t}^{\kappa}) \sum_{j=1}^{2} \Bigg[ \lVert \rho_{h} \rVert_{L_{t}^{\infty} H^{-1+\alpha}}^{1+\alpha} \lVert w_{j,h} \rVert_{L_{t}^{\infty} \dot{H}^{-1+\alpha}}^{1+ \alpha - 2 \kappa} \left( \int_{0}^{t} \lVert \rho_{h} \rVert_{\dot{H}^{1+\alpha}}^{2} ds \right)^{\frac{1-\alpha}{2}} \left( \int_{0}^{t} \lVert w_{j,h} \rVert_{\dot{H}^{1+\alpha}}^{\frac{2}{1+\alpha} (1- \alpha + 2 \kappa)} ds \right)^{\frac{1+\alpha}{2}} \\
&  + \lVert \rho_{h} \rVert_{L_{t}^{\infty} H^{-1+\alpha}}^{1+\alpha-2\kappa} \lVert w_{j,h} \rVert_{L_{t}^{\infty} \dot{H}^{-1+\alpha}}^{1+ \alpha} \left( \int_{0}^{t} \lVert w_{j,h} \rVert_{\dot{H}^{1+\alpha}}^{2} ds \right)^{\frac{1-\alpha+ 2 \kappa}{2}} \left( \int_{0}^{t} \lVert \rho_{h} \rVert_{\dot{H}^{1+\alpha}}^{\frac{2}{1+\alpha} (1- \alpha + 2 \kappa)} ds \right)^{\frac{1+\alpha-2\kappa}{2}}\Bigg]\lesssim 1. 
\end{align*}

Next, $G_{4,2}$ from \eqref{Def:G42} can be shown to vanish as $n\to\infty$ due to $W^{\frac{2}{3} + \kappa, 3}(\mathbb{T}^{2}) \hookrightarrow L^{\infty} (\mathbb{T}^{2})$ since 
\begin{align*}
&\lvert G_{4,2} \rvert \overset{\eqref{Def:Lt:kappa} \eqref{Est:MW17b:a}}{\lesssim} L_{t}^{\kappa} \sum_{j=1}^{2} \int_{0}^{t} \lVert w_{j,h} \rVert_{L^{2}}  \lVert \rho_{h} - \rho_{h}^{n} \rVert_{W^{\frac{2}{3} + 2 \kappa, 3}} \lVert \rho_{h}^{n} \rVert_{H^{2\kappa}} \\
& \qquad \qquad \qquad \qquad \qquad + \lVert w_{j,h} \rVert_{H^{2\kappa}} \lVert \rho_{h} - \rho_{h}^{n} \rVert_{W^{\frac{2}{3} + 2 \kappa, 3}} \lVert \rho_{h}^{n} \rVert_{L^{2}} ds  \\
\overset{\eqref{GN1:b} \eqref{GN1:d}}{\lesssim}&  L_{t}^{\kappa} \left( \int_{0}^{t} \lVert \rho_{h} - \rho_{h}^{n} \rVert_{W^{\frac{2}{3} + 2 \kappa, 3}}^{3} ds \right)^{\frac{1}{3}}  \\
& \times \sum_{j=1}^{2}\Bigg[ \lVert w_{j,h} \rVert_{L_{t}^{\infty} \dot{H}^{-1+\alpha}}^{\frac{1+\alpha}{2}} \lVert \rho_{h} \rVert_{L_{t}^{\infty} \dot{H}^{-1+\alpha}}^{\frac{1+ \alpha - 2 \kappa}{2}} \left( \int_{0}^{t} \lVert \rho_{h}^{n} \rVert_{\dot{H}^{1+\alpha}}^{2} ds \right)^{\frac{1-\alpha + 2 \kappa}{4}} \left( \int_{0}^{t} \lVert w_{j,h} \rVert_{\dot{H}^{1+\alpha}}^{\frac{ 6(1-\alpha)}{5+ 3 \alpha - 6 \kappa}} ds \right)^{\frac{5+ 3 \alpha - 6 \kappa}{12}} \\
&   +\lVert w_{j,h} \rVert_{L_{t}^{\infty} \dot{H}^{-1+\alpha}}^{\frac{1+\alpha-2\kappa}{2}} \lVert \rho_{h} \rVert_{L_{t}^{\infty} \dot{H}^{-1+\alpha}}^{\frac{1+ \alpha}{2}} \left( \int_{0}^{t} \lVert \rho_{h}^{n} \rVert_{\dot{H}^{1+\alpha}}^{2} ds \right)^{\frac{1-\alpha}{4}} \left( \int_{0}^{t} \lVert w_{j,h} \rVert_{\dot{H}^{1+\alpha}}^{\frac{ 6(1-\alpha+2\kappa)}{5+ 3 \alpha}} ds \right)^{\frac{5+ 3 \alpha}{12}} \Bigg] \to 0 
\end{align*}
as $n\to\infty$. 

Next, the fact that $G_{5,1}$ from \eqref{Def:G51} vanishes as $n\to\infty$ can be deduced similarly to $G_{1,1}$ from the following estimates: 
\begin{equation*}
\lVert \rho_{h} Y X \rVert_{H^{-1-\alpha}} \overset{\eqref{Est:MW17b:a}\eqref{Def:Lt:kappa} \eqref{GN1:d}}{\lesssim} (L_{t}^{\kappa})^{2} \lVert \rho_{h} \rVert_{\dot{H}^{-1+\alpha}}^{\frac{1+ \alpha - 2 \kappa}{2}} \lVert \rho_{h} \rVert_{\dot{H}^{1+\alpha}}^{\frac{1-\alpha + 2 \kappa}{2}} 
\end{equation*}
that leads to 
\begin{align*}
\int_{0}^{t} \lVert \mathcal{L}_{n} [ \rho_{h} Y X] - \rho_{h} YX \rVert_{\dot{H}^{-1-\alpha}}^{2} ds \lesssim C(L_{t}^{\kappa}) \lVert \rho_{h} \rVert_{L_{t}^{\infty}\dot{H}^{-1+ \alpha}}^{1+ \alpha - 2 \kappa} \left( \int_{0}^{t} \lVert \rho_{h} \rVert_{\dot{H}^{1+\alpha}}^{2} ds \right)^{\frac{1- \alpha + 2 \kappa}{2}} \lesssim 1. 
\end{align*}

Next, $G_{5,2}$ from \eqref{Def:G52} can be estimated by relying on the Sobolev embedding $W^{\frac{2}{3} + 2\kappa, 3}(\mathbb{T}^{2}) \hookrightarrow L^{\infty} (\mathbb{T}^{2})$ again as 
\begin{align*}
\lvert G_{5,2} \rvert \overset{\eqref{Def:Lt:kappa} \eqref{Est:MW17b:a}}{\lesssim}& L_{t}^{\kappa} \int_{0}^{t} \lVert \rho_{h} - \rho_{h}^{n} \rVert_{L^{\infty}} \lVert Y \rho_{h}^{n} \rVert_{B_{1,1}^{\kappa}} + \lVert \rho_{h} - \rho_{h}^{n} \rVert_{B_{\infty,1}^{\kappa}} \lVert Y \rho_{h}^{n} \rVert_{L^{1}} ds  \\
\overset{\eqref{Def:Lt:kappa}\eqref{GN1:d}}{\lesssim}&  (L_{t}^{\kappa})^{2} \lVert \rho_{h}^{n} \rVert_{L_{t}^{\infty} \dot{H}^{-1+ \alpha}}^{\frac{1+ \alpha - 2 \kappa}{2}} \left( \int_{0}^{t} \lVert \rho_{h} - \rho_{h}^{n} \rVert_{W^{\frac{2}{3} + 2 \kappa, 3}}^{3} ds \right)^{\frac{1}{3}} \left( \int_{0}^{t} \lVert \rho_{h}^{n} \rVert_{\dot{H}^{1+ \alpha}}^{\frac{3}{2} (\frac{1- \alpha + 2 \kappa}{2})} ds \right)^{\frac{2}{3}} \overset{\eqref{Conv:c}}{\to} 0 
\end{align*}
as $n\to\infty$. 

Next, the fact that $G_{6,1}$ from \eqref{Def:G61} vanishes as $n\to\infty$ similarly to $G_{1,1}$ can be deduced from the following estimate
\begin{align*}
\lVert \rho_{h} X^{\diamondsuit 2} \rVert_{H^{-1-\alpha}}
\overset{\eqref{Est:MW17b:a}}{\lesssim}& L_{t}^{\kappa} \sup_{\lVert \phi \rVert_{H^{1+ \alpha}} \leq 1} \left( \lVert \rho_{h} \rVert_{L^{2}} \lVert \phi \rVert_{B_{2,1}^{2\kappa}} + \lVert \rho_{h} \rVert_{B_{2,1}^{2\kappa}} \lVert \phi \rVert_{L^{2}} \right) \nonumber \\
& \qquad \qquad \overset{\eqref{GN1:d}}{\lesssim} L_{t}^{\kappa}\lVert \rho_{h} \rVert_{\dot{H}^{-1+\alpha}}^{\frac{1+ \alpha - 3\kappa}{2}} \lVert \rho_{h} \rVert_{\dot{H}^{1+ \alpha}}^{\frac{1- \alpha + 3 \kappa}{2}} 
\end{align*}
which leads to 
\begin{align*}
\int_{0}^{t} \lVert \mathcal{L}_{n} [ \rho_{h} X^{\diamondsuit 2} ] - \rho_{h} X^{\diamondsuit 2} \rVert_{H^{-1-\alpha}}^{2} ds \lesssim  (L_{t}^{\kappa})^{2} \lVert \rho_{h} \rVert_{L_{t}^{\infty} \dot{H}^{-1+\alpha}}^{1+ \alpha - 3 \kappa} \int_{0}^{t} \lVert \rho_{h} \rVert_{\dot{H}^{1+ \alpha}}^{1- \alpha + 3 \kappa} ds \lesssim 1. 
\end{align*}

Finally, $G_{6,2}$ from \eqref{Def:G62} can be estimated by 
\begin{align*}
\lvert G_{6,2} \rvert \overset{\eqref{Def:Lt:kappa} \eqref{Est:MW17b:a}}{\lesssim}& L_{t}^{\kappa} \int_{0}^{t} \lVert \rho_{h} - \rho_{h}^{n} \rVert_{L^{2}} \lVert \rho_{h}^{n} \rVert_{B_{2,1}^{2\kappa}} + \lVert \rho_{h} - \rho_{h}^{n} \rVert_{B_{2,1}^{2\kappa}} \lVert \rho_{h}^{n} \rVert_{L^{2}} ds  \\
\overset{\eqref{Embed:Besov:element}\eqref{GN1:d}}{\lesssim}& L_{t}^{\kappa} \lVert \rho_{h} \rVert_{L_{t}^{\infty} \dot{H}^{-1+\alpha}}^{\frac{1+ \alpha - 3 \kappa}{2}} \left( \int_{0}^{t} \lVert \rho_{h} - \rho_{h}^{n} \rVert_{\dot{W}^{3\kappa, 3}}^{3} ds \right)^{\frac{1}{3}} \left( \int_{0}^{t} \lVert \rho_{h} \rVert_{\dot{H}^{1+ \alpha}}^{\frac{3}{2} (\frac{1-\alpha + 3 \kappa}{2})} ds \right)^{\frac{2}{3}} \overset{\eqref{Conv:c}}{\to} 0 
\end{align*}
as $n\to\infty$.

\section{Proof of Theorem~\ref{Thm:SQG}}\label{Appendix:Pf:Thm:SQG} 
In this section we sketch the proof of Theorem~\ref{Thm:SQG}. Although the momentum formulation of the SQG equations solved by $\Upsilon$ that satisfies $-\nabla^{\bot} \cdot \Upsilon = \theta$ is the motivation of Theorem~\ref{Thm:SQG}, we present a simpler proof that does not rely on such a formulation; we simply estimate $\dot{H}^{-\frac{1}{2}}(\mathbb{T}^{2})$-norm of $w^{n}$, similarly to the approach of \cite{HZZ23, HLZZ23} (see \eqref{Prop:C.2:est:a}). 

We define for any $t \in [0,\infty)$ and $\kappa > 0$,  
\begin{equation}\label{Def:Lt:kappa:SQG}
L_{t}^{\kappa, n} \triangleq 1 + \lVert X^{n} \rVert_{C_{t} \mathscr{C}^{-\kappa}} + \lVert  (\mathcal{R}^{\bot} X^{n}) \diamondsuit X^{n} \rVert_{C_{t} \mathscr{C}^{-2\kappa}} + \lVert Y^{n} \rVert_{C_{t}\mathscr{C}^{1-3\kappa}}, \qquad L_{t}^{\kappa} \triangleq \sup_{n\in\mathbb{N}} L_{t}^{\kappa,n}, 
\end{equation}
where $X^{n} \triangleq \mathcal{L}_{n}X$ and $X$ solves \eqref{Eq:X:SQG}, $(\mathcal{R}^{\bot}X^{n}) \diamondsuit X^{n}  \overset{\eqref{Def:diamond:SQG}}{=} (\mathcal{R}^{\bot}X^{n})   X^{n} - \mathbb{E} [ (\mathcal{R}^{\bot}X^{n}) \diamondsuit X^{n}]$, and $Y^{n}$ solves \eqref{Eq:Y:SQG} with $X$ and $(\mathcal{R}^{\bot} X) \diamondsuit X$ replaced respectively by $X^{n}$ and $(\mathcal{R}^{\bot} X^{n}) \diamondsuit X^{n}$. The following result is analogous to Proposition~\ref{Prop:3.1}

\begin{proposition}\label{Prop:C.1}
Let $(\Omega, \mathcal{F}, \mathbb{P})$ be a probability space on which STWN $\xi$ satisfies Definition~\ref{Def:STWN}. Then there exists a null set $\mathcal{N} \subset \Omega$ such that 
\begin{equation}\label{Bd:Lt:kappa:SQG}
L_{t}^{\kappa}(\omega) < \infty \qquad \forall \, \omega \in \Omega \setminus \mathcal{N}, \forall \, t \geq 0, \forall \, \kappa > 0.    \end{equation}
\end{proposition}

We consider a smooth solution $w^{n}$ satisfying 
\begin{subequations}\label{Eq:Gal:w:SQG} 
\begin{align}
& \partial_{t} w^{n} + \divergence \Bigg(  \mathcal{R}^{\bot} w^{n} w^{n} + \mathcal{R}^{\bot} w^{n} Y^{n}  \nonumber \\
& \qquad \qquad + \mathcal{R}^{\bot} Y^{n} w^{n} + \mathcal{R}^{\bot} Y^{n} Y^{n} + \mathcal{R}^{\bot} w^{n} X^{n} + \mathcal{R}^{\bot} X^{n} w^{n} \Bigg) =  \Delta w^{n},  \\
& w^{n}(0) = \mathcal{L}_{n}\theta^{\text{in}}. 
\end{align}  
\end{subequations} 
Analogously to Proposition~\ref{Prop:3.2} we can prove the following proposition. 

\begin{proposition}\label{Prop:C.2}
Under the hypothesis of Theorem~\ref{Thm:SQG}, for all $\omega \in \Omega \setminus \mathcal{N}$, with $\mathcal{N}$ from Proposition~\ref{Prop:C.1}, the solution $w^{n}$ to \eqref{Eq:Gal:w:SQG} satisfies  
\begin{equation}\label{Prop:C.2:main:claim} 
\sup_{t\in [0,T]} \lVert w^{n} (t) \rVert_{\dot{H}^{- \frac{1}{2} }}^{2} +\int_{0}^{T} \lVert w^{n} (t) \rVert_{\dot{H}^{ \frac{1}{2} }}^{2} dt \leq C 
\end{equation}
with the bound independent of $n$. 
\end{proposition}

\begin{proof}
This proof is simpler than that of Proposition~\ref{Prop:3.2} and thus we only sketch its idea. Taking $L^{2}(\mathbb{T}^{2})$-inner products on \eqref{Eq:Gal:w:SQG} with $\Lambda^{-1} w^{n}$, we obtain 
\begin{equation}
\frac{1}{2} \partial_{t} \lVert w^{n} \rVert_{\dot{H}^{-\frac{1}{2}}}^{2} + \lVert w^{n} \rVert_{\dot{H}^{\frac{1}{2}}}^{2} =  \sum_{k=1}^{4} \RomanVI_{k} 
\end{equation}
where 
\begin{subequations}
\begin{align}
&  \RomanVI_{1} \triangleq - \int_{\mathbb{T}^{2}} (\mathcal{R}^{\bot} w^{n}) \cdot \nabla w^{n} \Lambda^{-1} w^{n} dx,  \label{Def:VI:1}\\
& \RomanVI_{2} \triangleq - \int_{\mathbb{T}^{2}} (\mathcal{R}^{\bot} w^{n}) \cdot \nabla Y^{n} \Lambda^{-1} w^{n}  + (\mathcal{R}^{\bot} Y^{n}) \cdot\nabla w^{n} \Lambda^{-1} w^{n} dx, \label{Def:VI:2}\\\
& \RomanVI_{3} \triangleq - \int_{\mathbb{T}^{2}} (\mathcal{R}^{\bot} Y^{n}) \cdot \nabla Y^{n} \Lambda^{-1} w^{n} dx, \label{Def:VI:3}\\\
& \RomanVI_{4} \triangleq - \int_{\mathbb{T}^{2}} ( \mathcal{R}^{\bot} w^{n}) \cdot \nabla X^{n}  \Lambda^{-1} w^{n} + ( \mathcal{R}^{\bot} X^{n}) \cdot \nabla w^{n} \Lambda^{-1} w^{n} dx. \label{Def:VI:4}\
\end{align}
\end{subequations}

The crux of the proof is the fact that the difficult nonlinear term $\RomanVI_{1}$ vanishes as follows: via integration by parts, 
\begin{equation}\label{Prop:C.2:est:a}
\int_{\mathbb{T}^{2}} ( \mathcal{R}^{\bot} w^{n} ) \cdot \nabla w^{n}    \Lambda^{-1} w^{n} dx = - \int_{\mathbb{T}^{2}} \nabla^{\bot} \Lambda^{-1} w^{n} \cdot \nabla \Lambda^{-1} w^{n} w^{n} dx = 0.  
\end{equation}
Besides $\RomanVI_{1}$, we can estimate  
\begin{align*}
& \lvert \RomanVI_{2} \rvert \overset{\eqref{Def:VI:2}}{\lesssim} \lVert ( \mathcal{R}^{\bot} w^{n}) Y^{n} + (\mathcal{R}^{\bot} Y^{n}) w^{n} \rVert_{\dot{H}^{\frac{1}{2}}} \lVert w^{n} \rVert_{\dot{H}^{-\frac{1}{2}}} 
\overset{\eqref{Classical product estimate} \eqref{Def:Lt:kappa:SQG}}{\leq} \frac{1}{4} \lVert w^{n} \rVert_{\dot{H}^{\frac{1}{2}}}^{2} + C(L_{t}^{\kappa}) \lVert w^{n} \rVert_{\dot{H}^{-\frac{1}{2}}}^{2},  \\
& \lvert \RomanVI_{3}\rvert \overset{\eqref{Def:VI:3}}{\lesssim} \lVert (\mathcal{R}^{\bot} Y^{n}) Y^{n} \rVert_{\dot{H}^{\frac{1}{2}}} \lVert w^{n} \rVert_{\dot{H}^{-\frac{1}{2}}}  
\overset{\eqref{Classical product estimate}}{\lesssim} C(L_{t}^{\kappa}) \lVert w^{n}\rVert_{\dot{H}^{-\frac{1}{2}}}, \\
&\lvert \RomanVI_{4} \rvert \overset{\eqref{Def:VI:4}\eqref{Def:Lt:kappa:SQG}}{\lesssim} L_{t}^{\kappa}  \lVert w^{n} \rVert_{\dot{H}^{3\kappa}} \lVert w^{n} \rVert_{\dot{H}^{2\kappa}} \leq \frac{1}{4} \lVert w^{n} \rVert_{\dot{H}^{\frac{1}{2}}}^{2} + C(L_{t}^{\kappa}) \lVert w^{n} \rVert_{\dot{H}^{-\frac{1}{2}}}^{2},
\end{align*}
so that Gronwall's inequality completes the proof of Proposition~\ref{Prop:C.2}. 
\end{proof}

\begin{remark}\label{Rem:C.1}
The structure of the nonlinearity of the SQG equations admits $\int_{\mathbb{T}^{2}} ( \mathcal{R}^{\bot} w^{n} ) \cdot \nabla w^{n}    \Lambda^{-1} w^{n} dx = 0$ as we saw in \eqref{Prop:C.2:est:a}, as well as 
\begin{equation}
\int_{\mathbb{T}^{2}} ( \mathcal{R}^{\bot} w^{n} ) \cdot \nabla w^{n} w^{n} dx = 0;
\end{equation}
however, 
\begin{equation*}
\int_{\mathbb{T}^{2}} ( \mathcal{R}^{\bot} w^{n} ) \cdot \nabla w^{n}    \Lambda^{-1+\alpha} w^{n} dx \neq 0 
\end{equation*}
for any $\alpha \in (0,1)$. Due to this obstacle, it seems difficult to boost up to higher regularity from \eqref{Prop:C.2:main:claim} and \eqref{Prop:C.2:main:claim} seems insufficient to prove path-wise uniqueness. 
\end{remark} 

Next, similarly to Proposition~\ref{Prop:3.4}, we can deduce the following bound on $\partial_{t} w^{n}$. 
\begin{proposition}\label{Prop:C.3}
Under the hypothesis of Theorem~\ref{Thm:SQG}, for all $\omega \in \Omega \setminus \mathcal{N}$, with $\mathcal{N}$ from Proposition~\ref{Prop:C.1}, the solution $w^{n}$ to \eqref{Eq:Gal:w:SQG} satisfies  
\begin{equation}\label{Est:time:derivative:SQG}
\int_{0}^{T} \lVert \partial_{t} w^{n} \rVert_{H^{-\frac{3}{2}}}^{\frac{4}{3}} dt \leq C    
\end{equation} 
with the bound independent of $n$. 
\end{proposition}
The proof is simpler than that of Proposition~\ref{Prop:3.4} and thus omitted. The reason for $-\frac{3}{2}$ in $H^{-\frac{3}{2}}$ is due to the $L_{T}^{2} \dot{H}^{\frac{1}{2}}_{x}$-bound from Proposition~\ref{Prop:C.2} that needs to handle the diffusive term $\lVert \Delta w^{n} \rVert_{L_{x}^{2}}$.  

The nonlinearity of the SQG equations requires only strong convergence in $L_{t,x}^{2}$ to pass the limit, in contrast to $L_{t,x}^{3}$ for the Cahn-Hilliard equation. We can deduce it by relying on \eqref{Compact 1} with $w^{n} \in L_{T}^{2} \dot{H}_{x}^{\frac{1}{2}}$ from \eqref{Prop:C.2:main:claim} and $\partial_{t}w^{n} \in L_{T}^{\frac{4}{3}} H_{x}^{-\frac{3}{2}}$ from \eqref{Est:time:derivative:SQG}. We omit further details. 

\section*{AI usage}
Concerning AI usage, the author used ChatGPT to clarify some of the arguments in \cite{CC18, HR24}. 

\section*{Acknowledgments}
The author expresses deep gratitude to Prof. Matthew Novack, Prof. Masato Hoshino, and Prof. In-Jee Jeong for valuable discussions concerning the effect of double Laplacian.

\end{document}